\documentclass[11pt]{amsart}
\usepackage{amsmath,empheq}
\usepackage{cases}
\usepackage{graphicx}
\usepackage{amssymb}
\usepackage{epstopdf}
\usepackage{color}
\usepackage{framed}
\usepackage{hyperref}
\usepackage{mdframed}
\usepackage{empheq}
\usepackage{bm}
\usepackage{hyperref}
\usepackage{mathabx} 
\usepackage{enumitem}
\usepackage[dvipsnames]{xcolor}
\usepackage[linesnumbered,ruled,vlined]{algorithm2e}
\usepackage{subcaption}

\newdimen\dropht \newbox\dropbox
\def\dropshadow#1#2{\setbox\dropbox=%
\vbox{\hrule\hbox{\vrule\kern6pt
        \vbox{\kern6pt#1\kern6pt}\kern6pt\vrule}\hrule}
\dropht=\ht\dropbox\advance\dropht by -#2
       \vbox{\baselineskip0pt\lineskip0pt
             \hbox{\copy\dropbox\vrule width#2 height\dropht}
             \hbox{\kern#2\vrule height#2 width \wd\dropbox}}}

\newtheorem{theorem}{Theorem}
\newtheorem{proposition}{Proposition}

\newtheorem{remark}{Remark}
\newtheorem{lemma}{Lemma}
\newtheorem{assumption}{Assumption}
\newtheorem*{example*}{Example}

\def\balpha{\boldsymbol{\alpha}}
\def\bbeta{\boldsymbol{\beta}}

\def\bmu{\boldsymbol{\mu}}
\def\bnu{\boldsymbol{\nu}}
\def\bphi{\boldsymbol{\phi}}
\def\bpsi{\boldsymbol{\psi}}
\def\btheta{\boldsymbol{\theta}}

\newcommand\cB{\mathcal B}
\newcommand\cC{\mathcal C}

\newcommand\cE{\mathcal E}
\newcommand\cF{\mathcal F}

\newcommand\cH{\mathcal H}

\newcommand\cK{\mathcal K}
\newcommand\cL{\mathcal L}
\newcommand\cM{\mathcal M}
\newcommand\cP{\mathcal P}
\newcommand\cU{\mathcal U}

\newcommand\bA{\mathbf A}

\newcommand\bW{\mathbf W}
\newcommand\bX{\mathbf X}

\newcommand\bv{\mathbf v}

\def\AA{\mathbb{A}}
\newcommand\EE{\mathbb E}

\newcommand\HH{\mathbb H}
\newcommand\PP{\mathbb P}
\newcommand\NN{\mathbb N}
\newcommand\RR{\mathbb R}

\AtBeginDocument{
  \let\div\relax
  \DeclareMathOperator{\div}{div}
}

\DeclareMathOperator*{\arginf}{arg\,inf}
\newcommand{\supp}{\mathrm{supp}}

\title{Non-standard Stochastic Control with Nonlinear Feynman-Kac Costs}
\author{Ren\'e Carmona \& Mathieu Lauri\`ere \& Pierre-Louis Lions}

\begin{document}
\maketitle

\begin{abstract}
We consider the conditional control problem introduced by P.L. Lions in his lectures at the Coll\`ege de France in November 2016. In his lectures, Lions emphasized some of the major differences with the analysis of classical stochastic optimal control problems, and in so doing, raised the question of the possible differences between the value functions resulting from optimization over the class of Markovian controls as opposed to the general family of open loop controls. The goal of the paper is to elucidate this quandary and provide elements of response to Lions' original conjecture. First, we justify the mathematical formulation of the conditional control problem by the description of practical model from evolutionary biology. Next, we relax the original formulation by the introduction of \emph{soft} as opposed to hard killing, and using a \emph{mimicking} argument, we reduce the open loop optimization problem to an optimization over a specific class of feedback controls. After proving existence of optimal feedback control functions, we prove a superposition principle allowing us to recast the original stochastic control problems as deterministic control problems for dynamical systems of probability Gibbs measures. Next, we characterize the solutions by forward-backward systems of coupled non-linear Partial Differential Equations (PDEs) very much in the spirit of the Mean Field Game (MFG) systems. From there, we identify a common optimizer, proving the conjecture of equality of the value functions. Finally we illustrate the results by convincing numerical experiments.
\end{abstract}

\vskip 36pt
\section{\textbf{Introduction}}
\label{se:introduction}
In this paper, we consider the conditional control problem introduced by P.L. Lions in his lectures at the Coll\`ege de France in November 2016. See \cite{Lions2016CDF}. As originally stated, the problem does not fit in the usual categories of stochastic control problems considered in the literature, so its solution requires new ideas, if not new technology. In his lectures, Lions emphasized some of the major differences with the analysis of classical stochastic optimal control problems, and in so doing, raised the question of the possible differences between the value functions resulting from optimization over the class of Markovian controls as opposed to the general family of open loop controls merely assumed to be adapted. The equality of the values of these optimization problems is accepted as a \emph{folk theorem} in the classical theory of stochastic control. However, optimizing an objective function whose values strongly depend upon the past history of the controlled trajectories of the system is a strong argument in favor of differences between the optimization results over these two classes of control processes. The goal of this paper is to elucidate this quandary and provide elements of response to Lions' original conjecture. 

\vskip 2pt
A standard stochastic control problem is concerned with the minimization of an expected cost when the latter is incurred by a controller, and is aggregated over a specific time horizon.
In the problem considered in this paper, the aggregation over time is done through the integral over the time horizon of \emph{conditional} expectations of instantaneous and terminal costs. Instead of accumulating expected costs incurred by controlling a single stochastic process over time, the distribution of the process over which the expectation is computed changes at each time because of the conditional nature of the expectations. Intuitively, at each time $t$, the expected running cost at that time could be interpreted as the limit of an expectation over a large particle system of the Fleming-Viot type. See for example \cite{ethier1993fleming}. Still, there does not seem to be a single particle system which could be used for different times.

This formulation of the optimization is highly unusual. For this reason, we provide a simple model of evolutionary biology to motivate and justify the mathematical formulation of the optimization problem. We argue that egalitarian resource sharing leads naturally to the formulation of a fitness criterion in terms of time aggregation of conditional expectations of instantaneous rewards. Such a special form of egalitarian cooperation has been observed in  many species (see for example the study of social spiders in \cite{spiders}).

The purpose of the present paper is to provide a thorough analysis of an instance of this new form of conditional control with complete proofs. Our approach is to work with a relaxed version of the problem in which we replace the hard conditioning of Lions original proposal, by a soft conditioning. Still, the main thrust of the paper is to highlight and take advantage of the role played by the distribution of the state in the evaluation of the costs. In both cases (feedback Markovian and open loop controls), we reformulate the problem as a standard control problem in infinite dimensions, the controlled dynamics being given by the time evolution of a flow of probability measures obtained by distorting and renormalizing the original distributions of the state, pretty much in the same way Gibbs measures are introduced in statistical physics. See for example \cite{DelMoral}. 

\vskip 4pt
Given the introduction of the problem in \cite{Lions2016CDF} and the work \cite{achdou2021optimal} on the large time asymptotics of the Markovian case, the contributions of the present paper are summarized in the following list: 
1) justification of the mathematical formulation of the conditional control problem by the description of practical models from evolutionary biology (see Section~\ref{sub:motivation}); 
2) relaxation of the original exit problem formulation into the analysis of smoother (distorted) Feynman-Kac semigroups (see Section~\ref{se:model-feynmankac}); 
3) reformulation in the case of Markovian feedback controls, of the original stochastic control problems as a deterministic control problems over a space of probability measures, the time evolution of the states being given by controlled dynamical systems of Gibbs measures;
4) proof of a non-local superposition principle in Subsection \ref{sub:superposition}, guaranteeing that the deterministic formulation is equivalent to the original optimization problem;
5) existence of optimal feedback control functions, derivation of a form of maximum principle; 6) characterization of optimality by a forward-backward system of non-linear, non-local partial differential equations (PDEs) very much in the spirit of the Mean Field Game (MFG) systems~\cite{MR2295621}, and analysis of the system  (see Sections~\ref{sec:formulation-markovian-deterministic} and~\ref{sec:formulation-open-deterministic}); 
7) introduction of mimicking arguments to reduce the open loop optimization problem to an optimization over a class of feedback controls (see Theorem~\ref{thm:openloop-to-newfeedback}); 
8) characterization of the solutions by similar forward-backward systems of coupled non-linear Partial Differential Equations  (see Sections~\ref{sec:fbpde-markovian} and~\ref{sec:fbpde-open}); 
9) identification of the PDE systems and proof of the value equality conjecture (see Theorem~\ref{th:equality}); and finally, 
9) convincing numerical experiments illustrating the validity of our result (see Section~\ref{sec:numerics}).
We managed to prove the results we were after without proving existence and uniqueness of solutions of the forward-backward system of non-linear, non-local PDEs characterizing optimality. So for the sake of completeness, we provide a complete proof of existence and uniqueness of classical solutions for small data, namely short time horizon and small terminal condition.

\vskip 6pt
The rest of the paper is organized as follows. In Section~\ref{sec:conditional_exit}, we introduce the problem of optimal control with conditional exit, we provide a motivation, and we propose an approximate problem with nonlinear Feynman-Kac semigroups, on which we focus in the sequel. In Section~\ref{sec:V-closedloop}, we present a detailed analysis of the  problem with feedback Markovian controls. Among other things, we prove existence of optimal controls and a non-local superposition principle of independent interest, see Theorem \ref{th:superposition} in Subsection \ref{sub:superposition}, and a form of maximum principle tailored to the present non-local dynamical equations. There, we also provide a viscosity analysis of the forward-backward PDE system characterizing the optimum. In Section~\ref{sec:open-loop}, we analyze the problem with open-loop controls: we show the equivalence with feedback Markovian controls for an extended state leading to our proof of the existence of a common optimal control, and consequently of the equality of the value functions. We conclude with Section~\ref{sec:numerics} reporting on numerical experiments corroborating our theoretical result. As mentioned above the appendix contains a detailed proof of the well posedness (i.e. existence and uniqueness of a classical solution) for the fundamental forward-backward PDE system in the case of small data (i.e. time horizon and terminal condition).

\vskip 12pt\noindent
\emph{Acknowledgments:} The first two named authors benefited from the support of NSF grant DMS-1716673, 
ARO grant W911NF-17-1-0578, and AFOSR awards FA9550-19-1-0291 and FA9550-23-1-0324. We would like to thank Dan Lacker for pointing out to us the relevance of the superposition principle and Samuel Daudin for providing us with the argument reproduced in Remark \ref{re:Daudin} on the equality of the infima.

\vskip 36pt
\section{\textbf{The Conditional Exit Control Problem}}
\label{sec:conditional_exit}
For the sake of definiteness, we review the conditional control problem originally introduced by P.L. Lions.
Let  $D$ be a bounded open domain in $\RR^d$ with a smooth boundary $\partial D$. Let us 
denote by $C([0,\infty);\RR^d)$ the space of continuous functions of time $t\in[0,\infty)$ with values in $\RR^d$, and by $C_0([0,\infty);\RR^d)$ the subspace of those functions $x\in C([0,\infty);\RR^d)$ satisfying $x(0)=0$. If $x\in C([0,\infty);\RR^d)$, we denote by $\tau^x_D$ the first exit time of the path $x$ from $D$, namely the quantity:
\begin{equation}
    \label{eq:def-tauD}
    \tau^x_D=\inf\{t\ge 0;\,x(t)\notin D\}
\end{equation}
with the convention that $\inf \emptyset = \infty$. We shall skip the superscript $x$ and/or the subscript $D$ when their values are clear from the context.

\subsection{The Optimization Problem}
\label{sec:conditional-exit-stopping}
We consider an optimization problem which is underpinned by a  controlled state process $\bX=(X_t)_{t\ge 0}$ whose dynamics are given by:
\begin{equation}
\label{fo:state}
dX_t=\alpha_t dt + \sigma dW_t
\end{equation}
\begin{equation}
\label{fo:control_integrability}
\EE\int_0^T|\alpha_t|^pdt<\infty,
\end{equation}
for some $p\ge 1$, and we shall use $p=2$ most often.

\vskip 6pt
The goal of the optimization problem is to minimize a cost $J^\tau(\balpha)$ associated to the control process $\balpha$. This cost is derived from a running cost function $f:\RR^d\times A\mapsto\RR$ and a terminal cost function $g:\RR^d\mapsto\RR$ (whose regularity properties will be specified later on) in the form:
\begin{equation}
\label{fo:J_tau_of_alpha}
\begin{split}
J^\tau(\balpha)&=\int_0^T\EE[f(X_t,\alpha_t)|\tau^X_D\ge t]\;dt +\EE[g(X_T)|\tau^X_D\ge T]\\
&=\int_0^T
\frac{\EE\Bigl[f(X_t,\alpha_t)\textbf{1}_{\tau^X_D\ge t}\Bigr]}{\PP[\tau^X_D\ge t]}dt
+\frac{\EE\Bigl[g(X_T)\textbf{1}_{\tau^X_D\ge T}\Bigr]}{\PP[\tau^X_D\ge T]}.
\end{split}
\end{equation}
We shall assume that for each $x\in\RR^d$, the function $A\ni\alpha\mapsto f(x,\alpha)$ is convex.
Moreover, for illustration purposes, we shall often restrict ourselves to the case of separable running cost functions $f$ of the form:
\begin{equation}
\label{fo:separable_running_cost}
f(x,\alpha)=\frac12|\alpha|^2 + \tilde f(x)
\end{equation}
for some bounded measurable function $\tilde f$ on $\RR^d$.

\subsection{Evolutionary Biology Motivation}
\label{sub:motivation}

The mathematical formulation of the above optimization problem is non-standard, and except for the original lectures of P.L. Lions \cite{Lions2016CDF} and the numerical experiments presented in \cite{achdou2021optimal}, we do not know of any reported mathematical analysis of such a model. However, we claim that it is very natural from the perspective of the study of populations of altruistic individuals in high resource environments practicing egalitarian resource sharing. The analysis of the evolution of these populations could be based on dynamical models of the following type. We consider the evolution of identical individuals foraging for food independently of each other, in a safe territory. We assume that the outcome of foraging is random, and that at the end of each time period (one can think of a period as the weaning period for an offspring generation) the food is shared among the surviving individuals in an egalitarian manner which allots the same amount of food to each member still alive.

Let us denote by $D$ the territory, and let us assume that the individuals disappear or die when they leave the territory.
If for $i=1,\cdots,N_t$ we denote by $X^i_t$ the positions at time $t$ of the $N_t$ individuals still alive at the beginning of period $t+1$, we assume that foraging will take those who survive (i.e. do not exit the territory) to positions $X^i_{t+1}$ at the end of the period, and that they will have accumulated the amount $f(X_t^i)$ of resources (say food for example). Resource sharing takes place in the following form: all the resources are first aggregated, and then redistributed in equal amounts to the surviving members of the population. In other words, the resource allocated to each individual still alive is:
$$
\frac{1}{N_{t+1}}\sum_{i=1}^{N_{t+1}}f(X^i_{t+1}).
$$ 
So an individual still alive at the end of the $T$-th period will have benefitted from the resources:
$$
\sum_{t=0}^{T-1}\frac{1}{N_{t+1}}\sum_{i=1}^{N_{t+1}}f(X^i_{t+1}).
$$ 
Not surprisingly, in the limit of small foraging periods, the summation over time will converge toward the integral between $0$ and $T$ of the resource enjoyed at time $t$. What is more interesting is the form of the integrand when the size, say $N$,  of the population increases. Indeed, notice first that the ratio $N_t/N$ converges toward the probability that a typical individual is still alive at time $t$, in other words $\PP[\tau>t]$ if we use the notation $\tau$ for the time of death of the individual. Recall that the latter is the first exit time of the domain $D$. Next, the quantity
$$
\frac{1}{N}\sum_{i=1}^{N_{t}}f(X^i_{t})=\frac{1}{N}\sum_{i=1}^{N}f(X^i_{t})\textbf{1}_{\tau(X^i)>t}
$$ 
converges toward $\EE[f(X_t)\textbf{1}_{\tau>t}]$ and the integrand at time $t$ is indeed given by the conditional expectation of the resource at time $t$ given that the individual is still alive at that time. Optimization of the fitness of the individuals still alive naturally leads to the conditional control problem which we propose to study in this paper.

\subsection{Control of Nonlinear Feynman-Kac Semigroups}
\label{se:model-feynmankac}
We shall not solve the model with hard killing through the exit time introduced above. Instead, we shall provide a complete analysis of a relaxed version of the model based on \emph{soft killing}.
Indeed, it is natural to consider the following generalization of the original model proposed by P.L. Lions.
Working with the same basic  controlled state equation \eqref{fo:state}, we can generalize the conditioning by considering a measurable function $V:\RR^d\mapsto [0,\infty]$ which we assume to be non-negative for the sake of simplicity.
$V$ could as well be bounded below, or even have some negative singularities of a specific type, but we shall not worry about this type of generality in this paper.
The goal of the new formulation of the control problem is still to minimize a cost $J^V(\balpha)$ associated to a control process $\balpha$, and this cost is still derived from a running cost function $f:\RR^d\times A\mapsto\RR$ and a terminal cost function $g:\RR^d\mapsto\RR$ in the form:
\begin{equation}
\label{fo:J_V_of_alpha}
J^V(\balpha)=\int_0^T
\frac{\EE\Bigl[f(X_t,\alpha_t)e^{-\int_0^tV(X_s)ds}\Bigr]}{\EE\Bigl[e^{-\int_0^tV(X_s)ds}\Bigr]}dt
+\frac{\EE\Bigl[g(X_T)e^{-\int_0^TV(X_s)ds}\Bigr]}{\EE\Bigl[e^{-\int_0^TV(X_s)ds}\Bigr]}.
\end{equation}
Lions' model based on conditioning the state to remain in a given domain $D$ is recovered by considering the function $V=V^\infty$ given by:
\begin{equation}
\label{fo:V_infinity}
V^\infty(x)=
\begin{cases}
0&\text{if } x\in D\\
\infty&\text{otherwise},
\end{cases}
\end{equation}
in which case:
\begin{equation}
\label{fo:V_exit}
\int_0^tV^\infty(X_s)ds=
\begin{cases}
0&\text{if }X_s\in \overline D,\; 0\le s \le t\\
\infty&\text{if } X_s\notin \overline D \text{ for some } 0\le s\le t,
\end{cases}
\end{equation}
so that:
$$
e^{-\int_0^tV^\infty(X_s)ds}=\textbf{1}_{[X_s\in \overline D,\; 0\le s\le t]}=\textbf{1}_{[\tau_D\ge t]},
$$
where $\tau_D = \tau^{X}_D$ is the first exit time of the domain $D$ defined in~\eqref{eq:def-tauD}. 
Accordingly:
\begin{equation}
\label{fo:Lions_J_of_alpha}
J^{V^\infty}(\balpha)=\int_0^T
\EE\Bigl[f(X_t,\alpha_t)\big| \tau_D\ge t\Bigr]dt
+\EE\Bigl[g(X_T)\big| \tau_D\ge T\Bigr],
\end{equation}
which is indeed the case considered earlier in \eqref{fo:J_tau_of_alpha}. In what follows, we approximate $V^\infty$ by potential functions $V^n=nV^1$ where $V^1$ is a continuous approximation of the indicator function of the domain $D$. To be specific, we choose $V^1(x)=\chi^\epsilon(d(x,D))$ where $d(x,D)$ denotes the distance from $x\in\RR^d$ to the domain $D$, $\epsilon>0$ is an arbitrary fixed number whose specific value will not matter, and $\chi^\epsilon$ is the continuous function:
\begin{equation}
\label{fo:chi}
\chi^\epsilon(d)=
\begin{cases}
0&\text{if }d\le 0\\
\text{linear }&\text{if }0\le d\le \epsilon\\
1&\text{if } d\ge \epsilon.
\end{cases}
\end{equation}
Notice that since we assume that the boundary $\partial D$ is smooth, we can modify $\chi^\epsilon(d)$ when $0\le d\le \epsilon$ in such a way that $\chi^\epsilon$ can also be assumed to be smooth.
The choice of this family of potential functions is justified by the following simple result.

\begin{lemma}
\label{le:limit}
If $\bX=(X_t)_{t\ge 0}$ satisfies $X_t=x_0+\int_0^t\alpha_s ds+W_t$ for some $x_0\in D$ and $\balpha=(\alpha_t)_{t\ge 0}$ is admissible, then for any bounded function $g$
\begin{equation}
\label{fo:g_limit}
\EE[g(X_T)\;|\;\tau_D>T]=\lim_{n\to\infty}\frac{\EE[g(X_T)e^{-n\int_0^TV^1(X_s)ds}]}{\EE[e^{-n\int_0^TV^1(X_s)ds}]}.
\end{equation}
Similarly, if $\int_0^T\EE[|f(X_t,\alpha_t)|]dt<\infty$, we also have:
\begin{equation}
\label{fo:f_limit}
\int_0^T\EE[f(X_t,\alpha_t)\;|\;\tau_D>t]dt=\lim_{n\to\infty}\int_0^T\frac{\EE[f(X_t,\alpha_t)e^{-n\int_0^tV^1(X_s)ds}]}{\EE[e^{-n\int_0^tV^1(X_s)ds}]}.
\end{equation}
\end{lemma}

\begin{proof}
Notice that if $\tau_D\ge T$, $X_t\in D$ for $0\le t\le T$ and $\int_0^TV^1(X_s)ds=0$. On the other hand, if $\tau_D< T$, the set of times $t\le T$ for which $X_t\notin D$ is of positive Lebesgue's measure which implies that $\int_0^TV^1(X_s)ds>0$ since $V^1(x)>0$ if $x\notin D$. Consequently, 
$$
    \lim_{n\to\infty}e^{-n\int_0^TV^1(X_s)ds}=\textbf{1}_{\tau_D\ge T}
$$
and since $g$ is bounded, Lebesgue's dominated convergence theorem gives:
$$
    \lim_{n\to\infty}\EE[g(X_T)e^{-n\int_0^TV^1(X_s)ds}]=\EE[g(X_T)\textbf{1}_{\tau_D\ge T}]
$$
and using the same result with $g\equiv 1$ for the denominator, we get the desired limit \eqref{fo:g_limit}. The argument needed for the proof of \eqref{fo:f_limit} involving the running cost is exactly the same.
\end{proof}

\begin{remark}
Notice that the above lemma does not say anything when $x_0\notin D$. Indeed in this case, since the boundary is smooth, we have $\tau_D=0$ almost surely, and the left hand sides of \eqref{fo:g_limit} and \eqref{fo:f_limit} are referring to conditional expectations with respect to an event of probability zero.
\end{remark}

\vskip 6pt
The advantages of a soft killing given by a bounded continuous potential function $V$ are twofold: 1) it sets all the equations in the whole space $\RR^d$ and avoids having to deal with boundary conditions on $\partial D$; 2) it makes technical proofs easier as it provides continuity with respect to the time variable $t$ which may not be available otherwise. To be specific, we shall make the following assumptions.

\subsection{Assumptions}
\label{sub:assumptions}

Our analysis is predicated on the following assumptions which will be in force throughout the remainder of the paper, even if many of the individual results still hold under weaker conditions.

\begin{assumption}
\label{assumption:g-f-cont}
The running cost and the terminal cost functions satisfy (recall that the action space $A$ is a closed convex subset of $\RR^d$):
\begin{itemize}\itemsep=-1pt
\item The function $g$ is Lipschitz continuous and bounded on $\RR^d$;
\item For each $\alpha\in A$, the function $f(\cdot,\alpha)$ is continuous and bounded on $\RR^d$.
\item For each $x \in \RR^d$, the function $f(x,\cdot)$ is convex on $A$.
\item There exist $C_1>0$ and $C_2>0$ such that $C_1(1+|\alpha|^2)\le f(x,\alpha)\le C_2(1+|\alpha|^2),\;\; x\in\RR^d,\alpha\in A$.
\end{itemize}
\end{assumption}
As stated earlier, we shall often concentrate on the case of a separable running cost function of the form
\eqref{fo:separable_running_cost} for a bounded Lipschitz continuous function $\tilde f$. As for the potential function $V$, in order to be specific, we will make the following assumption.
\begin{assumption}
\label{assumption:V-bdd}
	The function $V$ is Lipschitz continuous on $\RR^d$ and $0\le V\le 1$. In fact, without any loss of generality, we shall assume that it is continuously differentiable with bounded derivatives when needed.
\end{assumption}

\subsection{Approximation by Bounded Controls}
\label{sub:bounded_controls}
In this subsection we work on a probability space $(\Omega,\cF,\PP)$ equipped with a Wiener process $\bW=(W_t)_{0\le t\le T}$ and for each process $\balpha=(\alpha_t)_{0\le t\le T}$ adapted to the filtration of the Brownian motion $\bW$ satisfying the integrability condition \eqref{fo:control_integrability} with $p=2$ we consider the corresponding state process $\bX=(X_t)_{0\le t\le T}$ satisfying the state dynamics \eqref{fo:state}. For the purpose of what we are about to do, $\balpha$ could be a general adapted process, or it could be given in feedback form $\alpha_t=\phi_t(X_t)$ for a measurable function $\phi$.
As we do throughout the paper, we denote by $\nu_t$ the distribution of $X_t$ and by $\mu_t$ the conditioned probability measure defined on $\RR^d$ by:
\begin{equation}
    \label{fo:mu_t}
    \mu_t(dx)=\frac{1}{\EE[e^{-A_t}]}\EE\bigl[\delta_{X_t}(dx)e^{-A_t}\bigr]
\end{equation}
with $A_t=\int_0^T V(X_s)ds$.

\vskip 4pt

For each $K>0$, we denote by $\balpha^K=(\alpha^K_t)_{0\le t\le T}$ given by $\alpha_t^K=\alpha_t\mathbf{1}{|\alpha_t|\le K}$, by $\bX^K=(X^K_t)_{0\le t\le T}$ the corresponding state process and by $A^K_t=\int_0^TV(X^K_s)ds$ the corresponding additive functional. 

\begin{lemma}
    \label{le:bounded_convergence}
    For any adapted control process $\balpha=(\alpha_t)_{0\le t\le T}$ satisfying 
\begin{equation}
    \label{fo:integrability_assumption}
    \EE\int_0^T |\alpha_t|^2 dt <\infty
\end{equation}
we have 
    \begin{equation}
        \label{fo:alpha_upper_bound}
        |J^V(\balpha)-J^V(\balpha^K)|\le C\bigl[\varepsilon(K)^{1/2} + \varpi(K) \bigr]
    \end{equation}
    where we used the notations
    \begin{equation}
    \label{fo:varepsilon}
    \varepsilon(K)=\EE\int_0^T|\alpha_t|^2\mathbf{1}_{|\alpha_t|>K}dt,
    \quad\text{and}\quad
    \varpi(K) = \int_0^T\EE\Bigl[ |\alpha^K_t|^2 |e^{-A_t}-e^{-A^K_t}|\Bigr]dt,
\end{equation}
and where the constant $C$ depends only upon the data (i.e. $T$, the sup-norms of $\tilde f$ and $g$, and the Lipschitz constants of $\tilde f$, $g$ and $V$). In particular, since both $\varepsilon(K)$ and $\varpi(K)$ converge to $0$ when $K\nearrow\infty$, we have:
\begin{equation}
    \label{fo:convergence}    
    \lim_{K\nearrow\infty}J^V(\balpha^K)=J^V(\balpha).
\end{equation}
\end{lemma}

\begin{proof}
For the sake of simplicity, we only consider the case of separable running costs given by \eqref{fo:separable_running_cost}.
    \begin{equation}
        \begin{split}
            J^V(\balpha)-J^V(\balpha^K)
            &=\int_0^T\frac1{\EE[e^{-A_t}]}\EE\bigl[\bigl(f(X_t,\alpha_t)-f(X^K_t,\alpha^K_t)\bigr)e^{-A_t}\bigr]dt
            +\frac{\EE\bigl[\bigl(g(X_T)-g(X^K_T)\bigr)e^{-A_T}\bigr]}{\EE[e^{-A_T}]}\\
            &+\int_0^T\EE\Bigl[f(X^K_t,\alpha^K_t)\Bigl(
            \frac{e^{-A_t}}{\EE[e^{-A_t}]}-\frac{e^{-A^K_t}}{\EE[e^{-A^K_t}]}\Bigr)\Bigr]dt
            +\EE\Bigl[g(X^K_T)\Bigl(\frac{e^{-A_T}}{\EE[e^{-A_T}]}-\frac{e^{-A^K_T}}{\EE[e^{-A^K_T}]}\Bigr)\Bigr]\\
            &= (i) + (ii) + (iii) + (iv).
        \end{split}
    \end{equation}
Notice that
\begin{equation}
        \begin{split}
\EE\int_0^T|X_t-X^K_t|dt&\le\EE\int_0^T\int_0^t|\alpha_s-\alpha^K_s|ds\;dt\\
&=\EE\int_0^T\int_0^t|\alpha_s|\mathbf{1}_{|\alpha_s|>K}ds\;dt\\
&\le T\EE\int_0^T|\alpha_t|\mathbf{1}_{|\alpha_t|>K}dt,
        \end{split}
\end{equation}
and that
\begin{equation}
        \begin{split}
|e^{-A_t}-e^{-A^K_t}|&\le|A_t-A^K_t|\\
&=\bigl|\int_0^t[V(X_s)-V(X^K_s)]ds\bigr|\\
&\le\|V\|_{Lip1}\int_0^t|X_s-X^K_s|ds\\
&\le T\|V\|_{Lip1}\int_0^T|\alpha_t|\mathbf{1}_{|\alpha_t|>K}dt.
        \end{split}
\end{equation}
For every $K>0$ and every $(t,\omega)\in[0,T]\times \Omega$ we have $|\alpha^K_t|^2|e^{-A_t}-e^{-A^K_t}|\le 2|\alpha_t|^2$ whose integral with respect to $dt\PP(d\omega)$ is finite. Moreover, $\lim_{K\nearrow \infty}\int_0^T|\alpha_t|\mathbf{1}_{|\alpha_t|>K}dt=0$ $\PP$-almost surely proving that 
$$
\lim_{K\nearrow \infty}\varpi(K)=0.
$$
Now,
\begin{equation}
\label{fo:my_(i)}
        \begin{split}
|(i)|&\le \int_0^T\frac1{\EE[e^{-A_t}]}\EE\bigl[\bigl|\tilde f(X_t)-\tilde f(X^K_t)\bigr|e^{-A_t}\bigr]dt
+\frac12\int_0^T\frac1{\EE[e^{-A_t}]}\EE\bigl[\bigl(|\alpha_t|^2-|\alpha^K_t|^2\bigr)e^{-A_t}\bigr]dt\\
&\le e^T\|\tilde f\|_{Lip1}\EE\int_0^T|X_t-X^K_t|dt +\frac12\int_0^T\frac1{\EE[e^{-A_t}]}\EE\bigl[\bigl||\alpha_t|^2\mathbf{1}_{|\alpha_t|>K}e^{-A_t}\bigr]dt\\
&\le Te^T\|\tilde f\|_{Lip1}\EE\int_0^T|\alpha_t|\mathbf{1}_{|\alpha_t|>K}dt +\frac{e^T}2\EE\int_0^T|\alpha_t|^2\mathbf{1}_{|\alpha_t|>K}dt\\
&\le \frac{e^T}2\varepsilon(K) +T^{3/2}e^T\|\tilde f\|_{Lip1}\varepsilon(K)^{1/2}
        \end{split}
\end{equation}
if we use the notation  introduced in \eqref{fo:varepsilon} for $\varepsilon(K)$,
and use H\"older's inequality. Similarly
\begin{equation}
\label{fo:my_(ii)}
        \begin{split}
|(ii)|&\le e^T\|g\|_{Lip1}\EE[|X_T-X^K_T|]\\
&\le e^T\|g\|_{Lip1}\EE\int_0^T|\alpha_t|\mathbf{1}_{|\alpha_t|>K}dt\\
&\le e^T\|g\|_{Lip1}\varepsilon(K)^{1/2}.
        \end{split}
\end{equation}
Next
\begin{equation}
\label{fo:(iii)}
\begin{split}
|(iii)|
&\le \int_0^T\EE\Bigl[|\tilde f(X^K_t)|\Bigl|\frac{e^{-A_t}}{\EE[e^{-A_t}]}-\frac{e^{-A^K_t}}{\EE[e^{-A^K_t}]}\Bigr|\Bigr]dt
+\frac12\int_0^T\EE\Bigl[ |\alpha^K_t|^2\Bigl|\frac{e^{-A_t}}{\EE[e^{-A_t}]}-\frac{e^{-A^K_t}}{\EE[e^{-A^K_t}]}\Bigr|\Bigr]dt\\
&\le \|\tilde f\|_\infty\int_0^T\Bigl(
\frac{1}{\EE[e^{-A_t}]}\EE[|e^{-A_t}-e^{-A^K_t}|] +\EE[e^{-A^K_t}]\Bigl(\frac1{\EE[e^{-A_t}]}-\frac1{\EE[e^{-A^K_t}]}\Bigr)\Bigr)dt\\
&\hskip 45pt
+\frac12\int_0^T\frac1{\EE[e^{-A^K_t}]}\EE\Bigl[ |\alpha^K_t|^2\Bigl|e^{-A_t}-e^{-A^K_t}\Bigr|\Bigr]dt
+\frac12\int_0^T\EE\Bigl[ |\alpha^K_t|^2 e^{-A_t}\Bigl|\frac1{\EE[e^{-A_t}]}-\frac1{\EE[e^{-A^K_t}]}\Bigr|\Bigr]dt\\
&\le 2e^T\|\tilde f\|_\infty\int_0^T\EE[|e^{-A_t}-e^{-A^K_t}|]dt\\
&\hskip 45pt
+\frac{e^T}{2}\int_0^T\EE\Bigl[ |\alpha^K_t|^2 |e^{-A_t}-e^{-A^K_t}|\Bigr]dt
+\frac{e^T}{2}\int_0^T\EE\Bigl[ |\alpha^K_t|^2\frac{e^{-A_t}}{\EE[e^{-A_t}]}\Bigr]
\EE[|e^{-A_t}-e^{-A^K_t}|]dt\\
&\le \bigl( 2e^T\|\tilde f\|_\infty+\frac{T^{3/2}}2 e^{2T}\|V\|_{Lip1}\varepsilon(K)^{1/2}
+\frac{e^T}{2}\int_0^T\EE\Bigl[ |\alpha^K_t|^2 |e^{-A_t}-e^{-A^K_t}|\Bigr]dt.
\end{split}
\end{equation}
Finally
\begin{equation}
\label{fo:(iv)}
\begin{split}
|(iv)|
&\le \EE\bigl[|g(X^K_t)|\frac1{\EE[e^{-A_T}]} |e^{-A_T}-e^{-A^K_T}|\bigr]+\EE\bigl[|g(X^K_t)|\;e^{-A^K_T}\bigl|\frac1{\EE[e^{-A_T}]} -\frac1{\EE[e^{-A^K_T}]}\bigr|\bigr]\\
&\le \|g\|_\infty e^T \EE[|e^{-A_T}-e^{-A^K_T}|] +\|g\|_\infty e^{2T} \EE[|e^{-A_T}-e^{-A^K_T}|]\\
&\le 2\|g\|_\infty e^{2T}T\|V\|_{Lip1}\varepsilon(K)^{1/2}.
\end{split}
\end{equation}
Putting together \eqref{fo:my_(i)}, \eqref{fo:my_(ii)}, \eqref{fo:(iii)}, and \eqref{fo:(iv)} gives the desired estimate \eqref{fo:alpha_upper_bound}. Clearly, 
$$
\lim_{K\nearrow\infty}\varepsilon(K)=0,
$$
and we already argued that $\varpi(K)$ also converges to $0$.
\end{proof}

The first obvious consequence of the above result is that the search for the infimum
\begin{equation}
    \label{fo:Jstar}    
    J^{V*}=\inf_{\balpha}J^V(\balpha)
\end{equation}
over all the adapted processes satisfying the integrability condition \eqref{fo:control_integrability} with $p=2$ can be restricted to bounded control processes $\balpha$. However, if and when we are able to prove that this infimum is attained, it will require extra work to prove that the minimizer is in fact a bounded control.

\section{\textbf{The Case of Markovian Feedback Controls}}
\label{sec:V-closedloop}

\subsection{Formulation of the problem}
\label{sec:formulation-markovian-deterministic}
In this section, we restrict the optimization problem to controls processes $\balpha=(\alpha_t)_{0\le t\le T}$ of the form:
\begin{equation}
\label{fo:feedback}
\alpha_t=\phi_t(X_t)
\end{equation}
where $\phi:[0,T]\times \RR^d\mapsto A$ is a  (deterministic) measurable function, and $\bX=(X_t)_{0\le t\le T}$ (which we shall sometimes denote $\bX^{\phi}$ to emphasize the dependence upon the feedback function $\phi$) satisfies:
\begin{equation}
\label{fo:phi_sde}
dX_t=\phi_t(X_t)dt+dW_t.
\end{equation}
 We will refer to such controls as \emph{Markovian} controls or \emph{feedback} controls.
For a given feedback control function $\phi$, the existence of the controlled state process requires the solution of a stochastic differential equation. So in order to be admissible, a feedback function should be at a minimum, a measurable functions $\phi$ for which such a solution exists and satisfies
\begin{equation}
\label{fo:necessary}
\EE\Bigl[\int_0^T|\phi_t(X_t)|^p\Bigr]<\infty,
\end{equation}
with $p=2$ in order for the objective function we plan to minimize to make sense. 
We shall denote by $\Phi^{(p)}$ the set of admissible feedback control functions, namely the set of $\RR^d$-valued measurable functions $\phi$ 
on $[0,T]\times\RR^d$ for which the stochastic differential equation \eqref{fo:phi_sde} has a weak solution satisfying \eqref{fo:necessary}.
In some cases, we shall make a stronger assumption on the feedback controls. For this reason, we introduce the space $\Phi^{(\infty)}$ of $\RR^d$-valued bounded measurable functions on $[0,T]\times\RR^d$.
\vskip 2pt
In dimension $d=1$, all the measurable bounded functions $\phi$ are in all the $\Phi^{(p)}$ because of the classical result of Zvonkin \cite{zvonkin1974transformation} which guarantees existence of strong solutions for \eqref{fo:phi_sde}. In dimension $d>1$  existence and uniqueness of a strong solutions still hold for bounded drifts because the volatility is the identity matrix. More general volatility terms could be accommodated by Veretennikov's extension \cite{veretennikov1981strong} which provides a large class of bounded measurable functions $\phi$ for which the result still holds in more general non-degenerate cases. However, as we shall point out in several instances, our main requirement is the existence of weak solutions, and this is guaranteed under much weaker conditions on $\phi$ as long as the volatility is the identity matrix. 

\vskip 2pt
Throughout the section, we alternatively use the notation $\bphi$ or $\balpha$ for the control, even if $\phi$ is only the feedback function determining the actual control process $\balpha$ through the solution of the equation \eqref{fo:phi_sde}. 

\subsection{A Non-local Fokker-Planck-Kolmogorov (FPK) Equation}
\label{sub:FPK}

Given the form of the running and terminal costs, the objective function \eqref{fo:J_V_of_alpha} can be rewritten in the form:

\begin{equation}
\label{fo:new_J_of_phi}
J(\phi)=\int_0^T\int f(x,\phi_t(x))\mu_t(dx)\;dt+\int g(x)\mu_T(dx),
\end{equation}
where we use the notation $\mu_t$ for the probability measure:
\begin{equation}
    \label{fo:mu_1_t}
    \mu_t(dx)=\frac{\EE[\delta_{X_t}(dx) e^{- A_t}]}{\EE[ e^{- A_t}]},\qquad 0\le t\le T,
\end{equation}
where from now on, we use the notation $A_t=\int_0^t V(X_s)ds$. Notice that $t\mapsto \mu_t$ is continuous for the topology of weak convergence of probability measures, so in the sequel, we shall restrict ourselves to continuous flows of probability measures.

\begin{lemma}
\label{le:FPK1}
The measure valued function $t\mapsto \mu_t$ defined in \eqref{fo:mu_1_t} satisfies the (non-local) forward Fokker-Planck-Kolmogorov (FPK) equation:
\begin{equation}
\label{fo:kolmo1}
\partial_t \mu_t =\frac12\Delta \mu_t - \div(\phi_t \mu_t) - (V-<\mu_t,V>) \mu_t,
\end{equation}
 in the sense of Schwartz distributions.
 \end{lemma}
The notation $<\mu,V>$ used in formula \eqref{fo:kolmo1} stands for $<\mu,V>=\int_{\RR^d} V(x)\mu(dx)$. 
\begin{proof}
If $\varphi$ is a smooth function on $\RR^d$ with compact support, It\^o's formula gives:
\begin{equation*}
\begin{split}
\frac{d}{dt} <\mu_t,\varphi> & = \frac{d}{dt}\frac{\EE[\varphi(X_t) e^{- A_t}]}{\EE[ e^{- A_t}]}\\
& =\frac{1}{\EE[ e^{- A_t}]}
\EE\Bigl[\Bigl(\frac12\Delta_x \varphi(X_t)  +\phi_t(X_t)\cdot\nabla_x\varphi(X_t)\Bigr)  e^{- A_t}\\ 
&\hskip 75pt
 - \varphi(X_t) V(X_t)  e^{- A_t} \Bigr]+\EE[\varphi(X_t) e^{- A_t}]\frac{\EE[V(X_t)e^{-A_t}]}{\EE[ e^{- A_t}]^2}\\
& = <\mu_t,\; \frac12\Delta_x \varphi  +\phi_t \cdot \nabla_x\varphi - (V -<\mu_t,V>)\varphi >\\
& = <\frac12\Delta_x \mu_t - \div_x(\phi_t \mu_t) -(V-<\mu_t,V>) \mu_t,\; \varphi>
\end{split}
\end{equation*}
where we used stochastic integration by parts and the fact that $\varphi$ has compact support.
\end{proof}

\vskip 6pt
For general existence and uniqueness results for classical, i.e. without the third term in the right hand side of \eqref{fo:kolmo1}, FPK equations, together with existence and regularity results for possible density for the solutions, we refer the interested reader to  \cite[Chapter 6]{bogachev2022fokker} and references therein. However, these results are not general enough to cover the case of equation \eqref{fo:kolmo1} of interest to us. Indeed, as far as we can tell, the mean field nature of equation \eqref{fo:kolmo1} makes it escape the realm of those results. For the purpose of our analysis, we proceed in the following way. Given a measurable feedback function $\phi$ for which there exists a weak solution $X_t$ of the stochastic differential equation
\begin{equation}
\label{fo:state_SDE}
dX_t=\phi_t (X_t)dt + dW_t,
\end{equation}
Lemma~\ref{le:FPK1} guarantees that the flow of probability measures $\mu_t$ defined by \eqref{fo:mu_1_t} with $A_t=\int_0^tV(X_s)ds$ 
is continuous and is a solution of \eqref{fo:kolmo1}. This is mostly what we shall need for the existence of solutions of \eqref{fo:kolmo1}. Still the following properties of such solutions will come handy in the sequel.

\begin{proposition}
    \label{pr:densities} 
    Let us assume that $\bphi=\bigl(\phi_t(x)\bigr)_{0\le t\le T,\;x\in\RR^d}$ is an $\RR^d$-valued measurable function on $[0,T]\times\RR^d$ and $\bmu=(\mu_t)_{0\le t\le T}$ is a  measurable flow  of probability measures satisfying
    \begin{equation}
    \label{fo:integrability}
    \int_0^T\int_{\RR^d}|\phi_t(x)|^p\mu_t(dx)dt<\infty,
    \end{equation}
    for some $p\ge 1$. If $\bmu=(\mu_t)_{0\le t\le T}$ solves the Fokker-Planck-Kolmogorov equation \eqref{fo:kolmo1} in the sense of distributions, then
    \begin{itemize}
    \item[(i)] $\bmu(dx,dt)=\mu_t(dx)dt=\rho_t(x)dxdt$ for some non-negative measurable function $\rho\in L^r_{loc}([0,T]\times\RR^d)$ for every $r\in[1,(d+2)')$.
    \item[(ii)] If $p>d+2$, the density $\rho$ can be chosen to belong to the Sobolev spaces $\cH^{1,p}(I\times K)$ for every closed interval $I\subset (0,T)$ and compact set $K\subset\RR^d$. 
    \item[(iii)] If $\phi$ is bounded, the continuous version of the density is strictly positive, hence bounded below away from $0$ on every closed interval $I\subset (0,T)$ and compact set $K\subset\RR^d$. 
    \end{itemize}
\end{proposition}

\begin{proof}
We use a prime $p'$ to denote the conjugate exponent of $p\in[0,\infty)$.   (i) is a direct consequence of \cite[Corollary 6.3.2]{bogachev2022fokker} with $c=V-<\mu_t,V>$. (ii) is a direct consequence of \cite[Corollary 6.4.3]{bogachev2022fokker}, given the fact that we use the definition of the space $\cH^{1,p}$ which can be found on page 245 of this book. As for (iii), it follows directly from the properties of $\nu_t=\cL(X_t)$.
\end{proof}

\subsection{A Form of Superposition Principle}
\label{sub:superposition}
 
In this subsection, we start with a couple $(\bphi,\bmu)$ where $\bphi=\bigl(\phi_t(x)\bigr)_{0\le t\le T,\; x\in\RR^d}$ is a $\RR^d$-valued measurable function on $[0,T]\times\RR^d$ and $\bmu=(\mu_t)_{0\le t\le T}$ is a  measurable flow  of probability measures satisfying
\begin{equation}
\label{fo:compatibility}
\int_0^T\int_{\RR^d}\;|\phi_t(x)|^2\mu_t(dx)dt<\infty
\end{equation}
and the non-local FPK equation \eqref{fo:kolmo1}, and we construct a stochastic process $\bX=(X_t)_{0\le t\le T}$ solution of the state stochastic differential equation \eqref{fo:phi_sde} which is related to $\mu_t$ through formula \eqref{fo:mu_1_t}. This result should be viewed as a non-linear or non-local superposition principle in the spirit of the superposition principle proved by Trevisan in \cite{Trevisan} in the classical case.

\vskip 2pt
Notice that since the potential function $V$ is bounded, assumption \eqref{fo:compatibility} is equivalent to our prior assumption \eqref{fo:necessary}.

\begin{lemma}
\label{le:gamma_t}
If $(\mu_t)_{0\le t\le T}$ solves the FPK equation \eqref{fo:kolmo1}, then the flow $(\gamma_t)_{0\le t\le T}$ of non-negative measures defined by
\begin{equation}
\label{fo:gamma_t}
\gamma_t = e^{-\int_0^t<\mu_s,V>ds} \mu_t
\end{equation}
is the unique solution of the linear PDE
\begin{equation}
\label{fo:gamma_t_pde}
\partial_t\gamma_t = \frac12\Delta\gamma_t-\div(\phi_t\gamma_t)-V\gamma_t
\end{equation}
with initial condition $\gamma_0=\mu_0$.
\end{lemma}

\begin{proof}
If $\gamma_t$ is defined by \eqref{fo:gamma_t} in Lemma~\ref{le:gamma_t}, and $\varphi$ is a test function, we have
\begin{equation*}
\begin{split}
\frac{d}{dt}<\gamma_t,\varphi>&=<\mu_t,V>e^{\int_0^t<\mu_s,V>ds} <\mu_t,\varphi>+e^{\int_0^t<\mu_s,V>ds} \frac{d}{dt}<\mu_t,\varphi>\\
&=-<\mu_t,V><\gamma_t,\varphi>+e^{-\int_0^t<\mu_s,V>ds} <\mu_t,\frac12\Delta\varphi+\phi_t\nabla\varphi
-(V-<\mu_t,V>)\varphi>\\
&=e^{-\int_0^t<\mu_s,V>ds} <\mu_t,\frac12\Delta\varphi+\phi_t\nabla\varphi
-V\varphi>\\
&= <\gamma_t,\frac12\Delta\varphi+\phi_t\nabla\varphi
-V\varphi>.
\end{split}
\end{equation*}
Uniqueness is a consequence of \cite[Theorem 9.3.6]{bogachev2022fokker} and the fact that $\int_0^T\int_{\RR^d}|\phi_t(x)|\gamma_t(dx)dt<\infty$ which is implied by our assumption \eqref{fo:compatibility} and the definition \eqref{fo:gamma_t}.
\end{proof}

\begin{lemma}
    There exists a weak solution $(X_t)_{0\le t\le T}$ of the stochastic differential equation \eqref{fo:phi_sde} for which
\begin{equation}
    \label{fo:gamma_t_FK}
    \gamma_t(dx)=\EE\bigl[\delta_{X_t}(dx)e^{-\int_0^t V(X_s)ds}\bigr].
\end{equation}
\end{lemma}
\begin{proof}
    For each $\epsilon>0$, let $\rho_\epsilon$ be the $d$-dimensional Gaussian density with mean $0$ and variance $\epsilon$ times the identity matrix in dimension $d$, and for each $t\in[0,T]$, let
\begin{equation}
    \label{fo:gamma_t_epsilon}
\gamma_t^\epsilon=\gamma_t*\rho_\epsilon,
\qquad\text{and}\qquad
\phi_t^\epsilon=\frac{d\bigl((\phi_t\gamma_t)*\rho_\epsilon\bigr)}{d(\gamma_t*\rho_\epsilon)},
\qquad\text{and}\qquad
V^\epsilon_t=\frac{d\bigl((V\gamma_t)*\rho_\epsilon\bigr)}{d(\gamma_t*\rho_\epsilon)}.
\end{equation}
$\gamma^\epsilon_t$ is the unique solution of the linear FPK equation
\begin{equation}
    \label{fo:FPK_epsilon}
    \partial\gamma_t^\epsilon=\frac12\Delta \gamma_t^\epsilon -\div(\phi^\epsilon_t\gamma_t^\epsilon)-V^\epsilon_t\gamma_t^\epsilon.
\end{equation}
Indeed, for each test function $\varphi\in C^\infty_c$ we have:
\begin{equation*}
    \begin{split}
        \partial_t<\varphi,\gamma_t^\epsilon>&=\partial_t<\varphi*\rho_\epsilon,\gamma_t>\\
        &=<\frac12\Delta(\varphi*\rho_\epsilon)+\phi_t\nabla(\varphi*\rho_\epsilon)-V(\varphi*\rho_\epsilon),\gamma_t>\\
        &=<\frac12\Delta\varphi,\gamma_t*\rho_\epsilon>+<\nabla\varphi,(\phi_t\gamma_t)*\rho_\epsilon>-<\varphi,(V\gamma_t)*\rho_\epsilon>\\
        &=<\frac12\Delta\varphi +\frac{d\bigl((\phi_t\gamma_t)*\rho_\epsilon\bigr)}{d\gamma_t^\epsilon}\nabla\varphi- \frac{d\bigl((V\gamma_t)*\rho_\epsilon\bigr)}{d\gamma_t^\epsilon}\varphi,\gamma_t^\epsilon>\\        
        &=<\varphi,\frac12\Delta\gamma_t^\epsilon - \div\bigl(\phi^\epsilon_t\gamma_t^\epsilon\bigr)
        -V_t^\epsilon\gamma_t^\epsilon>.      
    \end{split}
\end{equation*}
Using \cite[Lemma 8.1.10]{Ambrosio_et_al} we get 
\begin{equation}
    \label{fo:epsilon_bound}
    \int_0^T\int_{\RR^d}|\bphi_t^\epsilon|^2d\gamma_t^\epsilon dt\le \int_0^T\int_{\RR^d}|\bphi_t|^2d\gamma_t dt
    \le \int_0^T\int_{\RR^d}|\bphi_t|^2d\mu_t dt <\infty
\end{equation}
because of \eqref{fo:compatibility}. Using once more \cite[Theorem 9.3.6]{bogachev2022fokker}, we conclude that for each $\epsilon>0$, $(\gamma_t^\epsilon)_{0\le t\le T}$ is the unique solution of \eqref{fo:FPK_epsilon}.

\vskip 2pt
For each $\epsilon>0$, \cite[Theorem 6.6.2]{bogachev2022fokker} gives existence of a narrowly continuous flow $(\nu^\epsilon_t)_{0\le t\le T}$ of probability measures satisfying 
$$
\partial_t\nu^\epsilon_t=\frac12\Delta\nu^\epsilon_t-\div(\phi^\epsilon_t\nu^\epsilon_t)
$$
and using the superposition theorem (see for example \cite[Theorem 2.5]{Trevisan}) we deduce the existence of a probability measure $\PP^\epsilon$ solving the martingale problem in the sense that for every test function $\varphi\in C^\infty_c(\RR^d)$, 
\begin{equation}
    \label{fo:martingale_pb}
\varphi(X_t)-\varphi(X_0)-\int_0^t[\frac12\Delta\varphi(X_s)+\phi^\epsilon_s(X_s)\nabla\varphi(X_s)]ds
\end{equation}
where $(X_t)_{0\le t\le T}$ is the coordinate process on $C([0,T];\RR^d)$, is a $\PP^\epsilon$-martingale for the canonical filtration of $C([0,T];\RR^d)$, and for each $t\in[0,T]$, $\nu^\epsilon_t$ is the law of $X_t$ under $\PP^\epsilon$.

\vskip 1pt
We conclude the proof assuming momentarily the result of Lemma~\ref{le:tightness} below.
$\PP$ being a solution of the martingale problem for the drift $\phi_t$, for each test function $\varphi\in C^\infty_c(\RR^d)$, integration by parts implies that
\begin{equation}
    \label{fo:martingale}
\varphi(X_t)e^{-\int_0^t V(X_s)ds} -\varphi(X_0)-\int_0^t e^{-\int_0^s V(X_r)dr}\bigl[\frac12\Delta\varphi(X_s)+\phi_s(X_s)\nabla\varphi(X_s)-V(X_s)\varphi(X_s)\bigr]ds
\end{equation}
is a zero-expectation $\PP$-martingale. So if we define the flow of measures $(\tilde\gamma_t)_{0\le t\le T}$ by
$$
\tilde\gamma_t(dx)=\EE^{\PP}\bigl[\delta_{X_t}(dx)e^{-\int_0^tV(X_s)ds}\bigr]
$$
then, taking $\PP$-expectation of \eqref{fo:martingale} we get
$$
<\varphi,\tilde\gamma_t>=<\varphi,\tilde\gamma_0>+\int_0^t<\frac12\Delta\varphi+\phi_t\nabla\varphi-V\varphi,\tilde\gamma_s>ds
$$
which shows that $\tilde\gamma_T$ satisfies the PDE \eqref{fo:gamma_t_pde}, and by uniqueness, that $\tilde\gamma_t=\gamma_t$ and consequently, that formula \eqref{fo:gamma_t_FK} holds.
\end{proof}

\begin{lemma}
\label{le:tightness}
The family $(\PP^\epsilon)_{\epsilon>0}$ is tight and any limit point $\PP$ solves the martingale problem \eqref{fo:martingale_pb} with $(\phi_t)_{0\le t\le T}$ instead of $(\phi^\epsilon_t)_{0\le t\le T}$
\end{lemma}

 \begin{proof}
\emph{Step 1. Tightness. } For $\delta>0$, we denote by $\omega_x(\delta)=\sup_{0\le s,t\le T,\;|t-s|\le\delta}|x(t)-x(s)|$ the $\delta$-modulus of continuity of any function $x\in C([0,T];\RR^d)$. For $\delta'>0$ we have
\begin{equation*}
    \begin{split}
        \PP^\epsilon[\omega_\cdot(\delta)>\delta']
&\le\PP^\epsilon\Bigl[\sup_{|t-s|\le\delta}|\int_s^t\phi^\epsilon_r(X_r)dr|>\delta'/2]+\PP^\epsilon[\sup_{|t-s|\le\delta}|W_t-W_s|>\delta'/2\Bigr]\\
&\le\PP^\epsilon\Bigl[\sup_{|t-s|\le\delta}\delta^{1/2}\bigl(\int_0^T|\phi^\epsilon_r(X_r)|^2dr\bigr)^{1/2}>\delta'/2\Bigr]+\PP^\epsilon\bigl[\sup_{0\le r\le\delta}|W_r|>\delta'/2\bigr]\\
&\le\frac{4\delta}{\delta'^2}\EE^{\PP^\epsilon}\Bigl[\int_0^T|\phi^\epsilon_r(X_r)|^2dr\Bigr]+o(\delta)\\
&\le\frac{4\delta}{\delta'^2}\int_0^T\int_{\RR^d}|\phi_t^\epsilon(x)|^2\nu^\epsilon_t(dx)dt+o(\delta)
    \end{split}
\end{equation*}
and we conclude that for each $\delta'>0$, we have
$$
\lim_{\delta\searrow 0}\sup_{\epsilon>0}\PP^\epsilon[\omega_\cdot(\delta)>\delta']=0
$$
because
$$
\sup_\epsilon \int_0^T\int_{\RR^d}|\phi_t^\epsilon(x)|^2\nu^\epsilon_t(dx)dt\le \sup_\epsilon \int_0^T\int_{\RR^d}|\phi_t^\epsilon(x)|^2\gamma^\epsilon_t(dx)dt <\infty
$$
is implied by \eqref{fo:epsilon_bound}. Given the fact that the probability measures $(\nu^\epsilon_0)_\epsilon$ are tight because the $(\gamma^\epsilon_0)_\epsilon$ are, we conclude that the $(\PP^\epsilon)_\epsilon$ are themselves tight on $C([0,T];\RR^d)$.

\vskip 2pt\noindent
\emph{Step 2. Any limit point $\PP$ solves the martingale problem for $(\phi_t)_{0\le t\le T}$. } Let us fix $0\le s<t\le T$ and let $Y$ be a bounded continuous function on $C([0,T];\RR^d)$ measurable with respect to the past up to time $s$. For any test function $\varphi\in C^\infty_c(\RR^d)$ and for aany $\epsilon>0$ we have:
$$
0=\EE^{\PP^\epsilon}\Bigl[ Y\bigl[\varphi(X_t)-\varphi(X_s)-\int_s^t[\frac12\Delta\varphi(X_r)+\phi^\epsilon_r(X_r)\nabla\varphi(X_r)]dr\bigr]\Bigr].
$$
In order to prove the same equality for $\PP$ and $\phi$ instead of $\PP^\epsilon$ and $\phi^\epsilon$, it is enough to prove:
\begin{equation}
    \label{fo:(o)}
\lim_{\epsilon\searrow 0}\EE^{\PP^\epsilon}\Bigl[ Y\int_s^t\phi^\epsilon_r(X_r)\nabla\varphi(X_r)]dr\Bigr]
=\EE^{\PP}\Bigl[ Y\int_s^t\phi_r(X_r)\nabla\varphi(X_r)]dr\Bigr].
\end{equation}
Given a bounded Lipschitz continuous function $\tilde\phi\in C([0,T]\times\RR^d)$ with compact support,  we have:
\begin{equation}
    \label{fo:tilde}
\lim_{\epsilon\searrow 0}\EE^{\PP^\epsilon}\Bigl[ Y\int_s^t\tilde\phi_r(X_r)\nabla\varphi(X_r)]dr\Bigr]
=\EE^{\PP}\Bigl[ Y\int_s^t\tilde\phi_r(X_r)\nabla\varphi(X_r)]dr\Bigr]
\end{equation}
since $\PP^\epsilon$ converges weakly toward $\PP$.
So in order to prove \eqref{fo:(o)} it is be enough to prove
\begin{equation}
    \label{fo:(i)}
\lim_{\epsilon\searrow 0}\EE^{\PP^\epsilon}\Bigl[ Y\int_s^t[\phi^\epsilon_r(X_r)-\tilde\phi_r(X_r)]\nabla\varphi(X_r)]dr\Bigr]=0
\end{equation}
if one can choose $\tilde \phi$ so that the expectation 
\begin{equation}
    \label{fo:(ii)}
\EE^{\PP}\Bigl[ Y\int_s^t[\phi_r(X_r)-\tilde\phi_r(X_r)]\nabla\varphi(X_r)]dr\Bigr]
\end{equation}
can be made as small as desired. Prompted by the definition of $\phi^\epsilon_t$ given in \eqref{fo:gamma_t_epsilon}, we define $\tilde\phi^\epsilon_r(x)$ by
$$
\tilde\phi^\epsilon_r(x)=\frac{d\bigl((\tilde\phi_r\gamma_r)*\rho_\epsilon\bigr)}{d(\gamma_r*\rho_\epsilon)}.
$$
We first prove \eqref{fo:(i)}. Since $Y$ and $\nabla\varphi$ are bounded, we have
\begin{equation*}
\begin{split}
\EE^{\PP^\epsilon}\Bigl[ Y\int_s^t[\phi^\epsilon_r(X_r)-\tilde\phi_r(X_r)]\nabla\varphi(X_r)]dr\Bigr]   
&\le \|Y\|_\infty\|\nabla\varphi\|_\infty 
\EE^{\PP^\epsilon}\Bigl[\int_s^t|\phi^\epsilon_r(X_r)-\tilde\phi_r(X_r)|dr\Bigr]\\
&\le \|Y\|_\infty\|\nabla\varphi\|_\infty 
\int_0^T\int_{\RR^d}|\phi^\epsilon_r(x)-\tilde\phi_r(x)|\nu^\epsilon_r(dx)dr\\
&\le \|Y\|_\infty\|\nabla\varphi\|_\infty e^T
\int_0^T\int_{\RR^d}|\phi^\epsilon_r(x)-\tilde\phi_r(x)|\gamma^\epsilon_r(dx)dr.
\end{split}
\end{equation*}
We use the fact that 
\begin{equation*}
\begin{split}
\int_0^T\int_{\RR^d}|\phi^\epsilon_r(x)-\tilde\phi_r(x)|\gamma^\epsilon_r(dx)dr
&\le \int_0^T\int_{\RR^d}|\phi^\epsilon_r(x)-\tilde\phi^\epsilon_r(x)|\gamma^\epsilon_r(dx)dr+\int_0^T\int_{\RR^d}|\tilde\phi^\epsilon_r(x)-\tilde\phi_r(x)|\gamma^\epsilon_r(dx)dr\\
&=(i)+(ii)
\end{split}
\end{equation*}
and we estimate separately $(i)$ and $(ii)$.
\begin{equation*}
\begin{split}
(i)&=\int_0^T\int_{\RR^d}\Bigl|\frac{d\bigl((\phi_t\gamma_t)*\rho_\epsilon\bigr)}{d(\gamma_t*\rho_\epsilon)}-\frac{d\bigl((\tilde\phi_t\gamma_t)*\rho_\epsilon\bigr)}{d(\gamma_t*\rho_\epsilon)}\Bigr|\;d(\gamma_t*\rho_\epsilon) dt\\
&\le \int_0^T\int_{\RR^d}|\phi_t(x)-\tilde\phi_t(x)|\;\gamma_t(dx)dt
\end{split}
\end{equation*}
where we used \cite[Lemma 8.1.10]{Ambrosio_et_al}, and this quantity can be made as small as desired by density of the Lipschitz continuous functions with compact support in $L^1\bigl([0,T]\times\RR^d,\gamma_t(dx)dt\bigr)$. Next,
\begin{equation*}
\begin{split}
(ii)&=\int_0^T\int_{\RR^d}\Bigl|[(\tilde\phi_t(x)\gamma_t)*\rho_\epsilon] (x)-\tilde\phi_t(x)(\gamma_t*\rho_\epsilon)(x)\Bigr|dxdt\\
&= \int_0^T\int_{\RR^d}\Bigl|\int_{\RR^d}\rho_\epsilon(x-y)\tilde\phi_t(y)\gamma_t(dy)-\tilde\phi_t(x)\int_{\RR^d}\rho_\epsilon(x-y)\gamma_t(dy)\Bigr|dx dt\\
&\le \int_0^T\int_{\RR^d}\int_{\RR^d}\rho_\epsilon(x-y)|\tilde\phi_t(y)-\tilde\phi_t(x)|\gamma_t(dy)dx dt\\
&\le C\int_0^T\int_{\RR^d}\int_{\RR^d}\rho_\epsilon(x-y)(|x-y|\wedge1)\;\gamma_t(dy)dx dt\\
&= CT\int_{\RR^d}\rho_\epsilon(z)\,(|z|\wedge 1)\; dz
\end{split}
\end{equation*}
where we used Fubini's theorem, the Lipschitz property of $\tilde\phi$, and the fact that $\gamma_t(\RR^d)\le 1$, and this quantity converges to $0$ as $\epsilon\searrow 0$. This completes the proof since the above argument also proves \eqref{fo:(ii)}.
\end{proof}

We now state as a theorem the main result of this subsection.

\begin{theorem}[Non-local Superposition Principle]
\label{th:superposition}
Let us assume that the couple $(\bphi,\bmu)$ is such that  $\bphi=\bigl(\phi_t(x)\bigr)_{0\le t\le T,\;x\in\RR^d}$ is a $\RR^d$-valued measurable function on $[0,T]\times\RR^d$ and $\bmu=(\mu_t)_{0\le t\le T}$ is a  measurable flow  of probability measures satisfying \eqref{fo:compatibility} and the non-local FPK equation \eqref{fo:kolmo1} in the sense of distributions. 

Then there exists a weak solution $\bX=(X_t)_{0\le t\le T}$ of the stochastic differential equation \eqref{fo:state_SDE} satisfying \eqref{fo:compatibility}, and such that for each $t\in[0,T]$, the probability measure $\mu_t$ is given by \eqref{fo:mu_1_t}. Moreover,
\begin{equation}
    \label{fo:upper_bound}
    \sup_{0\le t\le T}\EE[|X_t|^2]<\infty.
\end{equation}
\end{theorem}
\begin{proof}
    The existence of $\bX$ follows the set of lemmas proven above. We only need to prove the bound \eqref{fo:upper_bound}. It\^o's formula gives:
    $$
    |X_t|^2=|X_0|^2+\int_0^t[2\phi_s(X_s)\cdot X_s+d]ds+2\int_0^tX_s\cdot dW_s
    $$
    and taking expectations we get
    $$
    \EE[|X_t|^2]\le\Bigl(\EE[|X_0|^2]+d\, T+\int_0^T\int_{\RR^d}|\phi_t(x)|^2d\nu_t(dx)dt\Bigr)+\int_0^t\EE[|X_s|^2]ds
    $$ 
    and we conclude using Gronwall's inequality.
\end{proof}

\subsection{Reformulation of the Optimization Problem}
\label{sec:reformulation-markovian-deterministic}
In this subsection we reformulate the relaxed optimization problem over stochastic state processes as a deterministic control problem on a space of probability measures.

For $\mu_0\in\cP_2(\RR^d)$, we now denote by $\AA=\AA(\mu_0)$ the set of couples  $(\theta,\mu)$ where $\mu\in\cM_+([0,T]\times \RR^d)$ is a non-negative measure on $[0,T]\times \RR^d$ of the form $\mu(dt,dx)=\mu_t(dx)dt$ for a measurable flow $\bmu=(\mu_t)_{0\le t\le T}$ of probability measures on $\RR^d$, and $\theta\in\cM([0,T]\times\RR^d;\RR^d)$ is an $\RR^d$-valued measure on $[0,T]\times \RR^d$ of the form $\theta(dt,dx)=\theta_t(dx)dt$ for a measurable flow $\btheta=(\theta_t)_{0\le t\le T}$ of $\RR^d$-valued measures on $\RR^d$, satisfying
\begin{equation}
\label{fo:transport}
\partial_t\mu_t-\frac{\sigma^2}{2}\Delta\mu_t+\div(\theta_t)+(V-<\mu_t,V>)\mu_t=0
\end{equation}
in the sense of distributions, and with initial condition $\mu_{| t=0}=\mu_0$. 

We denote by $\AA^{(2)}=\AA^{(2)}(\mu_0)$ the subset of $\AA(\mu_0)$ of couples $(\btheta,\bmu)$ for which $\theta_t$ is absolutely continuous with respect to $\mu_t$ for all $t\in[0,T]$, and for which there exists a measurable function $[0,T]\times\RR^d\ni (t,x)\mapsto\phi_t(x)\in\RR^d$ such that 
$$
\frac{d\theta_t}{d\mu_t}(x)=\phi_t(x)
\qquad\text{and}\qquad
\int_0^T\int_{\RR^d}\phi_t(x)^2\mu_t(dx)dt<\infty.
$$

    If $(\btheta,\bmu)\in\AA^{(2)}(\mu_0)$, Theorem \ref{th:superposition} implies that there exists a process $\bX=(X_t)_{0\le t\le T}$ satisfying $dX_t=\phi_t(X_t)dt + dW_t$, \eqref{fo:compatibility} and \eqref{fo:upper_bound}, and such that the probability measures $\mu_t$ are given by 
$$
\mu_t(dx)=\frac{\EE\bigl[\delta_{X_t}(dx)e^{-\int_0^t V(X_s)ds}\bigr]}{\EE\bigl[e^{-\int_0^t V(X_s)ds}\bigr]}.
$$  

\vskip 2pt
Later on, we shall use the analog $\AA^{(2)}_{[s,T]}$ of $\AA^{(2)}$ defined over the interval $[s,T]$ instead of $[0,T]$.
We now introduce the functional $J$ defined on $\AA$ by
 
\begin{equation}
\label{fo:J_of_theta_mu}
J(\btheta,\bmu)=
\begin{cases}\displaystyle
\int_0^T\int f(x,\phi_t(x))\mu_t(dx)\;dt+\int g(x)\mu_T(dx),\quad & \text{if } (\btheta,\bmu)\in\AA^{(2)}\\
\infty& \text{otherwise.}
\end{cases}
\end{equation}
When $(\btheta,\bmu)\in\AA^{(2)}$ we use the notations $J(\btheta,\bmu)$ and $J(\bphi,\bmu)$ interchangeably.

\subsection{Existence of an Optimal Control}
Next, we state and prove the existence of an optimal Markovian control.
\begin{proposition}
\label{pr:existence_control1}
    There exists a couple $(\btheta,\bmu)=(\phi_t,\mu_t)_{0\le t\le T}\in\AA^{(2)}$ minimizing $J(\bphi,\bmu)$ over $\AA$.
\end{proposition}

\begin{proof} 
The idea is to consider a minimizing sequence, and to show that it converges in a suitable sense to a minimizer of $J(\btheta,\bmu)$. $\AA^{(2)}$ being non-empty, 
$$
J^*:= \inf_{(\btheta,\bmu)\in\AA}J(\btheta,\bmu)<\infty
$$
and because of the definition \eqref{fo:J_of_theta_mu}, we can limit the search for a minimizer to $\AA^{(2)}$.

\vskip 2pt
\emph{Step 1}.  Let $(\bphi^n,\bmu^n)_{n\ge 0}$ be a
minimizing sequence in $\AA^{(2)}$. Since $\sup_n J(\bphi^n,\bmu^n)<\infty$, we have
\begin{equation}
\label{fo:C}
    C_\phi:=\sup_n\int_0^T\int_{\RR^d}|\phi^n_t(x)|^2\mu^n_t(dx)dt <\infty,
\end{equation}
from which we argue that, extracting a sub-sequence if necessary,  the sequence $(\theta^n)_{n\ge 0}$ of $\RR^d$-valued measures defined by
$$
\theta^n(dt,dx)=\phi^n_t(x)\mu^n_t(dx)dt
$$
converges weakly toward a measure $\theta(dt,dx)\in\cM([0,T]\times\RR^d;\RR^d)$. Moreover, and for similar reasons,  we may assume without any loss of generality, that $\bigl(\mu^n_t(dx)dt\bigr)_{n\ge 0}$ converges weakly toward a non-negative measure which is necessarily of the form $\mu_t(dx)dt$. The fact that the sequences $(\theta^n)_{n\ge 0}$ and $(\mu^n)_{n\ge 0}$ are tight despite the fact that the state space is not compact is a simple consequence of 
$$
\int_0^T\int_{\RR^d}|x|\;|\phi^n_t(x)|\mu^n_t(dx)dt\le\Bigl[\int_0^T\int_{\RR^d}|x|^2\mu^n_t(dx)dt\Bigr]^{1/2}
\Bigl[\int_0^T\int_{\RR^d}|\phi^n_t(x)|^2\mu^n_t(dx)dt\Bigr]^{1/2}<c
$$
for a finite constant $c$ independent of $n$. Indeed, using the test function $\varphi(x)=|x|^2$ in \eqref{fo:transport} we get:
\begin{equation}
\label{fo:mu_cont}
\begin{split}
    \partial_t <|x|^2,\mu^n_t>
    &=<\frac12 \Delta|x|^2 +\phi^n_t(x)\nabla|x|^2 -(V(x)-<V,\mu_t>)|x|^2,\mu^n_t>\\
    &\le d +2<x\cdot\phi^n_t(x),\mu^n_t>+<|x|^2,\mu^n_t>\\
    &\le d +4<|\phi^n_t(x)|^2,\mu^n_t>+5<|x|^2,\mu^n_t>\\
    &\le d +4C_\phi+5<|x|^2,\mu^n_t>
\end{split}
\end{equation}
and Gronwall's inequality gives
$$
\int_{\RR^d}|x|^2\mu^n_t(dx)\le e^{5T}\int_{\RR^d}|x|^2\mu_0(dx)+C',
$$
for a positive constant $C'$ independent of $n$.

\vskip 2pt
\emph{Step 2}. We now show that $\theta(dt,dx)$ is of the form $\theta(dt,dx)=\phi_t(x)\mu_t(dx)dt$ for some $\bphi=(\phi_t(x))_{t,x}\in L^2(\mu_t(dx)dt)$. If  $h\in C_b([0,T]\times\RR^d;\RR^d)$ is a bounded continuous function on $[0,T]\times\RR^d$ with values in $\RR^d$, we have: 
\begin{equation}
\label{fo:theta_cont}
\begin{split}
    |<\theta, h>|&=\lim_{n\to\infty}|<\theta^n,h>|\\
    &=\lim_{n\to\infty}\Bigl|\int_0^T\int_{\RR^d} h(t,x)\cdot \phi^n_t(x)\mu^n_t(dx)dt\Bigr|\\
    &\le \limsup_{n\to\infty}\Bigl(\int_0^T\int_{\RR^d}|h(t,x)|^2\mu^n_t(dx)dt\Bigr)^{1/2}\Bigl(\int_0^T \int_{\RR^d}|\phi^n_t(x)|^2\mu^n_t(dx)dt\Bigr)^{1/2}\;\\
    &\le C_\phi^{1/2}\limsup_{n\to\infty}\Bigl(\int_0^T\int_{\RR^d}|h(t,x)|^2\mu^n_t(dx)dt\Bigr)^{1/2} \\
    &= C_\phi^{1/2}\Bigl(\int_0^T\int_{\RR^d}|h(t,x)|^2\mu_t(dx)\;dt\Bigr)^{1/2}
\end{split}
\end{equation}
which shows that $\theta$ is a bounded linear form on the Hilbert space $L^2\bigl(\mu_t(dx)dt\bigr)$, proving the existence of 
$\bphi\in L^2(\mu_t(dx)dt)$ such that $\theta(dt,dx)=\phi_t(x)\mu_t(dx)dt$.

\vskip 2pt
\emph{Step 3}. For each integer $n\ge 0$, since $(\phi^n_t,\mu^n_t)_{0\le t\le T}\in \AA^{(2)}$, for each test function $(t,x)\mapsto \varphi(t,x)$ in $C^{1,2}_b([0,T]\times\RR^d)$, namely a smooth function with enough bounded derivatives so we can use integration by parts and push the derivatives from $\mu_t$ to $\varphi$, and hopefully not have boundary terms to deal with,  we have:
$$
\int_0^T\hskip -4pt\int_{\RR^d}\Bigl[\partial_t\varphi(t,x)+\frac{\sigma^2}{2}\Delta\varphi(t,x)-(V(x)-<V,\mu^n_t>)\varphi(t,x)\Bigr] \mu^n_t(dx)dt=-\int_0^T\hskip -4pt\int_{\RR^d}\partial_x\varphi(t,x)\cdot \phi^n_t(x)\,\mu^n_t(dx)dt.
$$
We can pass to the limit $n\to\infty$ using the convergence of $\mu^n_t(dx)dt$ in the left hand side and the convergence of  $\phi^n_t(x)\mu^n_t(dx)dt$ in the right hand side to conclude that $(\phi_t,\mu_t)_{0\le t\le T}\in\AA^{(2)}$.

\vskip 2pt
\emph{Step 4}. For the sake of convenience, we shall use the notation     
$$
\tilde J(\bpsi,\bnu)=\frac12\int_0^T\int_{\RR^d}\;|\psi_t(x)|^2\nu_t(dx)dt.
$$
For each $\epsilon>0$ and for each $t\in[0,T]$, we define $\theta^{\epsilon}_t=\theta_t*\rho_\epsilon$, and $\mu_t^{\epsilon}=\mu_t*\rho_\epsilon$ where $(\rho_\epsilon)_{\epsilon>0}$
is an approximate identity (say a Gaussian density in $\RR^{d}$ with variance $\epsilon$, $\rho_\epsilon(x) = (2 \pi x)^{-d/2} \exp(-|x|^2/2\epsilon)$), and where the operation of convolution is done component by component when appropriate. We then define $\phi^{\epsilon}_t(x)$ as the density of $\theta^{\epsilon}_t$ with respect to $\mu_t^\epsilon$. 
Using \cite[Lemma 8.1.10]{Ambrosio_et_al} we get 
\begin{equation}
    \label{fo:Ambrosio1}
    \tilde J(\bphi^\epsilon,\bmu^\epsilon)\le \tilde J(\bphi,\bmu)
\end{equation}
and since for each $t\in[0,T]$, $\mu^\epsilon_t$ and $\theta^\epsilon_t$ converge weakly toward $\mu_t$ and $\theta_t$ respectively, using the fact that the functional
$$
(\theta,\mu)\mapsto \int_{\RR^d}\Bigl|\frac{d\theta (x)}{d\mu (x)}\Bigr|^2 \mu(dx)
$$
is lower semi continuous (see for instance Theorem 2.34 and Example 2.36 in \cite{Ambrosio_1}), we conclude that: 
\begin{equation}
    \label{fo:Ambrosio2}
    \lim_{\epsilon\searrow 0}\tilde J(\bphi^\epsilon,\bmu^\epsilon) = \tilde J(\bphi,\bmu).
\end{equation}
Similarly, for each integer $n\ge 1$ we define $\theta^{n,\epsilon}_t=\theta^n_t*\rho_\epsilon$, and $\mu_t^{n,\epsilon}=\mu^n_t*\rho_\epsilon$, and $\phi^{n,\epsilon}_t$ as the density of $\theta^{n,\epsilon}_t$ with respect to $\mu_t^{n,\epsilon}$. Notice that for each $\epsilon>0$, $\bmu^{n,\epsilon}$ and $\btheta^{n,\epsilon}$ converge weakly toward $\bmu^{\epsilon}$ and $\btheta^{\epsilon}$ respectively. 
For the sake of notation we define $\phi^{(\epsilon,K)}_t(x)=(-K)\vee\phi^{\epsilon}_t(x)\wedge K$ where the operations of minimum and maximum are interpreted component by component. We have:
\begin{equation}
\label{fo:1}
\begin{split}
\tilde J(\bphi^\epsilon,\bmu^\epsilon)
&=\lim_{K\nearrow\infty}\int_0^T\int_{\RR^d}\phi^{(\epsilon,K)}_t(x)\cdot\phi^\epsilon_t(x)\mu^\epsilon_t(dx)dt\\
&=\lim_{K\nearrow\infty}\lim_{n\nearrow\infty}\int_0^T\int_{\RR^d}\phi^{(\epsilon,K)}_t(x)\cdot\phi^{n,\epsilon}_t(x)\mu^{n,\epsilon}_t(dx)dt\\
&\le\lim_{K\nearrow\infty}\lim_{n\nearrow\infty}\Bigl[\int_0^T\int_{\RR^d}|\phi^{(\epsilon,K)}_t(x)|^2\mu^{n,\epsilon}_t(dx)dt\Bigr]^{1/2}\Bigl[\int_0^T\int_{\RR^d}|\phi^{n,\epsilon}_t(x)|^2\mu^{n,\epsilon}_t(dx)dt\Bigr]^{1/2}\\
&=\lim_{K\nearrow\infty}\Bigl[\int_0^T\int_{\RR^d}|\phi^{(\epsilon,K)}_t(x)|^2\mu^\epsilon_t(dx)dt\Bigr]^{1/2}\lim_{n\nearrow\infty}\Bigl[\int_0^T\int_{\RR^d}|\phi^{n,\epsilon}_t(x)|^2\mu^{n,\epsilon}_t(dx)dt\Bigr]^{1/2}\\
&=\Bigl[\int_0^T\int_{\RR^d}|\phi^\epsilon_t(x)|^2\mu^\epsilon_t(dx)dt\Bigr]^{1/2}\lim_{n\nearrow\infty}\Bigl[\int_0^T\int_{\RR^d}|\phi^{n,\epsilon}_t(x)|^2\mu^{n,\epsilon}_t(dx)dt\Bigr]^{1/2}.
\end{split}
\end{equation}
Notice that
\begin{equation}
\label{fo:2}
\begin{split}
\lim_{n\nearrow\infty}\int_0^T\int_{\RR^d}|\phi^{n,\epsilon}_t(x)|^2\mu^{n,\epsilon}_t(dx)dt
&=\lim_{n\nearrow\infty}J(\btheta^{n,\epsilon},\bmu^{n,\epsilon})-\lim_{n\nearrow\infty}\int_0^T\int_{\RR^d}\tilde f(x)\mu^{n,\epsilon}_t(dx)dt - \lim_{n\nearrow\infty}\int_{\RR^d}g(x)\mu^{n,\epsilon}_T(dx)\\
&=\lim_{n\nearrow\infty}J(\btheta^{n,\epsilon},\bmu^{n,\epsilon})-\int_0^T\int_{\RR^d}\tilde f(x)\mu^{\epsilon}_t(dx)dt - \int_{\RR^d}g(x)\mu^{\epsilon}_T(dx)\\
\end{split}
\end{equation}
because we assume that $\tilde f$ and $g$ are bounded and continuous. 
Notice that \eqref{fo:2} implies that 
$$
\lim_{n\nearrow\infty}J(\btheta^{n,\epsilon},\bmu^{n,\epsilon})\ge \int_0^T\int_{\RR^d}\tilde f(x)\mu^\epsilon_t(dx)dt + \int_{\RR^d}g(x)\mu^\epsilon_T(dx)
$$
which in turn implies that, if $\tilde J(\btheta^\epsilon,\bmu^\epsilon)=0$,  
$$
\lim_{n\nearrow\infty}J(\btheta^{n,\epsilon},\bmu^{n,\epsilon})\ge J(\bphi^\epsilon,\bmu^\epsilon).
$$ 
On the other hand, if $\tilde J(\btheta^\epsilon,\bmu^\epsilon)>0$,
 \eqref{fo:1} implies
\begin{equation}
\begin{split}
\tilde J(\btheta^\epsilon,\bmu^\epsilon)
&\le \lim_{n\nearrow\infty}J(\btheta^{n,\epsilon},\bmu^{n,\epsilon})-\int_0^T\int_{\RR^d}\tilde f(x)\mu^\epsilon_t(dx)dt - \int_{\RR^d}g(x)\mu^\epsilon_T(dx)\\
&= \lim_{n\nearrow\infty}J(\btheta^{n,\epsilon},\bmu^{n,\epsilon}) + \tilde J(\btheta^\epsilon,\bmu^\epsilon) -J(\btheta^\epsilon,\bmu^\epsilon),
\end{split}
\end{equation}
implying that
$$
J(\btheta^\epsilon,\bmu^\epsilon)\le \lim_{n\nearrow\infty}J(\btheta^{n,\epsilon},\bmu^{n,\epsilon}) 
$$
holds in all cases. Now
\begin{equation}
\begin{split}
J(\btheta^{n,\epsilon},\bmu^{n,\epsilon})
& = \tilde J(\btheta^{n,\epsilon},\bmu^{n,\epsilon}) +\int_0^T\int_{\RR^d}\tilde f(x)\mu^{n,\epsilon}_t(dx)dt + \int_{\RR^d}g(x)\mu^{n,\epsilon}_T(dx)\\
&\le \tilde J(\btheta^{n},\bmu^{n}) +\int_0^T\int_{\RR^d}\tilde f(x)\mu^{n,\epsilon}_t(dx)dt + \int_{\RR^d}g(x)\mu^{n,\epsilon}_T(dx)\\
\end{split}
\end{equation}
if we use once more \eqref{fo:Ambrosio1} from \cite[Lemma 8.1.10]{Ambrosio_et_al}. Consequently:
\begin{equation}
\lim_{n\nearrow\infty} J(\btheta^{n,\epsilon},\bmu^{n,\epsilon})
\le \lim_{n\nearrow\infty}\tilde J(\btheta^{n},\bmu^{n}) +\int_0^T\int_{\RR^d}\tilde f(x)\mu^{\epsilon}_t(dx)dt + \int_{\RR^d}g(x)\mu^{\epsilon}_T(dx)\\
\end{equation}
and
\begin{equation}
\begin{split}
\lim_{\epsilon\searrow 0}\lim_{n\nearrow\infty} J(\btheta^{n,\epsilon},\bmu^{n,\epsilon})
& \le \lim_{n\nearrow\infty}\tilde J(\btheta^{n},\bmu^{n}) +\int_0^T\int_{\RR^d}\tilde f(x)\mu_t(dx)dt + \int_{\RR^d}g(x)\mu_T(dx)\\
& \le \lim_{n\nearrow\infty}J(\btheta^{n},\bmu^{n}) -\lim_{n\nearrow \infty}\int_0^T\int_{\RR^d}\tilde f(x)\mu^n_t(dx)dt -\lim_{n\nearrow \infty}\int_{\RR^d}g(x)\mu^n_T(dx)\\ 
&\hskip 95pt
+\int_0^T\int_{\RR^d}\tilde f(x)\mu_t(dx)dt + \int_{\RR^d}g(x)\mu_T(dx)\\
&= J^*,
\end{split}
\end{equation}
from which we deduce, using \eqref{fo:Ambrosio2}:
$$
J(\btheta,\bmu)=\lim_{\epsilon\searrow 0}J(\btheta^\epsilon,\bmu^\epsilon)\le \lim_{\epsilon\searrow 0}\lim_{n\nearrow \infty}J(\btheta^{n,\epsilon},\bmu^{n,\epsilon})\le J^*.
$$
 This completes the proof.
\end{proof}

\begin{remark}
\label{re:bound_on_phi}
The above argument shows that there is no loss of generality in limiting the search for optima to the subset of admissible feedback control functions satisfying 
\begin{equation}
    \label{fo:K_0_phi}
\EE\int_0^T|\phi_t(X_t)|^2\;dt\le K
\end{equation}
for a large enough constant $K>0$. Recall \eqref{fo:C} and the fact that the expectation is over a process $(X_t)_{0\le t\le T}$ satisfying the state dynamics \eqref{fo:phi_sde} driven by the control $\bphi$ and that 
$$
\EE\int_0^T|\phi_t(X_t)|^2\;dt\le e^T \int_0^T\frac{\EE\bigl[|\phi_t(X_t)|^2 e^{-\int_0^tV(X_s)ds}\bigr]}{\EE\bigl[ e^{-\int_0^tV(X_s)ds}\bigr]}=e^T\int_0^T\int_{\RR^d}|\phi_t(x)|^2\;\mu_t(dx)dt.
$$    
\end{remark}

\subsection{Solution of the Deterministic Control Problem}

In line with the computations of, and the notations used in the previous subsections, the running and terminal cost functions of the deterministic infinite dimensional control problem are defined as:
\begin{equation}
\label{fo:deterministic_costs}
F^{(1)}(\mu,\phi)= \int f(x,\phi(x))\; \mu(dx), \quad\text{and}\quad G^{(1)}(\mu)=\int g(x) \; \mu(dx),
\end{equation}
for $\mu$ and $\phi$ as above. So at east formally, the definition of the corresponding Hamiltonian $\HH^{(1)}$ should be:
\begin{equation}
\label{fo:deterministic_Hamiltonian}
\HH^{(1)}(\mu,\varphi,\phi)=<\frac12\Delta \mu - \div(\phi \mu) - (V-<\mu,V>) \mu,\;\varphi> + F^{(1)}(\mu,\phi)\\  
\end{equation}
where the bracket $<\,\cdot\,,\,\cdot\,>$ stands for the duality between measures and functions, and coincides with the inner product in $L^2(\RR^d,dx)$. After integration by parts of the first term, the definition of this Hamiltonian reads:
\begin{equation}
\label{fo:Hamiltonian1}
\HH^{(1)}(\mu,\varphi,\phi)=-\frac12<\nabla \mu, \nabla\varphi> + <\mu,\phi \nabla\varphi> - <\mu,V\varphi>+<\mu,V><\mu,\varphi>+F^{(1)}(\mu,\phi)
\end{equation}
which is well defined for $\mu\in\cP(\RR^d)$ as long as $\nabla\mu$ in the sense of distributions belongs to $L^2(\RR^d,dx;\RR^d)$, $\varphi\in L^2(\RR^d,dx)$ with a gradient (in the sense of distributions) belonging to $L^2(\RR^d,dx;\RR^d)$ and $L^2(\RR^d,\mu;\RR^d)$, and $\phi$ a $A$-valued measurable function on $\RR^d$ satisfying $\int|\phi(x)|^2\mu(dx)<\infty$.

\subsubsection{\textbf{The Adjoint PDE}}
\label{sub:adjoint2}

Let us assume that $(\bphi,\bmu)\in\AA^{(2)}$. We say that the function $u$ is an adjoint variable (or a co-state) if it satisfies  the PDE $\partial_t u=-(\delta \HH^{(1)}/\delta\mu)(\mu,u,\phi)$ with terminal condition $u_T(x)=(\delta G^{(1)}/\delta \mu)(x)= g(x)$ in the sense of distributions. Here the notation $\delta /\delta\mu$ stands for the \emph{flat derivative} which we now define. 
Referring to \cite[Definition 5.43]{CarmonaDelarue_book_I}  the flat derivative (also called the linear functional derivative) of a function $F:\cM(\RR^k) \to \RR$ of measures on $\RR^k$ for some integer $k$, is defined to satisfy:
\begin{equation}
\label{eq:def-linder}
    F(\mu')-F(\mu)
    =\int_0^1 \int \frac{\delta F}{\delta \mu}(\theta \mu' + (1-\theta)\mu)(x)\bigl[\mu'-\mu\bigr](dx) d \theta.
\end{equation}
Note that this notion of flat derivative is only defined up to a constant, but this will not matter in the present analysis.
Accordingly, the adjoint equation reads:
\begin{equation}
\label{fo:adjoint1}
    \partial_t u = -  \frac12\Delta_x u - \phi_t \cdot\nabla_x u + (V-<\mu,V>)u -V<\mu,u>- f\bigl(\cdot,\phi_t(\cdot)\bigr),
\end{equation}
which can be rewritten as
\begin{equation}
\label{fo:adjoint1'}
    0=\partial_t u + \frac12\Delta_x u+\phi_t\cdot\nabla_x u - (V-<\mu,V>)u + V<\mu,u> + \frac12|\phi_t|^2+\tilde f
\end{equation}
in the case of separable running cost functions of the form
\eqref{fo:separable_running_cost}. 

\vskip 6pt
The fact that from now on, we deal with Partial Differential Equations (PDEs) requires a strengthening of the assumptions made so far. In particular, assumption \eqref{fo:necessary} will be replaced by
\begin{assumption}
    \label{as:strong_necessary}
    \begin{equation}
        \label{fo:strong_necessary}
    K_\phi:=\sup_{(t,x)\in[0,T]\times\RR^d}\EE\int_t^T|\phi_s(X^{t,x}_s)|^2ds<\infty
    \end{equation}
    where $\bX^{t,x}=(X^{t,x}_s)_{t\le s\le T}$ satisfies the state equation $dX_s=\phi_s(X_s)ds+dW_s$ over the interval $[t,T]$ with initial condition $X_t=x$.
\end{assumption}
This assumption is obviously satisfied when $\phi$ is bounded. More generally
it is satisfied for larger classes of functions $\phi$. For example, \cite{FedrizziFlandoli} proves that Assumption \ref{as:strong_necessary} is satisfied if $\phi\in L^q\bigl([0,T];L^p(\RR^d)\bigr)$ for some
\begin{equation}
    \label{fo:pq}
    p\ge 2,\quad q>2,\quad \frac dp+\frac2q<1.
\end{equation}
Note also that earlier works like \cite{KrylovRockner} used a local version of assumption \ref{as:strong_necessary} in the sense that it is only required to be satisfied for $\phi\mathbf{1}_{|x|\le n}\in L^q\bigl([0,T];L^p(\RR^d)\bigr)$ for every integer $n\ge 1$.

\begin{proposition}
\label{pr:existence_adjoint1}
In the case of separable cost functions, for each continuous flow $(\mu_t)_{0\le t\le T}$ of probability measures on $\RR^d$ and each feedback function $\phi$ satisfying the strong assumption \ref{as:strong_necessary}, the adjoint PDE \eqref{fo:adjoint1'} admits a  solution in the sense of viscosity.
\end{proposition}

\begin{proof}
First notice that if $u$ is a classical solution of \eqref{fo:adjoint1}, the \emph{variation of the constant formula} and It\^o's formula give a form of Feynman-Kac implicit representation
\begin{equation}
\label{fo:FK1}
\begin{split}
&u_t(x)=\EE\Bigl[\int_t^Te^{-\int_t^r\bigl(V(X_\tau)-<\mu_\tau,V>\bigr)d\tau}\bigl[<\mu_r,u_r>V(X_r)+f\bigl(X_r,\phi_r(X_r)\bigr)\bigr]dr\\
&\hskip 155pt
+g(X_T)e^{-\int_t^T\bigl(V(X_\tau)-<\mu_\tau,V>\bigr)d\tau}
\;\Bigl|\;X_t=x\Bigr]
\end{split}
\end{equation}
where the expectation $\EE$ is over the process $(X_s)_{t\le s\le T}$ satisfying \eqref{fo:phi_sde} starting from $X_t=x$ at time $t$. The first step of the proof is to treat \eqref{fo:FK1} as a fixed point equation for the function $u$, and show that such an equation has a unique solution. This is identifying $u$ as a \emph{mild solution} of equation \eqref{fo:adjoint1}. Next, we argue that the latter is in fact the unique solution of an affine BSDE, and we conclude that it is a viscosity solution of \eqref{fo:adjoint1} using a classical result of the theory of BSDEs.
\vskip 1pt
\emph{Step 1}. For each $u\in C_b([0,T]\times\RR^d)$ where $C_b([0,T]\times\RR^d)$ denotes the Banach space of bounded continuous functions on $[0,T]\times\RR^d$, we denote by $\Phi(u)$ the right hand side of \eqref{fo:FK1} seen as a function of $(t,x)$. Using the fact that $f(x,\alpha)=\frac12|\alpha|^2+\tilde f(x)$ with $\tilde f$ bounded, assumption \ref{as:strong_necessary} implies that $\EE[\int_t^T|\phi_r(X_r)|^2]dr|X_t=x]$ is uniformly bounded in $t$ and $x$, and it is straightforward to check that  $\Phi(u)\in C_b([0,T]\times\RR^d)$, which we equip with the norm
$$
    \|u\|_\alpha = \sup_{(t,x)\in[0,T]\times\RR^d} e^{\alpha t}|u_t(x)|
$$
for a fixed number $\alpha>1$. If $u^1$ and $u^2$ are two functions in this Banach space,
\begin{equation}
\begin{split}
    \|\Phi(u^{1})-\Phi(u^{2})\|_\alpha
    &=\sup_{t,x} e^{\alpha t}\Bigl|
    \EE\Bigl[\int_t^Te^{-\int_t^r\bigl(V(X_\tau)-<\mu_\tau,V>\bigr)d\tau}|<\mu_r,u^1_r-u^2_r>|V(X_r)dr\;\Bigl|\;X_t=x\Bigr]\Bigr|\\
    &\le\sup_{t,x}  e^{\alpha t}\int_t^Te^{r-t}\sup_{y\in\RR^d}|u^1_r(y)-u^2_r(y)|dr\\
    &\le\sup_t  e^{\alpha t}\int_t^Te^{(1-\alpha)r-t} e^{\alpha r} \sup_{x}|u^1_r(x)-u^2_r(x)|dr\\
    &\le\|u^1-u^2\|_\alpha\sup_t  e^{(\alpha-1) t}\int_t^Te^{(1-\alpha) r} dr\\
    &=\frac{1-e^{(1-\alpha) T}}{\alpha-1}\|u^1-u^2\|_\alpha,
\end{split}
\end{equation}
and the fraction is smaller than $1$ when $\alpha$ is large enough, because $e^{(1-\alpha) T} \to 0$ and $1/(\alpha-1) \to 0$ as $\alpha\to+\infty$ which proves that $\Phi$ is a strict contraction since $\alpha>e^T$.
Its unique fixed point is what we shall use as solution of the adjoint equation \eqref{fo:adjoint1}.
\vskip 1pt
\emph{Step 2}. In order to match the notation of \cite[Remark 3.3]{PardouxNATO} we set $c(t,x)=-[V(x)-<\mu_t,V>]$ and 
$h(t,x) = <\mu_t,u_t>V(x)+f\bigl(x,\phi_t(x)\bigr)$. Notice that the function $c$ is bounded and since the function $u$ constructed above as a fixed point is bounded, for each $(t,x)\in[0,T]\times \RR^d$ we have:
$$
\EE\Bigl[\int_t^T|h(s,X^{t,x}_s)|ds\Bigr]\le C + \EE\Bigl[\int_t^T|\phi_s(X^{t,x}_s)|^2dt\Bigr]\le C+K_\phi<\infty
$$
 where $(X^{t,x}_s)_{t\le s\le T}$ is the solution of $dX_s=\phi_s(X_s)ds+dW_s$ over the interval $[t,T]$ and initial condition $X_t=x$. For each $(t,x)$ we denote by $(Y^{t,x}_s,Z^{t,x}_s)_{t\le s\le T}$ the unique solution of the affine BSDE:
$$
    dY^{t,x}_s=-[c(s,X^{t,x}_s) Y^{t,x}_s + h(s,X^{t,x}_s)]ds + Z^{t,x}_s dW_s
$$
with terminal condition $Y^{t,x}_T=g(X^{t,x}_T)$.  Being \emph{linear}, this BSDE has an explicit solution
\begin{equation}
Y^{t,x}_s
=g(X^{t,x}_T)e^{\int_s^T c(\tau,X^{t,x}_\tau)d\tau} +\int_s^Th(r,X^{t,x}_r)e^{\int_s^r c(\tau,X^{t,x}_\tau)d\tau}dr
-\int_s^T e^{\int_s^r c(\tau,X^{t,x}_\tau)d\tau} Z^{t,x}_rdW_r
\end{equation}
and taking $s=t$, one recovers the Feynman-Kac formula 
$$
Y^{t,x}_t=\EE\Bigl[g(X^{t,x}_T)e^{\int_t^T c(\tau,X^{t,x}_\tau)d\tau} +\int_t^Th(r,X^{t,x}_r)e^{\int_t^r c(\tau,X^{t,x}_\tau)d\tau}dr\\
\Bigr]
$$
which is exactly our formula \eqref{fo:FK1} if we set $u(t,x)=Y^{t,x}_t$. Strictly speaking, since $f(x,\phi_t(x))=\tilde f(x)+\frac12|\phi_t(x)|^2$, the assumptions of \cite{PardouxNATO} would require that $|\phi_t(x)|$ be of polynomial growth. However, given the special (linear) nature of the BSDE involved, the fact that $\int_t^T\EE[\phi_t(X^{t,x}_s)|^2ds<\infty$ for every $t\in[0,T]$ and $x\in\RR^d$ is enough for the argument of \cite{PardouxNATO} to go through. We conclude using 
\cite[Theorem 3.2]{PardouxNATO}. 
\end{proof}

\begin{remark}[\textbf{A first a-priori bound.}]
\label{re:weak_apriori_bound}
We claim that under the weaker assumption \ref{fo:K_0_phi}, if $u$ admits the Feynman-Kac representation \eqref{fo:FK1}, then
\begin{equation}
    \label{fo:weak_upper_bound}
|<\mu_t,u_t>|\le e^{2(T-t)}\bigl(\|\tilde f\|_\infty+\|g\|_\infty\bigr)+\frac{e^T}4(e^{2(T-t)}+1).
\end{equation}
Recall that we denote by $\nu_t$ the marginal distribution of the state $X_t$ controlled by the drift $\phi_t$. So by definition of the probability measure $\mu_t$, for any non-negative random variable $Y$ which depend only upon the future after time $t$, we have
$$
\int_{\RR^d}\EE[Y|X_t=x]\mu_t(dx)\le e^T\int_{\RR^d}\EE[Y|X_t=x]\nu_t(dx)=e^T\EE[Y]
$$
where the last expectation is with respect to the state process with initial distribution $\mu_0$.
So integrating both sides of the Feynman-Kac representation \eqref{fo:FK1} with respect to $\mu_t$ we get:
\begin{equation}
\label{fo:integrated_FK}
\begin{split}
&<\mu_t,u_t>=\int_{\RR^d}\EE\Bigl[\int_t^Te^{-\int_t^r\bigl(V(X_\tau)-<\mu_\tau,V>\bigr)d\tau}\bigl[<\mu_r,u_r>V(X_r)+f\bigl(X_r,\phi_r(X_r)\bigr)\bigr]dr\\
&\hskip 155pt
+g(X_T)e^{-\int_t^T\bigl(V(X_\tau)-<\mu_\tau,V>\bigr)d\tau}
\;\Bigl|\;X_t=x\Bigr]\mu_t(dx)
\end{split}
\end{equation}
and  if we limit ourselves to feedback functions satisfying \eqref{fo:K_0_phi}, we get:
\begin{equation*}
    \begin{split}
        |<\mu_t,&u_t>| \\
        &\le \int_t^Te^{r-t}|<\mu_r,u_r>|\, dr +e^{T-t}\|\tilde f\|_\infty+ \frac{e^T}2\int_{\RR^d}\EE[\int_t^T|\phi_s(X_s)|^2]ds\,|X_t=x]\nu_t(dx) +e^{T-t}\|g\|_\infty\\
        &\le e^{T-t}\bigl(\|\tilde f\|_\infty+\|g\|_\infty\bigr) +\frac{e^T}2 K_\phi +\int_t^Te^{r-t}|<\mu_r,u_r>| dr
    \end{split}
\end{equation*}
and using the form of Gronwall inequality in Lemma \ref{le:my_Gronwall} we get:
\begin{equation*}
e^t|<\mu_t,u_t>|\le e^{2T-t}\bigl(\|\tilde f\|_\infty+\|g\|_\infty\bigr)+\frac{e^T}4(e^{2T-t}+1).
\end{equation*}
\end{remark}

\begin{remark}[\textbf{A stronger a-priori bound.}]
\label{re:stronger_apriori_bound}
Under the assumptions of Proposition \ref{pr:existence_adjoint1}, we have:
\begin{equation}
    \label{fo:strong_upper_bound}
\|u_t\|_\infty\le \frac{e^t}{2}\bigl(e^{2(T-t)}-1\bigr)\|\tilde f\|_\infty+\frac{e^{2T}}{2}K_\phi+e^T\|g\|_\infty.
\end{equation}
Indeed, computing the supremum of the left hand side of the Feynman-Kac representation \eqref{fo:FK1} instead of integrating with respect to $\mu_t$, we get:
\begin{equation*}
    \begin{split}
       |u_t(x)| & \le \int_t^Te^{r-t}\bigl(\|u_r\|_\infty+\|\tilde f\|_\infty\bigr)\;dr+\int_t^T\EE\bigl[|\phi_r(X_r)|^2\,|\,X_t=x\bigr]  +e^{T-t}\|g\|_\infty
    \end{split}
\end{equation*}
and 
$$
e^t\|u_t\|_\infty\le\int_t^Te^r\|u_r\|_\infty dr + (e^T-e^t)\|\tilde f\|_\infty  +e^tK_\phi+e^T\|g\|_\infty.
$$
Finally,  we get \eqref{fo:smooth_apriori_bound_1} using Lemma \ref{le:my_Gronwall} with $\zeta(t)=e^T \|g\|_\infty+e^tK_\phi +(e^T-e^t)\|\tilde f\|_\infty $, $c=1$ and $\xi(t)=e^t\|u_t\|_\infty$.

\end{remark}

\subsubsection{\textbf{Analysis of the Adjoint PDE}}
\label{sub:adjoint2}

In this subsection, we assume that $(\bphi,\bmu)\in\AA^{(2)}(\mu_0)$. In particular, $\bmu$ solves the Fokker-Planck-Kolmogorov equation driven by $\bphi$, and we denote by $\bnu=(\nu_t)_{0\le t\le T}$ the marginal distributions of the state process $\bX=(X_t)_{0\le t\le T}$ whose existence is guarateed by our form of the superposition principle proven in Theorem \ref{th:superposition}.

\vskip 4pt
For the sake of later reference, we provide without proof, a simple version of Gronwall's lemma which we shall use repeatedly in the sequel.

\begin{lemma}
    \label{le:my_Gronwall}
Let us assume that $\xi(t)\le \zeta(t)+c\int_t^T\xi(r)dr$, then
    \begin{equation}
        \label{fo:my_Gronwall}
        \xi(t)\le \zeta(t)+c e^{c(T-t)}\int_t^T\zeta(r)e^{-c(T-r)}dr.
    \end{equation}
\end{lemma}

\begin{lemma}
    \label{le:smooth_adjoint}
    If $\tilde{\bphi}=(\tilde\phi_t)_{0\le t\le T}\in C_b([0,T]\times\RR^d;\RR^d)$ and $\tilde F=(\tilde F_t)_{0\le t\le T}\in C_b([0,T]\times\RR^d)$ are smooth enough, then the equation
    \begin{equation}
        \label{fo:smooth_adjoint}
        0=\partial_t v_t +\frac12\Delta v_t+\tilde\phi_t\cdot\nabla v_t -(V-<\mu_t,V>)v_t +V<\mu_t,v_t>+\tilde F
    \end{equation}
    with bounded terminal condition $v_T=g$, has a unique classical solution satisfying the upper bound
    \begin{equation}
        \label{fo:smooth_apriori_bound_1}
    \|v_t\|_\infty\le e^{2T-t}\Bigl(\|g\|_\infty +\frac12\|\tilde F\|_\infty\Bigr). 
\end{equation}
Moreover, when $g=0$, this solution satisfies:
    \begin{equation}
        \label{fo:smooth_apriori_bound_3}
        \int_0^T\int_{\RR^d}|\nabla v_t(x)|^2\mu_t(dx)dt\le 4e^{4T}\|\tilde F\|_\infty^2\bigl(1+\frac32T+\|\tilde \bphi-\bphi\|^2_{L^2(\bmu)}\bigr).
    \end{equation}
\end{lemma}

\begin{proof}
    Equation \eqref{fo:smooth_adjoint} is not a standard PDE because of the presence of the non-local term $V<\mu_t,v_t>$ in the right hand side. The solution is obtained  by constructing a fixed point to the mapping $\tilde{\bv}\mapsto \Phi(\tilde{\bv})=\bv$ where $\bv$ is the unique classical solution of the regular linear PDE
    $$
    0=\partial_t v_t +\frac12\Delta v_t+\tilde\phi_t\cdot\nabla v_t -(V-<\mu_t,V>)v_t +V<\mu_t,\tilde v_t>+\tilde F,
    $$
with terminal condition $v_T = g$. From now on, we let $\bv$ be a solution to~\eqref{fo:smooth_adjoint}. 
The first a-priori bound \eqref{fo:smooth_apriori_bound_1} follows easily from the Feynman-Kac representation in terms of the solution $\tilde\bX=(\tilde X_t)_{0\le t\le T}$ of the state stochastic differential equation $d\tilde X_t=\tilde\phi_t(\tilde X_t)dt + dW_t$. Indeed, $\tilde\phi$ being bounded, it satisfies Assumption~\ref{as:strong_necessary} and we can repeat the argument of Remark \ref{re:stronger_apriori_bound}: from~\eqref{fo:smooth_adjoint} we get: 
\begin{equation*}
    \begin{split}
        v_t(x) & =\EE\Bigl[\int_t^Te^{-\int_t^r\bigl(V(\tilde X_\tau)-<\mu_\tau,V>\bigr)d\tau}\bigl[<\mu_r, v_r>V(\tilde X_r)+\tilde F(r,\tilde X_r)\bigr]dr\\
&\hskip 155pt
+g(\tilde X_T)e^{-\int_t^T\bigl(V(\tilde X_\tau)-<\mu_\tau,V>\bigr)d\tau}
\;\Bigl|\;\tilde X_t=x\Bigr]
    \end{split}
\end{equation*}
from which we get
\begin{equation*}
    \begin{split}
       | v_t(x)| & \le \int_t^Te^{r-t}\bigl(\|v_r\|_\infty+\|\tilde F\|_\infty\bigr)\;dr  +e^{T-t}\|g\|_\infty
    \end{split}
\end{equation*}
and 
$$
e^t\|v_t\|_\infty\le\int_t^Te^r\|v_r\|_\infty dr + (e^T-e^t)\|\tilde F\|_\infty  +e^{T-t}\|g\|_\infty.
$$
Finally,  we get \eqref{fo:smooth_apriori_bound_1} using Lemma \ref{le:my_Gronwall} with $\zeta(t)=e^{T-t} \|g\|_\infty +(e^T-e^t)\|\tilde F\|_\infty $, $c=1$ and $\xi(t)=e^t\|v_t\|_\infty$.
\vskip 2pt
The derivation of the a-priori bound
\eqref{fo:smooth_apriori_bound_3} is more involved.
We notice that since $\bmu$ solves the Fokker-Planck-Kolmogorov equation driven by $\bphi$, a simple integration by parts implies that if $\varphi$ is a smooth bounded function on $[0,T]\times\RR^d$ we must have:
\begin{equation}
\label{fo:varphi2}
\int_0^T\int_{\RR^d}\Bigl[
\partial_t\varphi^2+\frac12\Delta_x\varphi^2+\phi_t\cdot\nabla_x\varphi^2-(V-<\mu_t,V>)\varphi^2
\Bigr]\mu_t(dx)dt
=<\mu_T,\varphi^2_T>-<\mu_0,\varphi^2_0>.
\end{equation}
as long as the above integrals exist. Using \eqref{fo:smooth_adjoint} we get:
\begin{equation*}
\begin{split}
&\int_0^T\int_{\RR^d}v_t(x)\bigl(\tilde F_t(x)+\tilde\phi_t\cdot\nabla_xv_t(x)\bigr)\mu_t(dx)dt\\
&\hskip 15pt
=\int_0^T\int_{\RR^d}v_t(x)\Bigl(-\partial_tv_t(x)-\frac12\Delta_xv_t(x)+(V(x)-<\mu_t,V>)v_t(x)-V(x)<\mu_t,v_t>\Bigr)\mu_t(dx)dt\\
&\hskip 15pt
=-\frac12\int_0^T\int_{\RR^d}\Bigl(\partial_tv^2_t(x)+\frac12\Delta_xv^2_t(x)
+\phi_t(x)\cdot\nabla_xv^2_t(x)-(V(x)-<\mu_t,V>)v^2_t(x)\Bigr)\mu_t(dx)dt\\
&\hskip 55pt
+\int_0^T\int_{\RR^d}\Bigl(\frac12(V(x)-<\mu_t,V>)v^2_t(x)-V(x)<\mu_t,v_t>v_t(x)\\
&\hskip 105pt
+\frac12|\nabla_xv_t(x)|^2+\frac12\phi_t(x)\cdot\nabla_xv^2_t(x)
\Bigr)\mu_t(dx)dt\\
&\hskip 15pt
=\frac12 <v^2_0,\mu_0>
+\int_0^T\int_{\RR^d}\Bigl(\frac12(V(x)-<\mu_t,V>)v^2_t(x)-V(x)<\mu_t,v_t>v_t(x)\\
&\hskip 105pt
+\frac12|\nabla_xv_t(x)|^2+\frac12\phi_t(x)\cdot\nabla_xv^2_t(x)\Bigr)\mu_t(dx)dt
\end{split}
\end{equation*}
where we use \eqref{fo:varphi2} with $\varphi=v$ and the fact that $v_T=0$. Consequently: 
\begin{equation*}
    \begin{split}
&\int_0^T\int_{\RR^d}|\nabla_xv_t(x)|^2\mu_t(dx)dt\\
&
\le2\int_0^T\int_{\RR^d}v_t(x)\bigl(\tilde F_t(x)+\tilde\phi_t\cdot\nabla_xv_t(x)\bigr)\mu_t(dx)dt\\
&\hskip 15pt
-\int_0^T\int_{\RR^d}\Bigl((V(x)-<\mu_t,V>)v^2_t(x)+2V(x)<\mu_t,v_t>v_t(x)
-\phi_t(x)\cdot\nabla_xv^2_t(x)\Bigr)\mu_t(dx)dt\\
&
= 2\int_0^T\int_{\RR^d}v_t(x)\tilde F_t(x)\mu_t(dx)dt
+2\int_0^T\int_{\RR^d}v_t(x)\bigl(\tilde\phi_t(x)-\phi_t(x)\bigr)\cdot\nabla_xv_t(x)\mu_t(dx)dt\\
&\hskip 15pt
-\int_0^T\int_{\RR^d}\Bigl((V(x)-<\mu_t,V>)v^2_t(x)+2V(x)<\mu_t,v_t>v_t(x)\Bigr)\mu_t(dx)dt\\
&
\le 2\|v\|_\infty\|\tilde F\|_{L^1(\bmu)}
+2\int_0^T\int_{\RR^d}v^2_t(x)\bigl|\tilde\phi_t(x)-\phi_t(x)\bigr|^2\mu_t(dx)dt
+\frac12\int_0^T\int_{\RR^d}|\nabla_xv_t(x)|^2\mu_t(dx)dt+3T\|v\|_\infty^2,
\end{split}
\end{equation*}
where we used the inequality $ab\le a^2+b^2/4$. Consequently
$$
\int_0^T\int_{\RR^d}|\nabla_xv_t(x)|^2\mu_t(dx)dt
\le 4\|v\|_\infty\|\tilde F\|_{L^1(\bmu)}+6T\|v\|^2_\infty+4 \|v\|^2_\infty\|\tilde{\bphi}-\bphi\|^2_{L^2(\bmu)}
$$
and we conclude using the a-priori bound \eqref{fo:smooth_apriori_bound_1} and the fact that integrals with respect to $\bnu$ are controlled by integrals with respect to $\bmu$.
\end{proof}

\begin{remark}
    Remark \ref{re:stronger_apriori_bound} implies that the proofs of the above a-priori bounds do not really need that $\tilde\phi$ is bounded, but merely that it satisfies assumption \ref{as:strong_necessary}. We stated with the boundedness assumption for the sake of simplicity.
\end{remark}

We now tackle the issue of existence, uniqueness and regularity of solutions of the adjoint equation.

\begin{theorem}
\label{th:adjoint_solution}
If $(\bphi,\bmu)\in\AA^{(2)}(\mu_0)$ is such that $\bphi$ is bounded, the viscosity solution of the adjoint equation is a bounded continuous function on $\RR^d$ whose first order derivatives in $x\in\RR^d$  in the sense of distributions are functions in $L^2([0,T]\times\RR^d,\bmu)$ and  $L_{loc}^2([0,T]\times\RR^d,dt\,dx)$.
\end{theorem}

\begin{proof}
    Let us introduce the family $(\bphi^\epsilon)_{\epsilon>0}$ of bounded smooth Markovian feedback functions defined by $\phi_t^\epsilon(x)=[\phi_t*\rho^\epsilon](x)$ where the convolution is performed component by component, and where $(\rho^\epsilon)_{0<\epsilon\le 1}$ is an approximate identity, the support of $\rho^1$ being included in the unit ball of $\RR^d$. For every $p\in[1,\infty)$, the $C^\infty$ bounded functions $\bphi^\epsilon$  converges to $\bphi$ in $L^p_{loc}([0,T]\times\RR^d,dx\,dt)$ when $\epsilon\searrow 0$, and for each $\epsilon>0$, $\|\bphi^\epsilon\|_\infty\le \|\bphi\|_\infty$. For each $\epsilon>0$ we set $F^\epsilon_t(x)=\frac12|\phi^\epsilon_t(x)|^2+\tilde f(x)$, we denote by $w^\epsilon$ the classical solution of the PDE \eqref{fo:smooth_adjoint} for $\tilde\phi_t(x)=\phi^\epsilon_t(x)$ and $\tilde F_t(x)=F^\epsilon_t(x)$, and we set $v^\epsilon_t(x)=w^\epsilon_t(x)\chi^\epsilon(x)$ where $\chi^\epsilon$ is $C^\infty$, $\chi^\epsilon(x)=1$ if $|x|\le 1/\epsilon$ and $\chi^\epsilon(x)=0$ if $|x|\ge 1+1/\epsilon$. We have: 
    $$
    0=\partial_t w^\epsilon_t(x) +\frac12\Delta w^\epsilon_t(x)+\phi^\epsilon_t(x)\cdot\nabla w^\epsilon_t(x)
    -(V(x)-<\mu_t,V>)w^\epsilon_t(x) +V(x)<\mu_t, w^\epsilon_t>+F^\epsilon_t(x)
    $$
 and we can replace $w^\epsilon_t(x)$ by $v^\epsilon_t(x)$ whenever $|x|\le 1/\epsilon$.   
    
    \vskip 4pt    
    Applying the a-priori estimates \eqref{fo:smooth_apriori_bound_1} to $w^\epsilon$
    we get that for each compact set $K\subset\RR^d$ we have 
    $$
    \limsup_{\epsilon\searrow 0}\|v^\epsilon\|_{L^2([0,T]\times K,dt\,dx)}
    \le \limsup_{\epsilon\searrow 0}\|w^\epsilon\|_{L^2([0,T]\times K,dt\,dx)}<\infty
    $$
    giving the existence of $v\in L^2_{loc}([0,T]\times\RR^d, dt\,dx)$ for which, after extracting a sub-sequence if needed,
    $$
    v=\lim_{\epsilon\searrow 0} v^\epsilon
    $$
locally (i.e. for each compact subset $K$) for the weak topology. Also, the a-priori estimates \eqref{fo:smooth_apriori_bound_1} implies that the $v^\epsilon$'s form a bounded set in $L^2(\bmu)$, so extracting a further sub-sequence if necessary, one can assume that $v^\epsilon$ converges weakly toward $v$ in $L^2(\bmu)$ implying that, extracting a further subsequence if needed, for almost every $t\in[0,T]$:
\begin{equation}
\label{fo:mu_tv}
    \lim_{\epsilon\searrow 0}<\mu_t,w^\epsilon_t>=\lim_{\epsilon\searrow 0}<\mu_t,v^\epsilon_t>=\lim_{\epsilon\searrow 0}<1,v^\epsilon_t>_{L^2(\mu_t)}=<1,v_t>_{L^2(\mu_t)}=<\mu_t,v_t>.
\end{equation}
Furthermore, the a-priori estimate \eqref{fo:smooth_apriori_bound_3} implies that one can extract a further sub-sequence for which $\nabla_x v^\epsilon$ converges weakly in $L^2([0,T]\times\RR^d,\bmu;\RR^d)$ toward a $\RR^d$-valued function $\tilde\varphi$ in $L^2([0,T]\times \RR^d,\bmu;\RR^d)$. 
This measurable vector field $\tilde \varphi$ can be identified with the gradient (in the sense of distributions) of $v$. Indeed, if $\varphi$ is a smooth test function with compact support in $(0,T)\times\RR^d$
$$
\lim_{\epsilon\searrow 0} \int_0^T\int_{\RR^d} \nabla_xv^\epsilon_t(x) \varphi_t(x)dxdt=-\lim_{\epsilon\searrow 0} \int_0^T\int_{\RR^d} \nabla_x\varphi_t(x)v^\epsilon_t(x) dxdt=-\int_0^T\int_{\RR^d}\nabla_x\varphi_t(x) v_t(x)\; dxdt
$$
and since $\bphi$ is assumed to be bounded
$$
\lim_{\epsilon\searrow 0} \int_0^T\int_{\RR^d} \nabla_xv^\epsilon_t(x) \varphi_t(x)dx=\lim_{\epsilon\searrow 0} \int_0^T\int_{\RR^d} \nabla_x v^\epsilon_t(x)\frac{\varphi_t(x)}{\rho_t(x)}\mu_t(dx) =\int_0^T\int_{\RR^d}\tilde\varphi_t(x) \varphi_t(x)\; dx
$$
where we used part (iii) of Proposition \ref{pr:densities} to benefit from the fact that the continuous density $\rho$ is locally strictly positive and bounded below away from $0$. This shows that the gradient in the sense of distributions of $v$ is a function (namely the function $\tilde \varphi$), and that $\nabla_x v^\epsilon$ converges weakly in $L_{loc}^2([0,T]\times\RR^d,dx\,dt;\RR^d)$ toward $\tilde \varphi=\nabla_x v$. This implies that 
\begin{equation}
    \label{fo:limit0}
\lim_{\epsilon\searrow 0}\int_0^T\int_{\RR^d}\varphi_t(x)\;\phi^\epsilon_t(x)\cdot\nabla w^\epsilon_t(x)\;dx dt=
\int_0^T\int_{\RR^d}\varphi_t(x)\;\phi_t(x)\cdot\nabla v_t(x)\;dx dt
\end{equation}
where $\nabla v_t(x)$ is the function $\tilde\varphi_t(x)$ identified as the gradient in the sense of distributions of the function $v$ constructed above. Indeed, if $\epsilon>0$ is small enough, 
\begin{equation*}
    \label{fo:limit0}
    \begin{split}
    &    \int_0^T\int_{\RR^d}\varphi_t(x)\;\phi^\epsilon_t(x)\cdot\nabla w^\epsilon_t(x)\;dx dt
    -\int_0^T\int_{\RR^d}\varphi_t(x)\;\phi_t(x)\cdot\nabla v_t(x)\;dx dt\\
    &\hskip 45pt
    =
\int_0^T\int_{\RR^d}\varphi_t(x)\;\bigl(\phi^\epsilon_t(x)-\phi_t(x)\bigr)\cdot\nabla v^\epsilon_t(x)\;dx dt
+\int_0^T\int_{\RR^d}\varphi_t(x)\;\phi_t(x)\cdot\bigl(\nabla v^\epsilon_t(x)-\nabla v_t(x)\bigr)\;dx dt.
    \end{split}
\end{equation*}
The first term goes to $0$ because
$$
\limsup_{\epsilon\searrow 0}\int_0^T\int_{\RR^d}|\varphi_t(x)|\;|\nabla v^\epsilon_t(x)|^2\;dx dt<\infty
$$
since the $L^2(\bmu)$-norm of $\nabla_x v^\epsilon$ is uniformly bounded because of the a-priori estimate \eqref{fo:smooth_apriori_bound_3} and the fact that the density of $\mu_t$ is bounded from below away from $0$ on the support of $\varphi$, and the fact that $\phi^\epsilon$ converges toward $\phi$ in $L^2_{loc}(dxdt)$. As for the second term, it also converges toward $0$ because $\nabla_x v^\epsilon$ converges toward $\tilde \varphi=\nabla_x v$
weakly in $L^2_{loc}(dxdt)$.
Again, if $\varphi$ is a smooth test function with compact support in $(0,T)\times\RR^d$, we have
\begin{equation}
    \label{fo:limit1}
    \begin{split}
&  \lim_{\epsilon\searrow 0}\int_0^T\int_{\RR^d}\Bigl[-\partial_t\varphi_t(x)+\frac12\Delta \varphi_t(x) -(V(x)-<\mu_t,V>)\varphi_t(x)\Bigr] v^\epsilon_t(x)dxdt
\\
&\hskip 35pt
=\int_0^T\int_{\RR^d}\Bigl[-\partial_t\varphi_t(x)+\frac12\Delta \varphi_t(x) -(V(x)-<\mu_t,V>)\varphi_t(x)\Bigr] v_t(x)dxdt.
    \end{split}
\end{equation}
Moreover
\begin{equation}
    \label{fo:limit2}
\lim_{\epsilon\searrow 0}\int_0^T\int_{\RR^d}\varphi_t(x)|\phi^\epsilon_t(x)|^2dxdt=
\int_0^T\int_{\RR^d}\varphi_t(x)|\phi_t(x)|^2dxdt.
\end{equation}
Putting together the limits \eqref{fo:mu_tv}, \eqref{fo:limit0}, \eqref{fo:limit1} and \eqref{fo:limit2} we find that 
\begin{equation}
\begin{split}
&    0=\int_0^T\int_{\RR^d}\Bigl(\bigl[-\partial_t\varphi_t(x)+\frac12\Delta \varphi_t(x) -(V(x)-<\mu_t,V>)\varphi_t(x)\bigr] v_t(x)+\varphi_t(x)\phi_t(x)\cdot\nabla_xv_t(x)\\
&\hskip 95pt
+\bigl[<\mu_t,v_t>V(x)+\frac12|\phi_t(x)|^2+\tilde f(x)\bigr]\varphi_t(x)\Bigr)dxdt.
\end{split}
\end{equation}
once more, given the test function $\varphi$, we use the fact that $w^\epsilon$ and $v^\epsilon$ coincide on the support of $\varphi$ for $\epsilon>0$ small enough.
This proves that $v$ is a weak solution (in the sense of distributions) of the adjoint equation.
\end{proof}

\subsubsection{\textbf{A Maximum Principle}}
\label{sub:MP}
This subsection is devoted to the statement and the proof of a version of the Pontryagin maximum principle tailored to our needs. Throughout the subsection we assume that $\bphi=(\phi_t)_{0\le t\le T}$ is a bounded measurable feedback control function, that $\bmu=(\mu_t)_{0\le t\le T}$ is the solution of the corresponding FPK equation \eqref{fo:kolmo1}, and that $u$ is a solution of the associated adjoint equation \eqref{fo:adjoint1}.
If $\bbeta=(\beta_t)_{0\le t\le T}$ is another bounded measurable feedback control function, for each $\epsilon>0$, we define $\bphi^\epsilon=(\phi^\epsilon_t)_{0\le t\le T}$ as $\phi^\epsilon_t=\phi_t+\epsilon\beta_t$, and we denote by $\bmu^\epsilon=(\mu^\epsilon_t)_{0\le t\le T}$ the solution of the FPK equation driven by the control $\phi^\epsilon_t$.

\begin{lemma}
    \label{le:mu_derivative}
    For each $t\in[0,T]$ and each bounded measurable function $\varphi$ on $\RR^d$, the limit
    $$
    \lim_{\epsilon\searrow 0}<\varphi,\frac{\mu_t^\epsilon -\mu_t}{\epsilon}>
    $$
    exists and is given by the integral of $\varphi$ with respect to a finite signed measure $\lambda_t$ such that $\lambda_t(\RR^d)=0$ and satisfying the PDE
    \begin{equation}
    \label{fo:lambdat_pde}
\partial_t\lambda_t=\frac12\Delta\lambda_t-\div(\phi_t\lambda_t)-\div(\beta_t\mu_t) +<\lambda_t,V>\mu_t-(V-<\mu_t,V>)\lambda_t
\end{equation}
in the sense of distributions with initial condition $\lambda_0=0$.
\end{lemma}

\begin{proof}
    If we denote by $X_t$ (resp. $X^\epsilon_t$) the state controlled by $\phi_t$ (resp. $\phi^\epsilon_t$), i.e. the solution of the stochastic differential equation $dX_t=\Phi_t(X_t)dt+dW_t$ (resp. $dX^\epsilon_t=\Phi^\epsilon_t(X^\epsilon_t)dt+dW_t$), Girsanov's theorem gives:
\begin{equation*}
\begin{split}
    \frac1\epsilon (<\varphi,\mu_t^\epsilon> - <\varphi,\mu_t>)
    &=\frac1\epsilon\Bigl(
    \frac{\EE\bigl[\varphi(X^\epsilon_t)e^{-\int_0^t V(X^\epsilon_s)ds}\bigr]}{\EE\bigl[e^{-\int_0^t V(X^\epsilon_s)ds}\bigr]}
    -\frac{\EE\bigl[\varphi(X_t)e^{-\int_0^t V(X_s)ds}\bigr]}{\EE\bigl[e^{-\int_0^t V(X_s)ds}\bigr]}
    \Bigr)\\
    &=\frac1\epsilon\Bigl(
    \frac{\EE\bigl[\varphi(X^\epsilon_t)e^{-\int_0^t V(X^\epsilon_s)ds}\bigr]
    -\EE\bigl[\varphi(X_t)e^{-\int_0^t V(X_s)ds}\bigr]}{\EE\bigl[e^{-\int_0^t V(X^\epsilon_s)ds}\bigr]}\\
    &\hskip 45pt
    -\EE\bigl[\varphi(X_t)e^{-\int_0^t V(X_s)ds}\bigr]
    \Bigl(\frac{1}{\EE\bigl[e^{-\int_0^t V(X_s)ds}\bigr]}
    -\frac1{\EE\bigl[e^{-\int_0^t V(X^\epsilon_s)ds}\bigr]}
    \Bigr)\Bigr)\\
    &=\frac1\epsilon
    \frac{\EE\bigl[\varphi(X_t)e^{-\int_0^t V(X_s)ds}
    \bigl(e^{\epsilon\int_0^t\beta_s(X_s)dW_s-\frac{\epsilon^2}{2}\int_0^t|\beta_s(X_s)|^2ds}-1
    \bigr)\bigr]}{\EE\bigl[e^{-\int_0^t V(X^\epsilon_s)ds}\bigr]}\\
    &\hskip 45pt
    -\frac{\EE\bigl[\varphi(X_t)e^{-\int_0^t V(X_s)ds}\bigr]}{\EE\bigl[e^{-\int_0^t V(X_s)ds}\bigr]}\frac1\epsilon
    \frac{\EE\bigl[e^{-\int_0^t V(X^\epsilon_s)ds}\bigr]-\EE\bigl[e^{-\int_0^t V(X_s)ds}\bigr]}{\EE\bigl[e^{-\int_0^t V(X^\epsilon_s)ds}\bigr]}
\end{split}
\end{equation*}
from which we conclude
\begin{equation}
    \label{fo:limit}
    \lim_{\epsilon\searrow 0} <\varphi,\frac{\mu_t^\epsilon-\mu_t}{\epsilon}>
    =\frac{\EE\bigl[\varphi(X_t)\bigl(\int_0^t\beta_s(X_s)dW_s\bigr)e^{-\int_0^t V(X_s)ds}\bigr]}{\EE\bigl[e^{-\int_0^t V(X_s)ds}\bigr]}
    -<\varphi,\mu_t>\frac{\EE\Bigl[\bigl(\int_0^t \beta_s(X_s)dW_s\bigr)e^{-\int_0^t V(X_s)ds}\Bigr]}{\EE\bigl[e^{-\int_0^t V(X_s)ds}\bigr]}
\end{equation}
because all the moments of the random variable $e^{-\epsilon\int_0^t \beta_s(X_s)dW_s-\frac{\epsilon^2}{2}\int_0^t|\beta_s(X_s)|^2ds}$ are finite, because
$$
\lim_{\epsilon\searrow 0}\EE\bigl[e^{-\int_0^t V(X^\epsilon_s)ds}\bigr]=\EE\bigl[e^{-\int_0^t V(X_s)ds}\bigr],
$$
and because
\begin{equation*}
\begin{split}
\lim_{\epsilon\searrow 0}\frac1\epsilon
\Bigl(\EE\bigl[e^{-\int_0^t V(X^\epsilon_s)ds}\bigr]-\EE\bigl[e^{-\int_0^t V(X_s)ds}\bigr]\Bigr)
 &= \frac1\epsilon
\Bigl(\EE\bigl[e^{-\int_0^t V(X_s)ds}\bigl[e^{\epsilon\int_0^t \beta_s(X_s)dW_s-\frac{\epsilon^2}{2}\int_0^t|\beta_s(X_s)|^2ds}-1\bigr]\Bigr)\\  
&= \EE\Bigl[\bigl(\int_0^t \beta_s(X_s)dW_s\bigr)e^{-\int_0^t V(X_s)ds}\Bigr].
\end{split}
\end{equation*}
Notice that the right hand side of \eqref{fo:limit} is a linear form in $\varphi$ which defines a signed measure $\lambda_t$ with total mass $0$.
In order to identify the PDE satisfied by $\lambda_t$ one can use the rules of It\^o calculus to compute the time derivative of $<\varphi,\lambda_t>$ by computing the derivative of the right hand side of \eqref{fo:limit}. Since these computations are long and tedious, we choose a more direct approach.
Subtracting one FPK equation from another we find that
$$
    \partial_t\frac{\mu^\epsilon_t - \mu_t}{\epsilon}
    =\frac12\Delta\frac{\mu^\epsilon_t - \mu_t}{\epsilon}
    -\div\bigl(\phi_t\frac{\mu^\epsilon_t - \mu_t}{\epsilon}\bigr)
    -\div\bigl(\frac{\phi^\epsilon_t - \phi_t}{\epsilon}\mu^\epsilon_t\bigr)
    +<\frac{\mu^\epsilon_t - \mu_t}{\epsilon},V>\mu_t
    -(V-<\mu_t,V>)\frac{\mu^\epsilon_t - \mu_t}{\epsilon}
$$
so taking the limit $\epsilon\searrow 0$ we derive the PDE \eqref{fo:lambdat_pde}. 
\end{proof}

\begin{lemma}
    \label{le:lambdat}
    $\nabla u$ is square integrable with respect to the measure $|\lambda|$ in the sense that
    \begin{equation}
    \label{fo:lambda_integrability}
    \int_0^T\int_{\RR^d}|\nabla u_t(x)|^2 |\lambda_t|(dx)dt<\infty.
    \end{equation}
\end{lemma}

The strategy of the proof is similar to the proof of estimate \eqref{fo:smooth_apriori_bound_3} in Lemma \ref{le:smooth_adjoint}.

\begin{proof}
    The PDE \eqref{fo:lambdat_pde} implies that for any test function $\varphi$ we have
\begin{equation}
    \label{fo:varphi2}
    \begin{split}
&\int_0^T\int \Bigl[\partial_t\varphi^2+\frac12\Delta\varphi^2 + \phi_t\nabla\varphi^2 -(V-<\mu_t,V>)\varphi^2\Bigr]\lambda_t(dx)dt\\
&\hskip 45pt
+\int_0^T\int \Bigl[\beta_t\nabla\varphi^2+<V,\lambda_t>\varphi^2
\Bigr]\mu_t(dx)dt=<\varphi^2,\lambda_T>.
    \end{split}
\end{equation}
Using the fact that $u$ solves the adjoint equation \eqref{fo:adjoint1} we can write
\begin{equation*}
\begin{split}
    &\int_0^T\int u_t\bigl[f\bigl(.,\phi_t(.)\bigr)+\phi_t\nabla u_t\bigr]\lambda_t(dx)dt\\
    &\hskip 35pt
    =\int_0^T\int u_t\bigl[-\partial_t u_t-\frac12 \Delta u_t+(V-<\mu_t,V>)u_t-V<\mu_t,u_t>\bigr]\lambda_t(dx)dt\\
    &\hskip 35pt
    =-\frac12 \int_0^T\int \bigl[\partial_t u_t^2+\frac12 \Delta u_t^2+\phi_t\nabla u_t^2-(V-<\mu_t,V>)u_t^2\bigr]\lambda_t(dx)dt\\
    &\hskip 75pt
    +\int_0^T\int \bigl[\frac12 |\nabla u_t|^2+\frac12 \phi_t\nabla u_t^2 +\frac12 (V-<\mu_t,V>)u_t^2-V<\mu_t,u_t>u_t\bigr]\lambda_t(dx)dt\\
    &\hskip 35pt
    =\frac12\int_0^T\int \bigl[\beta_t\nabla u_t^2+<V,\lambda_t> u_t^2\bigr]\mu_t(dx)dt+\frac12<u_T^2,\lambda_T>\\
    &\hskip 75pt
    +\int_0^T\int \bigl[\frac12 |\nabla u_t|^2+\frac12 \phi_t\nabla u_t^2 +\frac12 (V-<\mu_t,V>) u_t^2-V<\mu_t,u_t>u_t\bigr]\lambda_t(dx)dt
\end{split}
\end{equation*}
from which we get
\begin{equation*}
\begin{split}
    \frac12\int_0^T\int_{\RR^d}|\nabla u_t|^2\lambda_t(dx)dt
    &=\int_0^T\int u_t\bigl[f\bigl(.,\phi_t(.)\bigr)+\phi_t\nabla u_t\bigr]\lambda_t(dx)dt\\
    &\hskip 35pt
    -\frac12\int_0^T\int \bigl[\beta_t\nabla u_t^2+<V,\lambda_t> u_t^2\bigr]\mu_t(dx)dt
    -\frac12<u_T^2,\lambda_T>\\
    &\hskip 35pt
    -\int_0^T\int \bigl[\frac12 \phi_t\nabla u_t^2 +\frac12 (V-<\mu_t,V>) u_t^2-V<\mu_t,u_t>u_t\bigr]\lambda_t(dx)dt\\
    &\le \|u\|_\infty\|f\|_\infty|\lambda|([0,T]\times\RR^d) +2\|u\|_\infty^2\|\phi\|_\infty^2|\lambda|([0,T]\times\RR^d)\\
    &\hskip 35pt
    +\frac18\int_0^T|\int\nabla u_t|^2\lambda_t(dx)dt
    +\frac12\|\beta\|_\infty\Bigl(T\|u\|_\infty^2+\int_0^T\int|\nabla u_t|^2\mu_t(dx)dt\Bigr)\\
    &\hskip 35pt
    + 2(\|u\|_\infty^2+\|u\|_\infty)|\lambda|([0,T]\times\RR^d)\\
    &\hskip 35pt
    +2T\|\phi\|_\infty^2\|u\|_\infty^2+\frac18\int_0^T\int|\nabla u_t|^2\lambda_t(dx)dt
\end{split}
\end{equation*}
from which we conclude.
\end{proof}

\begin{proposition}
    \label{pr:MP}
Let $\bphi=(\phi_t)_{0\le t\le T}$ be a bounded measurable feedback control function, let $\bmu=(\mu_t)_{0\le t\le T}$ be the corresponding solution of the FPK equation \eqref{fo:kolmo1}, and let $u$ the solution of the corresponding adjoint equation \eqref{fo:adjoint1}.
If $\bbeta=(\beta_t)_{0\le t\le T}$ is another bounded measurable feedback control function, we have: 
\begin{equation}
    \label{fo:cost_derivative}
\frac{d}{d\epsilon}J(\bphi+\epsilon\bbeta)\Bigr|_{\epsilon=0}
=\int_0^T<\beta_t(\nabla u_t+\phi_t), \mu_t>\,dt.
\end{equation}
As a result, if $\bphi$ is a critical point, then
\begin{equation}
    \label{fo:optimality}
\phi_t(x)=-\nabla u_t(x),\qquad\qquad \mu_t-a.s.\; x\in\RR^d,\quad a.e.\; t\in[0,T].
\end{equation}
\end{proposition}

\begin{proof}
We use freely the notations introduced at the beginning of the subsection and in the statement of the previous lemma. 
\vskip 2pt\noindent
\textbf{Step 1.} We derive a first expression for the derivative \eqref{fo:cost_derivative}.
\begin{equation*}
    \begin{split}
    &\frac1{\epsilon}\bigl[J(\bphi+\epsilon\bbeta)-J(\bphi)\bigr]\\
    &\hskip 25pt
    =\frac1{\epsilon}\int_0^T\Bigl(
    <\tilde f,\mu^\epsilon_t>+\frac12<|\phi^\epsilon_t|^2,\mu^\epsilon_t>
    -<\tilde f,\mu_t>-\frac12<|\phi_t|^2,\mu_t>\Bigr)\;dt +\frac1\epsilon<g,\mu^\epsilon_T-\mu_T>\\
    &\hskip 25pt
    =\int_0^T\Bigl(<\tilde f,\frac{\mu^\epsilon_t - \mu_t}{\epsilon}> +\frac12<|\phi_t|^2,\frac{\mu^\epsilon_t - \mu_t}{\epsilon}> 
    +\frac12<\frac{|\phi^\epsilon_t|^2 - |\phi_t|^2}{\epsilon},\mu^\epsilon_t>\Bigr)dt +<g,\frac{\mu^\epsilon_T - \mu_T}{\epsilon}>
    \end{split}
\end{equation*}
so that 
\begin{equation}
\label{fo:first_derivative}
    \lim_{\epsilon\searrow 0}\frac1\epsilon\bigl[J(\bphi+\epsilon\bbeta)-J(\bphi)\bigr]=\int_0^T\Bigl(<\tilde f,\lambda_t> +\frac12<|\phi_t|^2,\lambda_t> 
    +<\beta_t\phi_t,\mu_t>\Bigr)dt +<g,\lambda_T>.
\end{equation}
where we used the stronger form of convergence toward $\lambda_t$ proven in Lemma \ref{le:mu_derivative} because $|\phi_t|^2$ is merely bounded measurable and not necessarily continuous.
\vskip 2pt\noindent
\textbf{Step 2.} We now prove the formula
\begin{equation}
    \label{fo:uTnuT}
<u_T,\lambda_T>=\int_0^T\Bigl(<\beta_t\nabla u_t,\mu_t>-\bigl(<\tilde f,\lambda_t>+\frac12<|\phi_t|^2,\lambda_t>\bigr)\Bigr)\,dt
\end{equation}
where the function $u$ is the solution of the adjoint equation driven by the control $\bphi$. Recall that Theorem~\ref{th:adjoint_solution} says that $u$ is bounded and continuous and that its derivatives in the sense of distributions are square integrable functions with respect to $\bmu$. 

Let us define the functions $u_t^n(x)=\chi^n(x)u_t(x)$ and  $\lambda_t^n(dx)=\chi^n(x)\lambda_t(dx)$ where the sequence $(\chi^n)_{n\ge 1}$ of $C^\infty$ functions is such that $0\le \chi^n(x)\le 1$, $\chi^n(x)=1$ if $|x|\le n$, $\chi^n(x)=0$ if $|x|>n+1$ and its derivatives of all orders are uniformly bounded. We first notice that 
\begin{equation}
    \label{fo:first_approx}
    <u_T,\lambda_T>=\lim_{n\nearrow \infty}<u^n_T,\lambda^n_T>
\end{equation}
because $u_T$ is bounded. Next,
using the facts that $\partial_t u^n_t=\chi^n \partial_t u_t$, $\partial_t  \lambda^n_t=\chi^n \partial_t \lambda_t$ and $\lambda_0=0$ , we get:
\begin{equation*}
    \begin{split}
<u^n_T,\lambda^n_T>&=<u^n_0,\lambda^n_0>+\int_0^T<u^n_t,\partial_t\lambda^n_t> + \int_0^T<\partial_t u^n_t,\lambda^n_t>\\
&=\int_0^T\Bigl(<\chi^n u_t,\frac12\chi^n\Delta \lambda_t> -<\chi^n u_t,\chi^n\div(\phi_t\lambda_t)>-<\chi^nu_t,\chi^n\div(\beta_t\mu_t)>
\\
&\hskip 95pt
+<\lambda_t,V><\chi^nu_t,\chi^n\mu_t>-<\chi^nu_t,\chi^n(V-<\mu_t,V>)\lambda_t>\\
&\hskip 25pt
+<-\frac12\chi^n\Delta u_t,\chi^n\lambda_t>-<\chi^n\phi_t\nabla u_t,\chi^n\lambda_t>+<\chi^n(V-<\mu_t,V>)u_t,\chi^n\lambda_t>\\
&\hskip 95pt
-<\mu_t,u_t><\chi^nV,\chi^n\lambda_t>-\frac12<\chi^n|\phi_t|^2,\chi^n\lambda_t>-<\chi^n\tilde f,\chi^n\lambda_t>\Bigr)\; dt.
    \end{split}
\end{equation*}
Using integration by parts we get:
\begin{equation*}
\begin{split}
<\chi^n u_t,\chi^n\Delta \lambda_t>
&=<(\chi^n)^2 u_t,\Delta\lambda_t>\\
&=-<(\chi^n)^2 \nabla u_t,\nabla\lambda_t>-<\nabla(\chi^n)^2 \; u_t,\nabla\lambda_t>\\
&=-<(\chi^n)^2 \nabla u_t,\nabla\lambda_t>+<\nabla\bigl(\nabla(\chi^n)^2 u_t\bigr),\lambda_t>,
\end{split}
\end{equation*}
and similarly
\begin{equation*}
\begin{split}
<\chi^n \Delta u_t,\chi^n \lambda_t>
&=<\Delta u_t,(\chi^n)^2 \lambda_t>\\
&=-<\nabla u_t,(\chi^n)^2 \nabla\lambda_t>-<\nabla(\chi^n)^2\nabla u_t, \lambda_t>,
\end{split}
\end{equation*}
so that 
$$
\lim_{n\to\infty}<\chi^n u_t,\chi^n\Delta \lambda_t>-<\chi^n \Delta u_t,\chi^n \lambda_t>=
\lim_{n\to\infty}<\nabla\bigl(\nabla(\chi^n)^2 u_t\bigr),\lambda_t>-\lim_{n\to\infty}<\nabla(\chi^n)^2\nabla u_t, \lambda_t>=0
$$
because $u_t$ is bounded, and $\nabla u_t$ is square integrable for $\lambda_t$ according to Lemma \ref{le:lambdat}.
Again, we use the fact that the derivatives of $\chi^n$ are uniformly bounded and supported in the annulus $n\le |x|\le n+1$ whose $\lambda_t$ measure tends to $0$.
Next
$$
    \lim_{n\to\infty}<\chi^n u_t,\chi^n\div(\phi_t\lambda_t)>+<\chi^n\phi_t\nabla u_t,\chi^n\lambda_t>
    =-\lim_{n\to\infty}<\nabla(\chi^n)^2 \phi_tu_t,\lambda_t>
    =0
$$
because $\phi_t$ and $u_t$ are bounded. Similarly
$$
    \lim_{n\to\infty}<\chi^nu_t,\chi^n\div(\beta_t\mu_t)>=-<\beta_t\nabla u_t,\mu_t>,
$$
as well as
$$
    \lim_{n\to\infty}<\chi^n|\phi_t|^2,\chi^n\lambda_t)>=<|\phi_t|^2,\lambda_t>,
$$
and
$$
    \lim_{n\to\infty}<\chi^n\tilde f,\chi^n\lambda_t)>=<\tilde f,\lambda_t>.
$$
Finally
$$
    \lim_{n\to\infty} <\lambda_t,V><\chi^nu_t,\chi^n\mu_t>-<\mu_t,u_t><\chi^nV,\chi^n\lambda_t>
=0,
$$
since $V$ is bounded. This concludes the proof of Step 2.

\vskip 2pt\noindent
\textbf{Step 3.} We conclude the proof of the proposition
by plugging expression \eqref{fo:uTnuT} in \eqref{fo:first_derivative}, and obtain the expression of the derivative of the expected cost stated in \eqref{fo:cost_derivative}.
From this we conclude that \eqref{fo:optimality} holds whenever $\bphi$ is a critical point of the expected cost $J$.
\end{proof}

\subsubsection{\textbf{Value Function.}}
\label{sub:value2}

We now introduce the value function:
\begin{equation}
\label{fo:deterministic_value}
    U^{(1)}(t,\mu)=\inf_{(\phi_s,\mu_s)_{t\le s\le T}}\int_t^T F^{(1)}(\mu_s,\phi_s)\,ds + G^{(1)}(\mu_T),
\end{equation}
where the infimum is over the couples $(\bphi,\bmu)\in\AA^{(2)}_{[t,T]}$ satisfying the \emph{initial condition} $\mu_t=\mu$. Because of the dynamic programming principle, 
this value function is expected to solve in some appropriate sense the HJB equation: \begin{equation}
\label{fo:HJB4m}
\partial_t U^{(1)}(t,\mu) + \HH^{(1)*}\left(\mu, \frac{\delta U^{(1)}}{\delta \mu}(t,\mu)\right)=0
\end{equation}
with terminal condition $U^{(1)}(T,\mu)=G^{(1)}(\mu)$, where the minimized Hamiltonian $\HH^{(1)*}(\mu,\varphi)$ is defined as:
$$
\HH^{(1)*}(\mu,\varphi)=\inf_{\phi}\HH^{(1)}(\mu,\varphi,\phi),
$$
where the infimum is taken over the  $A$-valued measurable functions $\phi$ on $\RR^d$ satisfying $\int|\phi(x)|^2\mu(dx)<\infty$.

\subsubsection{\textbf{Minimization of the Hamiltonian.}}
\label{sub:minimization2}
The following assumption is made in order to be able to use lighter notations. 
  
\begin{assumption}
We assume the existence of a function $\hat\phi: (\mu,\varphi)\mapsto \hat\phi[\mu,\varphi]$  which is an $A$-valued measurable function on $\RR^d$ satisfying for every $(\mu,\varphi)$ satisfying
$$
    \int_{\RR^d}|\hat\phi[\mu,\varphi]|^2\mu(dx)<\infty,\qquad\text{and}\qquad  \hat\phi[\mu,\varphi]\in\arginf_{\phi}\HH^{(1)}(\mu,\varphi,\phi).
$$
\end{assumption}
\noindent
Under this assumption, $\HH^{(1)*}(\mu,\varphi)$ can be written as:
$$
\HH^{(1)*}(\mu,\varphi)=\HH^{(1)}(\mu, \varphi,\hat\phi[\mu,\varphi]),
$$
and the HJB equation \eqref{fo:HJB4m} can be conveniently rewritten as:
\begin{equation}
\label{fo:HJB1}
    \partial_t U^{(1)}(t,\mu) + \HH^{(1)}\left(\mu, \frac{\delta U^{(1)}}{\delta \mu}(t,\mu), \hat\phi\Bigl[\mu,\frac{\delta U^{(1)}}{\delta \mu}(t,\mu)\Bigr]\right)=0.
\end{equation}
Recall that $\frac{\delta U^{(1)}}{\delta \mu}(t,\mu)$ is understood as a functional ``flat'' derivative, and as such, it is a function of $x\in \RR^d$.

\begin{example*} 
In the particular case of the separable running cost function $f$ given by~\eqref{fo:separable_running_cost}, if we assume that the derivative of the function $\varphi$ in the sense of distributions is a function, we have:
\begin{equation*}
\begin{split}
    \hat\phi[\mu,\varphi]& \in \arginf_\phi -\frac12<\nabla \mu,\nabla\varphi> + <\mu,\phi \nabla\varphi> - <\mu,(V-<\mu,V>)\varphi>\\
    &\hskip 135pt 
    +  \frac12 \int|\phi(x)|^2\;\mu(dx) + \int\tilde f(x)\mu(dx)\\
    & = \arginf_\phi \int_{\RR^d} [\phi(x) \cdot \nabla\varphi(x) +\frac{1}{2}|\phi(x)|^2]\; \mu(dx),
\end{split}
\end{equation*}
from which we get:
\begin{equation}
\label{fo:phi_hat}
	\hat\phi[\mu,\varphi](x) =\Pi_A\bigl( -\nabla\varphi(x)\bigr),\qquad\qquad  x\in\supp(\mu)
\end{equation}
by minimizing under the integral sign for each fixed $x\in \supp(\mu)$, where the notation $\text{supp}(\mu)$ stands for the topological support of the measure $\mu$, and $\Pi_A$ for the projection of $\RR^d$ onto the closed convex set $A$, defined as:
\begin{equation}
\label{fo:Pi_A}
\Pi_A(x)=\arginf_{y\in A}\|x-y\|
\end{equation}
which gives the element of $A$ at the shortest distance from $x$. 
For each $\varphi\in C_c^\infty(\RR^d)$, $x\hookrightarrow - \nabla\varphi(x)$ is a bounded measurable function on $\RR^d$ with values in $\RR^d$. Unfortunately, we are not guaranteed that it takes values in $A$ as the definition of admissibility requires, hence the need for the projection $\Pi_A$. Still, this projection remains bounded and the feedback function $\hat\phi(\mu,\varphi)$ is admissible. 

\vskip 2pt
	If, for the sake of simplicity, we take $A=\RR^d$, we then have:
        \begin{equation}
        \label{fo:phi_hat_wholespace}
        \hat\phi[\mu,\varphi](x)=-\nabla \varphi(x)
        \end{equation}
        and the minimized Hamiltonian takes the form
        \begin{equation}
        \begin{split}
        \inf_{\phi}\HH^{(1)}(\mu,\varphi,\phi)&=\int \mu(dx) \Bigl[\frac12\Delta\varphi(x)-\frac12   |\nabla \varphi(x)|^2 - (V(x)-<\mu,V>)\varphi(x) +\tilde f(x)\Bigr].
        \end{split}
        \end{equation}
        Then, the HJB equation reads:
        \begin{equation}
        \label{fo:HJB_2}
        \begin{split}
        &0=\partial_t U^{(1)}(t,\mu) +  \int \mu(dx)\Bigl[\frac12\Delta_x\frac{\delta U^{(1)}}{\delta \mu}(t,\mu)(x) 
        -\frac1{2} |\nabla_x \frac{\delta U^{(1)}}{\delta \mu} (t,\mu)(x)|^2
        \\
        &\hskip 85pt
          - (V(x)-<\mu,V>)  \frac{\delta U^{(1)}}{\delta \mu}(t,\mu)(x) + \tilde f(x)\Bigr].
        \end{split}
        \end{equation}
\end{example*}

\subsubsection{\textbf{The Forward-Backward PDE System}}
\label{sec:fbpde-markovian}

As before, we assume that the optimal control $\hat\bphi$ whose existence we proved in Proposition \ref{pr:existence_control1}
is bounded. Then,  the form of the Pontryagin maximum principle proved in Proposition \ref{pr:MP} implies that $\hat\phi_t(x)=-\nabla_x\hat u_t(x)$ $\hat\mu_t$ - almost surely where $\hat\bmu$ is the solution of the FPK equation driven by the control $\hat\bphi$ and $\hat u_t$ is the solution of the corresponding adjoint equation. So injecting this formula in the FKP and the adjoint equations we see that the couple $(\hat \mu,\hat u)$ is a solution of the forward backward PDE system:
\begin{subequations}\label{fo:pde_system1}
\begin{empheq}[left=\empheqlbrace]{align}
    &\partial_t \mu=\frac12 \Delta_x \mu +\div_x(\nabla_xu\;\mu)  -(V-<\mu,V>)\mu \label{fo:pde_system1a}
    \\
    &0=\partial_t u + \frac12\Delta_x u- \frac12 |\nabla_x u|^2-(V-<\mu,V>)u +V<\mu,u>+\tilde f \label{fo:pde_system1b}.
\end{empheq}
\end{subequations}
on the support of $\bmu$.

\begin{remark}
Both equations in the above system are strongly coupled, and the difficulty in the second equation is twofold. The combination of the nonlinear term $|\nabla_xu|^2$ and the non-local term $<\mu,u>$ prevents us from appealing directly to standard results from PDEs or BSDEs to study the well-posedness of the system.
\end{remark}

The sense in which equation \eqref{fo:pde_system1a} should be understood is clear when the term $\nabla_x u$ is well defined as a function: it should be understood in the sense of distributions. However, the situation with the second equation is not as clear, in part because of the presence of non-local and non-linear terms. While the notion of viscosity solution could be viewed as a natural option, we shall refrain from using this interpretation for two independent reasons. Firstly, the comparison constraints involved in the definition of a viscosity solution would need to be modified to account for the non-local terms. Secondly, viscosity solutions are not always differentiable, and our derivation of equation \eqref{fo:pde_system1b} was based on the substitution in the adjoint equation, of the identification of the optimal control as the gradient of its own adjoint function (i.e. co-state). In order to avoid this oxymoron created by a seemingly circular argument, we choose to search for solutions by first freezing the non-local terms, and then,  work with \emph{mild solutions} in order to avoid the possible lack of differentiability of the solution. The following lemma highlights the fact that, absent the non-local nature of the equations, viscosity and mild solutions are the same in our setting. We believe that such a result is part of the folklore on the subject. See for example \cite{BeckJentzen}. We state it in the form we need for further references, and we give a proof for the sake of completeness. 

\begin{lemma}
    \label{le:viscosity_FK}
   Let $G:\RR^d\to\RR$ be bounded and Lipschitz, and let $F:[0,T]\times \RR^d\times\RR\ni(t,x,u)\mapsto F(t,x,u)\in\RR$ be bounded and Lipschitz in $(x,u)$ uniformly in $t\in[0,T]$.
    \begin{itemize}
        \item[(i)] The semilinear PDE 
        \begin{equation}
        \label{fo:pde}
        \partial_t u+\frac12\Delta u +F(t,x,u(t,x))=0,\qquad u(T,x)=G(x),
        \end{equation}
    has a unique viscosity solution. 
    \item[(ii)] If for $(t,x)\in[0,T]\times\RR^d$, we set $B^{t,x}_r=x+B_r-B_t$ for $t\le r\le T$ where $(B_t)_{0\le t\le T}$ is some process of Brownian motion, the equation 
        \begin{equation}
        \label{fo:FK}
        v(t,x)=\EE\Bigl[G(B^{t,x}_T)+\int_t^TF\bigl(r,B^{t,x}_r,v(r,B^{t,x}_r)\bigr)dr\Bigr]
        \qquad (t,x)\in[0,T]\times\RR^d
        \end{equation}
        has a unique bounded (jointly) continuous solution.
\item[(iii)] $u=v$.
\item[(iv)] Furthermore, this unique solution $u$ satisfies: 
\begin{itemize}
    \item[(iv)-1] $u$ is continuously differentiable and its gradient is uniformly bounded whenever $F\in C^{0,1}_b([0,T]\times \RR^d\times\RR)$, that is when $F$ is continuously differentiable in $(x,u)$ with bounded derivatives.
    \item[(iv)-2] $u$ is a classical solution of \eqref{fo:pde} in $C^{1,2}_b([0,T]\times\RR^d)$ whenever $G$ is $3$-times continuously differentiable with derivatives growing at most polynomially, and for each fixed $t\in[0,T]$ $(x,u)\mapsto F(t,x,u)$ is $3$-times continuously differentiable with bounded derivatives of orders $1$ and $2$ and polynomially bounded derivatives of order $3$. 
\end{itemize}
    \end{itemize}
\end{lemma}
\vskip 4pt\noindent
Instead of looking for the most general statement, we merely stated the results under assumptions which will be satisfied in the situations in which we need them, and for which we can use powerful already existing results. For example, we could dispense with the boundedness assumptions for $G$ and $F$, in which case the functions $u$ and $v$ in (i) and (ii) would be of at most linear growth because of the global Lipschitz assumption on $F$ and $G$. 

\begin{proof}
    (i) This is a particular case of \cite[Theorem 4.3]{PardouxPeng}. Indeed, for each $(t,x)\in[0,T]\times \RR^d$ we consider the BSDE
\begin{equation}
        \label{fo:bsde}
Y^{t,x}_s= G(B^{t,x}_T) +\int_s^T F\bigl(r,B^{t,x}_r,Y^{t,x}_r\bigr)dr - \int_s^T Z^{t,x}_rd B_r
\end{equation}
where $B^{t,x}_r=x+B_r-B_t$ for $t\le r\le T$ is a process of Brownian motion conditioned to be at $x$ at time $t$.
Under our assumptions, this BSDE has a unique solution for each fixed $(t,x)\in[0,T]\times\RR^d$. Next we define the function $u$ by setting $u(t,x)=Y^{t,x}_t$ (which is deterministic), and
\cite[Theorem 4.3]{PardouxPeng} says that $u$ is the unique viscosity solution of \eqref{fo:pde}.

\vskip 1pt
(ii) Like in the proof of Proposition~\ref{pr:existence_adjoint1}, in order to define a function $\Psi$ from $C_b([0,T] \times\RR^d)$ into itself we posit that $[\Psi v](t,x)$ is given by the right hand side of \eqref{fo:FK}, and exactly as in the proof of Proposition~\ref{pr:existence_adjoint1}, we check that $\Psi$ is a strict contraction for the norm $\|\cdot\|_\alpha$ for $\alpha$ large enough.

\vskip1pt
(iii) By uniqueness of the solution of the BSDE \eqref{fo:bsde}, if $t\le r\le T$, we have:
$$
Y^{t,x}_{r}=Y_{r}^{r,B_{r}^{t,x}}=u\bigl(r,B^{t,x}_r\bigr)
$$
by definition of the function $u$ in (i). Accordingly, the BSDE \eqref{fo:bsde} can be rewritten as
\begin{equation}
        \label{fo:bsde'}
Y^{t,x}_s= G(B^{t,x}_T) +\int_s^T F\bigl(r,B^{t,x}_r,u\bigl(r,B^{t,x}_r\bigr)\bigr)dr - \int_s^T Z^{t,x}_rd B_r
\end{equation}
and setting $s=t$ and taking expectations on both sides we get:
$$
u(t,x)=Y^{t,x}_t=\EE\Bigl[ G(B^{t,x}_T) +\int_t^T F\bigl(r,B^{t,x}_r,u\bigl(r,B^{t,x}_r\bigr)\bigr)dr
\Bigr]
$$
showing that $u$ is a solution of \eqref{fo:FK}, and hence that $u=v$ by uniqueness of such a solution.
Notice that the stochastic integral in \eqref{fo:bsde'} is indeed a square integrable martingale so its expectation is $0$.

\vskip 1pt
(iv)-1 is a consequence of \cite[Theorem 3.1 p.1397]{MaZhang} and its corollary \cite[Corollary 3.2]{MaZhang}, and (iv)-2 follows \cite[Theorem 3.2]{PardouxPeng}. 
\end{proof}

\begin{proposition}
\label{pr:existence_pde1}
In the case of separable cost functions, for each continuous flow $\hat\bmu=(\hat\mu_t)_{0\le t\le T}$ of probability measures on $\RR^d$,
the PDE \eqref{fo:pde_system1b} admits a  solution in the sense of viscosity which is continuously differentiable with uniformly bounded first derivatives. Moreover, this solution is actually a classical solution when $\tilde f$ and $g$ are three times differentiable with bounded derivatives.
\end{proposition}

\begin{proof}
As suggested by the content of Remark \ref{re:stronger_apriori_bound}, we limit the search for a solution to bounded functions $u$. In fact, if we denote by $K$ the bound in that remark, using $K_{T}=Ke^{2 T}$, we look for $u$ in the class of functions satisfying $\|u_t\|_\infty\le K_{T} e^{-2 t}$ for all $t\in[0,T]$.

\vskip 6pt
\emph{Step 1}. 
In fact, instead of looking directly for solutions $\tilde u$ satisfying $\|\tilde u_t\|_\infty\le K_Te^{-2 t}$, we look for the exponentials of such functions. Indeed, if we set $\tilde v_t= e^{-\tilde u_t}$ and define
$$
F^{\tilde v}(t,x)=(V(x)-<\hat\mu_t,V>)\log \tilde v(t,x) - V(x)<\hat \mu_t,\log \tilde v_t> +\tilde f(x),
$$
then since the function  $\log\tilde v$ is bounded, the function 
$F^{\tilde v}$ is also bounded, and we
define $v=\Psi(\tilde v)$ as the solution of 
\begin{equation}
\label{fo:***}
0=\partial_t v +\frac12\Delta v - F^{\tilde v} v, \qquad v(T,\cdot)=G,
\end{equation}
which is given by
\begin{equation}
\label{fo:***2}
v(t,x)=\EE\Bigl[ G(B^{t,x}_T) \exp\bigl[-\int_t^TF^{\tilde v}(r,B^{t,x}_r)dr\bigr] \Bigr].
\end{equation}
Notice that 
\begin{equation}
\label{fo:bound1}
    \|F^{\tilde v}_t\|_\infty \le 2\|\log \tilde v_t\|_\infty +\|\tilde f\|_\infty
     \le 2K_T e^{-2t} +\|\tilde f\|_\infty,
\end{equation}
 so if $u=-\log v$ we have 
\begin{align}
    \|u_t\|_\infty&=\sup_x\;\Bigl|\log\EE\Bigl[ G(B^{t,x}_T) \exp\bigl[-\int_t^TF^{\tilde v}(r,B^{t,x}_r)dr\bigr] \Bigr]\Bigr|
    \notag
    \\
    &\le \|g\|_\infty +2K_T\int_t^Te^{-2 r}dr+(T-t)\|\tilde f\|_\infty
    \notag
    \\
    &\le K_Te^{-2 t}\Bigl(\frac{\|g\|_\infty +(T-t)\|\tilde f\|_\infty}{K}+1-e^{-2(T-t)}\Bigr)
    \label{eq:adjoint-proof-boundu-KT}
\end{align}
Now increasing the value of the original $K$ if necessary, we can assume without any loss of generality that 
\begin{equation}
    \label{fo:K_lower_bound}
K> e^{2T}\bigl(\|g\|_\infty + e^T\|\tilde f\|_\infty\bigr)
\end{equation}
in which case:
$$
    \frac{\|g\|_\infty +(T-t)\|\tilde f\|_\infty}{K}+1-e^{-2(T-t)}\le 1.
$$
So, by~\eqref{eq:adjoint-proof-boundu-KT}, we have $\|u_t\|_\infty\le K_Te^{-2 t}$ for $0\le t\le T$. So if $-K_Te^{-2t}\le \log \tilde v_t\le K_T e^{-2t}$, then $v=\Psi(\tilde v)$ satisfies 
$-K_Te^{-2t}\le \log v_t\le K_T e^{-2t}$ as well. 

\vskip 4pt
\emph{Step 2.} 
Given the above preliminaries, we define the closed subset
$$
    \cC_K=\{ v\in C_b([0,T]\times\RR^d;\RR_+);\;-K_Te^{-2t}\le \log v_t\le K_T e^{-2t}\;\;0\le t\le T\}
$$
of $C_b([0,T]\times\RR^d;\RR_+)$, which is left invariant by the map $\Psi$ since we just proved that $\Psi(\cC_K) \subseteq \cC_K$. Next, we prove that $\Psi$ has a unique fixed point in $\cC_K$, using a strict contraction argument on $\cC_K$ endowed with the norm $\| v \|_\alpha = \sup_{t \in [0,T]} e^{\alpha t}\|v_t\|_\infty$.
If $\tilde v^1$ and $\tilde v^2$ are in $\cC_K$, we have:
\begin{equation}
\begin{split}
    \bigl\|\bigl(\Psi(\tilde v^1)-\Psi(\tilde v^2)\bigr)_t\bigr\|_\infty 
    &= \sup_x\EE\Bigl[\Bigl| G(B^{t,x}_T) \Bigl(\exp\bigl[-\int_t^TF^{\tilde v^1}(r,B^{t,x}_r)dr\bigr]
    -\exp\bigl[-\int_t^TF^{\tilde v^2}(r,B^{t,x}_r)dr\bigr]  \Bigr) \Bigr|\Bigr]
    \\
    &\le e^{\|g\|_\infty} e^{T\sup_{0\le t\le T}\| F_t^{\tilde v^1}\|_\infty\vee \| F_t^{\tilde v^2}\|_\infty}
    \EE\Bigl[\int_t^T\bigl| F^{\tilde v^1}(r,B^{t,x}_r)-F^{\tilde v^2}(r,B^{t,x}_r)\bigr| \;dr\Bigr]
    \\
    &\le C'\int_t^T\|\log\tilde v^1_r- \log \tilde v^2_r\|_\infty\;dr
    \\
    &\le C'e^{K}\int_t^T\|\tilde v^1_r- \tilde v^2_r\|_\infty\;dr,
\end{split}
\end{equation}
where we used the fact that $\tilde  v^1$ and $\tilde v^2$ are in $\cC_K$ and  $C'=2\exp[\|g\|_\infty +T(2K_T+\|\tilde f\|_\infty)]$. Notice that we used the fact that $\tilde v^1$ and $\tilde v^2$ are uniformly bounded from abobe and from below away from $0$, so we only need to rely on the local Lipschitz properties of the exponential and logarithm functions.
Moreover, we have
\begin{equation}
e^{\alpha t}\bigl\|\bigl(\Psi(\tilde v^1)-\Psi(\tilde v^2)\bigr)_t\bigr\|_\infty
\le C'e^K e^{\alpha t}\int_t^Te^{-\alpha r} e^{\alpha r}\|\tilde v^1_r- \tilde v^2_r\|_\infty\;dr\le C'e^K \frac{1-e^{-\alpha (T-t)}}{\alpha}\sup_{0\le r\le T}e^{\alpha r}\|\tilde v^1_r- \tilde v^2_r\|_\infty
\end{equation}
so that 
$$
\bigl\|\Psi(\tilde v^1)-\Psi(\tilde v^2)\bigr\|_\alpha=\sup_{0\le t\le T}e^{\alpha t}\bigl\|\bigl(\Psi(\tilde v^1)-\Psi(\tilde v^2)\bigr)_t\bigr\|_\infty\le \frac{C'e^K}{\alpha}\|\tilde v^1- \tilde v^2\|_\alpha
$$
which shows that $\Psi$ is a strict contraction if $\alpha > C'e^K$. If $v$ is the unique fixed point of this strict contraction, then $v$ satisfies 
$$
0=\partial_t v +\frac12\Delta v - F^{v} v, \qquad v(T,\cdot)=G,
$$
in other words:
\begin{equation}
    \label{fo:v}
0=\partial_t v +\frac12\Delta v - \bigl[(V(x)-<\hat\mu_t,V>)\log v(t,x) - V(x)<\hat \mu_t,\log v_t> +\tilde f(x)\bigr] v, \qquad v(T,\cdot)=G.
\end{equation}

\vskip 4pt
\emph{Step 3.} We now return to the function $u$ and conclude the proof of the proposition.
The function 
\begin{align*}
    [0,T]\times\RR^d\times [e^{-K_T},e^{K_T}]\ni(t,x,u)\mapsto F(t,x,u)
    &=(V(x)-<\hat\mu_t,V>)u\log u \\
    &\qquad +[- V(x)<\hat \mu_t,\log  v_t> +\tilde f(x)]u
\end{align*}
satisfies all the assumptions of Lemma \ref{le:viscosity_FK} since it is bounded and Lipschitz continuous in $(x,u)$ uniformly in $t\in[0,T]$. Since $G$ also satisfies the required assumption, and since equation \eqref{fo:v} coincides with \eqref{fo:pde}, we can conclude that the function $v$ we just constructed as a fixed point of $\Psi$ is continuously differentiable and its gradient is uniformly bounded. Moreover, $v$ is a classical solution whenever $g$ is $3$-times continuously differentiable with bounded derivatives.

Since the logarithm function is monotone, the function $u=-\log v$ is a viscosity solution of
\begin{equation}
    \label{fo:u}
0=\partial_t u_t +\frac12\Delta u_t -\frac12|\nabla u_t|^2 + (V-<\hat\mu_t,V>) u_t +V<\hat\mu_t,u_t>+\tilde f, \qquad u_T=g,
\end{equation}
which is equation \eqref{fo:pde_system1b} with $\bmu = \hat\bmu$. Also, since $v$ is bounded away from $0$, $u$ is continuously differentiable and its gradient is uniformly bounded.
Moreover, $u$ is a classical solution of \eqref{fo:u} whenever $g$ is $3$-times continuously differentiable with bounded derivatives.
\end{proof}

\begin{remark}
    While the derivation of the necessity of the existence of a solution for the PDE system \eqref{fo:pde_system1} at optimality was done under the assumption that the optimal control was bounded, the above proposition actually implies that the optimal control obtained from the gradient of the solution is necessarily bounded.
\end{remark}

\section{\textbf{Analysis of the Open Loop Conditional Control Problem}}
\label{sec:open-loop}
In this section, we consider the case of general open loop controls. Our first goal is to show that the search for optimal controls can be restricted to a subclass of controls of a specific feedback form, namely deterministic functions of time, the controlled state $X_t$, and a specific function of the history of its path prior to time $t$.

\subsection{Precise Formulation of the Open Loop Conditional Control Problem}

We now define in detail the  open loop version of the control problem discussed in the introduction.
For each $t\in[0,T]$, we denote by $\Theta_t$ the set of $\theta=(\Omega,\cF,(\cF_s)_{t\le s\le T}, \PP,\bW=(W_s)_{t\le s\le T})$ where $(\Omega,\cF,\PP)$ is a probability space supporting a process $\bW$ which satisfies $W_t=0$ and is a Brownian motion for the filtration $(\cF_s)_{t\le s\le T}$. 
When for some $t$, $\theta\in\Theta_t$ is a generic element, we implicitly use the notation  $\EE = \EE^{\PP}$ for the expectation under $\PP$, and by \emph{almost surely}, we mean $\PP$-almost surely.

For each $\theta\in\Theta_t$, we define $\AA_t^\theta$ as the set of $A$-valued
progressively measurable processes $\balpha=(\alpha_s)_{t\le s\le T}$ satisfying the admissibility conditions defined earlier, namely
$$
\EE^{\PP}\int_t^T|\alpha_s|^2\,ds\;<\infty.
$$ 
Now for each stochastic basis $\theta\in\Theta_t$
and each admissible control $\balpha\in\AA_t^\theta$, the controlled state process $\bX^{\alpha}=(X^\alpha_s)_{t\le s\le T}$ is given by:
\begin{equation}
\label{fo:state_t}
dX^\alpha_s=\alpha_s ds + dW_s,\qquad t\le s\le T.
\end{equation}
We shall specify the initial condition $X_t$ when needed, and we shall skip the superscript $\alpha$ whenever convenient as long as no confusion is possible. Also, we shall systematically drop the subscript $t$ in $\AA_t^\theta$ whenever $t=0$.

\vskip 6pt
The goal of the control problem is to minimize the cost $J^V(\balpha)$ 
defined in \eqref{fo:J_V_of_alpha}. We denote by $U^{(2)}(t,\mu)$ the value function of the problem. It is defined for $t\in[0,T]$ and $x\in\RR^d$ as follows:
\begin{equation}
\label{fo:U2(t,x)}
    U^{(2)}(t,\mu)=\inf_{\theta\in\Theta_t,\; \balpha\in\AA_t^\theta,\;X_t\sim\mu}
    \Bigl(\int_t^T
    \frac{\EE\Bigl[f(X_s,\alpha_s)e^{-\int_t^sV(X_u)du}\Bigr]}{\EE\Bigl[e^{-\int_t^sV(X_u)du}\Bigr]}ds
    +\frac{\EE\Bigl[g(X_T)e^{-\int_t^TV(X_s)ds}\Bigr]}{\EE\Bigl[e^{-\int_t^TV(X_s)ds}\Bigr]}\Bigr).
\end{equation}
The initial condition $X_t\sim\mu$ stipulating that the distribution of $X_t$ is $\mu$, reduces to the classical case of a deterministic initial condition when $\mu$ is a pointwise measure, say $\mu=\delta_x$ for some $x\in D$. 

\subsection{Technical Preliminary: the Mimicking Theorem}
\label{sub:mimicking}
The proof of the main result of this section is based on an application of the mimicking theorem, sometimes
referred to as the Markovian projection theorem. It is originally
due to Gyongy \cite{gyongy1986mimicking}. We shall make use of the extension proven by Brunick and Shreve in \cite{brunick2013mimicking}.

\vskip 2pt
We introduce the following notation to check easily the conditions required for the result of \cite[Theorem 3.6]{brunick2013mimicking} which we want to use. Recall that we use the standard notation $C([0,\infty);\RR^d)$ for the space of $\RR^d$-valued continuous function on $[0,\infty)$, and we use the notation $C_{0}([0,\infty);\RR^d)$ for the set of elements $x\in C([0,\infty);\RR^d)$ satisfying $x(0)=0$. We define the set $\cE$ to be the Euclidean space $\RR^d\times\RR$. Next we define the function $\Phi: \cE\times C_{0}([0,\infty);\RR^d) \to C([0,\infty);\cE)$ by:
$$
    \Phi\bigl((e_1,e_2),x\bigr)=\Bigl(e_1+x,e_2+\int_0^\cdot V(x_s)ds\Bigr).
$$
It should be clear that $e_1+x$ is the notation for the function $[0,\infty)\ni t\mapsto e_1+x(t)\in\RR^d$
and that $e_2+\int_0^\cdot V(x_s)ds$ is the notation for the function $[0,\infty)\ni t\mapsto e_2+\int_0^tV(X_s)ds\in\RR$. Also, for the sake of convenience, we shall often write $(e_1,e_2,x)$
instead of $((e_1,e_2),x)$. It is clear that the range of $\Phi$ is contained in the space $C([0,\infty);\cE)$ 
and that $\Phi$ is continuous because the potential function $V$ is continuous and bounded on $\RR^d$.
Next, we check that $\Phi$ is an \emph{updating function} in the sense of \cite[Definition 3.1]{brunick2013mimicking}. Obviously:
$$
[\Phi(e_1,e_2,x)](0)=\Bigl(e_1+x(0),e_2+\int_0^0 V(e_1+x(u))du\Bigr)=(e_1,e_2).
$$
If $x$ is any function on $[0,\infty)$, for any $t\ge 0$, we denote by $x^t$ the function stopped at time $t$, namely the function:
$$
[0,\infty)\ni s\mapsto x^t(s)=x(s\wedge t).
$$
In particular:
$$
[\Phi(e_1,e_2,x)]^t(s)=\Bigl(e_1+x(s\wedge t),e_2+\int_0^{s\wedge t} V(e_1+x(u))du\Bigr),
$$
which is equal to 
$$
[\Phi(e_1,e_2,x^t)]^t(s)=\Bigl(e_1+x^t(s),e_2+\int_0^{s\wedge t}V(e_1+x^t(u))du).
$$
Finally, we check the third property of  \cite[Definition 3.1]{brunick2013mimicking}. Indeed:
$$
[\Phi(e_1,e_2,x)](t+s)=\Bigl(e_1+x(t+s),e_2+\int_0^{t+s}V(e_1+x(u))du\Bigr)
$$
while by definition, the right hand side of the desired equality is equal to:
\begin{equation*}
\begin{split}
&\Bigl[\Phi\Bigl([\Phi(e_1,e_2,x)](t),x(t+\cdot)-x(t)\Bigr)\Bigr](s)\\
&\hskip 5pt
=\Bigl([\Phi(e_1,e_2,x)](t)_1+x(t+s)-x(t)),[\Phi(e_1,e_2,x)](t)_2+\int_0^sV([\Phi(e_1,e_2,x)](t)_1+x(t+u)-x(t))du\Bigr)\\
&\hskip 5pt
=(e_1+x(t+s),e_2+\int_0^t V(e_1+x(u))du+\int_0^sV(e_1+x(t)+x(t+u)-x(t))du\Bigr)
\end{split}
\end{equation*}
which is equal to $[\Phi(e_1,e_2,x)](t+s)$ as desired.

\subsection{Reformulation Based on a Special Class of Feedback Controls}
\label{sub:reformulation}

Let us consider $\theta \in \Theta$, let us fix momentarily an admissible control process $\balpha\in\AA^\theta$, let us denote by $\bX=(X_t)_{0\le t\le T}$ the associated controlled process $\bX^\alpha$ over the interval $[0,T]$, and let us set $A_t=\int_0^t V(X^\alpha_s)ds$.
We now use \cite[Theorem 3.6]{brunick2013mimicking} with the updating function $\Phi$ introduced above in Subsection \ref{sub:mimicking}.
We define the measurable function $\psi$ from $[0,T]\times\RR^d\times \RR$ with values in $A$ by:
\begin{equation}
\label{fo:alphahat}
\psi(t,x,a)=\EE[\alpha_t\;|\;X_t=x,\,A_t=a]
\end{equation}
as the expectation of $\alpha_t$ with respect to a regular version of the conditional probability given $(X_t,A_t)$. It is possible to choose a measurable version such that for $t\notin N$ where $N\subset[0,\infty)$ is of zero Lebesgue's measure we have $\psi(t,X_t,A_t)=\EE[\alpha_t\;|\;X_t,\,A_t]$ $\PP$-almost surely. Moreover, there exists a stochastic basis  $\hat\theta=(\hat\Omega,\hat\cF,\{\hat\cF_t\}_t,\hat\PP,\hat\bW)$ supporting processes $\hat\bX=(\hat X_t)_{0\le t\le T}$ and $\hat\bA=(\hat A_t)_{0\le t\le T}$ satisfying:
\begin{equation}
\label{fo:state_sde_psi}
\begin{cases}
\hat X_t&=\hat X_0+\int_0^t \psi(s,\hat X_s,\hat A_s)ds +\hat W_t\\
\hat A_t&=\int_0^t V(\hat X_s)ds
\end{cases}
\end{equation}
and such that for each $t\ge 0$, the joint law of the pair $(X_t,A_t)$ under $\PP$ in the original stochastic basis $\theta$ coincides with the joint law of $(\hat X_t, \hat A_t)$ under $\hat\PP$ in the new stochastic basis $\hat\theta$.
As a result:
\begin{equation}
\begin{split}
J^V(\balpha)
&=\int_0^T
\frac{\EE^\PP\Bigl[f(X_t,\alpha_t)e^{-A_t}\Bigr]}{\EE^\PP\Bigl[e^{-A_t}\Bigr]}dt
+\frac{\EE^\PP\Bigl[g(X_T)e^{-A_T}\Bigr]}{\EE^\PP\Bigl[e^{-A_T}\Bigr]}\\
&=\int_0^T
\frac{\EE^\PP\Bigl[\EE^\PP[f(X_t,\alpha_t)|X_t,A_t]e^{-A_t}\Bigr]}{\EE^\PP\Bigl[e^{-A_t}\Bigr]}dt
+\frac{\EE^\PP\Bigl[g(X_T)e^{-A_T}\Bigr]}{\EE^\PP\Bigl[e^{-A_T}\Bigr]}\\
&\ge\int_0^T
\frac{\EE^\PP\Bigl[f(X_t,\EE[\alpha_t|X_t,A_t])e^{-A_t}\Bigr]}{\EE^\PP\Bigl[e^{-A_t}\Bigr]}dt
+\frac{\EE^\PP\Bigl[g(X_T)e^{-A_T}\Bigr]}{\EE^\PP\Bigl[e^{-A_T}\Bigr]}\\
&=\int_0^T
\frac{\EE^{\hat\PP}\Bigl[f(\hat X_t,\psi_t(\hat X_t,\hat A_t))e^{-\hat A_t}\Bigr]}{\EE^{\hat\PP}\Bigl[e^{-\hat A_t}\Bigr]}dt
+\frac{\EE^{\hat\PP}\Bigl[g(\hat X_T)e^{-\hat A_T}\Bigr]}{\EE^{\hat\PP}\Bigl[e^{-\hat A_T}\Bigr]}\\
&=J^V(\hat{\balpha}),
\end{split}
\end{equation}
where the admissible control process $\hat{\balpha}=(\hat\alpha_t)_{0\le t\le T}$ is defined as $\hat\alpha_t=\psi(t,\hat X_t, \hat A_t)$, and where we used the convexity in $\alpha$ of $f(x,\alpha)$ for $x$ fixed, and the fact that the joint law of the pair $(X_t,A_t)$ under $\PP$ is the same as the joint law of $(\hat X_t, \hat A_t)$ under $\hat\PP$. This implies that:
$$
J^V(\balpha)\ge \inf_{\theta\in\Theta,\; \balpha^s\in\AA^{\theta,(S)}} J^V(\balpha^s)
$$
where, for a given stochastic basis $\theta\in\Theta$, we denote by $\AA^{\theta,(S)}$ the set of special admissible control processes defined as the subset of $\AA^\theta$ formed by the $\balpha$ of the form $\alpha_t=\psi(t,X_t,A_t)$ given by a feedback (measurable) function $\psi$ of the state process $\bX$ controlled by $\balpha$ and the process $\bA$ given by $A_t=\int_0^tV(X_s)ds$. Taking the infimum on the left hand side we get:
$$
    \inf_{\theta\in\Theta,\; \balpha\in\AA^{\theta}}  J^V(\balpha)\ge \inf_{\theta\in\Theta,\; \balpha^s\in\AA^{\theta,(S)}} J^V(\balpha^s)
$$
and since the left hand side is obviously not greater than the right hand side because $\AA^{\theta,(S)}\subset\AA^\theta$, we conclude that these two infima are identical. Since the same argument applies for all $t\in[0,T]$ with initial condition $X_t=x$, we conclude that the value function $U^{(2)}$ introduced earlier in \eqref{fo:U2(t,x)} can be computed by minimizing over control processes in feedback form given by deterministic functions of time and the couple $(X_t,A_t)$.

\vskip 2pt
To streamline the notations, in analogy with the analysis of the Markovian feedback functions, we shall denote by $\Psi$ the set of admissible extended feedback control functions, namely the $\RR^d$-valued measurable functions $\psi$ on $[0,T]\times\RR^d\times\RR_+$ for which the stochastic differential system \eqref{fo:state_sde_psi} has a weak solution satisfying
$$
\EE\int_0^T|\psi_t(X_t,A_t)|^2 dt<\infty,
$$
and we shall use the notation $J^V(\psi)$ for $J^V(\balpha)$ when the control process $\balpha$ is given by $\alpha_t=\psi_t(X_t,A_t)$.
We have thus proved the following result, which is a cornerstone of our analysis from now on. 
\begin{theorem}
    \label{thm:openloop-to-newfeedback}
    It holds:
    $$
    \inf_{\theta\in\Theta,\; \balpha\in\AA^{\theta}}  J^V(\balpha)
    = \inf_{\psi\in\Psi} J^V(\psi).
$$
\end{theorem}

\subsection{Reformulation as a Deterministic Control Problem}
\label{sec:formulation-open-deterministic}
From now on, we limit the open loop optimization problem to open loop control processes $\balpha$ of the form  
$\alpha_t=\psi_t(X_t,A_t)$ for measurable feedback functions $\psi\in\Psi$ of time and the couple $(X_t,A_t)$. 
Given such a feedback function $\psi$ the definition we chose of the set $\Psi$ of admissible extended feedback functions includes the existence of a corresponding controlled process, in other words, a solution of the stochastic differential equation
\begin{equation}
\label{fo:psi_sde}
\begin{cases}
d X_t&=\psi_t(X_t,A_t)dt +d W_t\\
d A_t&= V(X_t)dt,
\end{cases}
\end{equation}
with the same initial condition for $X_0$ and $A_0=0$.
At first glance, this system of stochastic differential equations appears to be degenerated and existence of strong solutions is not guaranteed. However, one can look at the first of the above equations as an equation of the form $dX_t=b(t,X_\cdot)dt + dW_t$ with a  progressively measurable drift $b$ which depends upon the past of the trajectory, and existence and uniqueness of weak solutions for those equations is guaranteed for example when $\psi$ is bounded. In fact, given the analysis of the Markov feedback case performed in the previous section, we shall limit ourselves in this section to bounded open loop control processes $\balpha$, and hence to bounded measurable extended feedback control functions $\psi$. This is all we need since we already appealed to the theory of weak solutions to reduce the open loop optimization problem to what we are considering now, and for which
the cost function $\balpha\mapsto J^V(\balpha)$ which depends upon the values of the couple $(X_t,A_t)$ and its distribution can be rewritten in the form:
\begin{equation}
\label{fo:J_of_psi}
\begin{split}
    J^V(\psi)
    &=\int_0^T
    \frac{\EE\Bigl[f(X_t,\psi_t(X_t,A_t))e^{-A_t}\Bigr]}{\EE\Bigl[e^{-A_t}\Bigr]}dt
    +\frac{\EE\Bigl[g(X_T)e^{-A_T}\Bigr]}{\EE\Bigl[e^{-A_T}\Bigr]}\\
    &=\int_0^T
    \Bigl(\int \mu^{(2)}_t(dx,da)f(x,\psi(x,a))\Bigr)dt +\int \mu^{(2)}_T(dx,da)g(x)\\
    &=\int_0^T <\mu^{(2)}_t,f(\cdot,\psi(\cdot,\cdot))>dt +<\mu^{(2)}_T,g>,
 \end{split}
\end{equation}
where we use the notation $\mu^{(2)}_t$ for the Gibbs probability measure:
\begin{equation}
\label{fo:mu_2_t}
    \mu^{(2)}_t(dx,da)=\frac{\EE[\delta_{(X_t,A_t)}(dx,da) e^{- A_t}]}{\EE[ e^{- A_t}]},\qquad 0\le t\le T.
\end{equation}

\begin{lemma}
\label{le:FPK2}
The measure valued function $t\mapsto \mu^{(2)}_t$ satisfies the forward FPK equation:
\begin{equation}
\label{fo:kolmo2}
\partial_t \mu =\frac12\Delta_x \mu - \div_x(\psi_t \mu) -V\partial_a\mu - (V-<\mu,V>) \mu,
\end{equation}
with initial condition $\mu_{|t=0}(dx,da)=\mu_0(dx)\delta_0(da)$ in the sense of Schwartz distributions.
 \end{lemma}
Now, the notation $<\mu,V>$ used in formula \eqref{fo:kolmo2} should be understood as $<\mu,V>=\int_{\RR^d}\int_{\RR} V(x)\mu(dx,da)$ or $<\mu,V>=<\tilde\mu,V>=\int_{\RR^d} V(x)\tilde\mu(dx)$ if we use the notation $\tilde\mu$ for the first marginal of $\mu$.
\begin{proof}
If $\varphi$ is a smooth function on $\RR^d\times \RR$ with compact support, It\^o's formula gives:
\begin{equation*}
\begin{split}
\frac{d}{dt} <\mu^{(2)}_t,\varphi> & = \frac{d}{dt}\frac{\EE[\varphi(X_t,A_t) e^{- A_t}]}{\EE[ e^{- A_t}]}\\
& =\frac{1}{\EE[ e^{- A_t}]}
\EE\Bigl[\Bigl(\frac12\Delta_x \varphi(X_t,A_t)  +\psi_t(X_t,A_t)\cdot\nabla_x\varphi(X_t,A_t)+\partial_a\varphi(X_t,A_t)V(X_t)\Bigr)  e^{- A_t}\\ 
&\hskip 75pt
 - \varphi(X_t,A_t) V(X_t)  e^{- A_t} \Bigr]+\EE[\varphi(X_t,A_t) e^{- A_t}]\frac{\EE[V(X_t)e^{-A_t}]}{\EE[ e^{- A_t}]^2}\\
& = <\mu^{(2)}_t,\; \frac12\Delta_x \varphi  +\psi_t \cdot \nabla_x\varphi +V\partial_a\varphi- (V -<\mu^{(2)}_t,V>)\varphi >\\
& = <\frac12\Delta_x \mu^{(2)}_t - \div_x(\psi_t \mu^{(2)}_t) - V\partial_a\mu^{(2)}_t-(V-<\mu^{(2)}_t,V>) \mu^{(2)}_t,\; \varphi>
\end{split}
\end{equation*}
where we used stochastic integration by parts and the fact that $\varphi$ has compact support.
\end{proof}

For the purpose of our analysis, we resolve the existence problem for the non-local FPK equation \eqref{fo:kolmo2} as before. Given  $\psi\in\Psi$, Lemma~\ref{le:FPK2} guarantees that the flow of probability measures $\mu_t=\mu^{(2)}_t$ defined by \eqref{fo:mu_2_t} with $A_t=\int_0^tV(X_s)ds$ and $X_t$ being a weak solution of the stochastic differential equation
$$
dX_t=\psi_t\Bigl(X_t,\int_0^tV(X_s)ds\Bigr)dt + dW_t
$$
is a solution of \eqref{fo:kolmo2}. This is all we shall need for the existence of solutions of \eqref{fo:kolmo2}.

\vskip 2pt
The proof of Theorem \ref{th:superposition} did not depend upon the fact that the state stochastic differential equation was not degenerate, so a similar superposition principle holds in the present situation, and our original optimization problem reduces to the deterministic control problem of the minimization of the functional:
\begin{equation}
J(\psi)=\int_0^T F^{(2)}(\mu^{(2)}_t,\psi_t)dt + G^{(2)}(\mu^{(2)}_T)
\end{equation}
under the dynamical constraint \eqref{fo:kolmo1}, where the running and terminal cost functions $F^{(2)}$ and $G^{(2)}$ are given by:
\begin{equation}
\label{fo:second_costs}
F^{(2)}(\mu,\psi)= \int\int \mu(dx,da)f(x,\psi(x,a)), \quad\text{and}\quad G^{(2)}( \mu)=\int \int\mu(dx,da)g(x) ,
\end{equation}
for $\mu\in\cM(\RR^d\times\RR_+)$ the space of finite measures on $\RR^d\times\RR_+$. Note that as per our discussion of the integral of the potential $V$, the terminal cost $G^{(2)}( \mu)$ only depends upon the first marginal of the measure $\mu$.

\vskip 4pt
The following lemma guarantees the existence of an optimal control for the problem at hand. It is proven exactly in the same way as  Proposition~\ref{pr:existence_control1}, so we do not repeat the proof.

\begin{lemma}
\label{le:existence_control2}
There exists an optimal admissible feedback function $\hat\psi$.
\end{lemma}

\vskip 6pt
This type of problem is usually approached by computing the value function of the problem as a solution of an HJB equation written in terms of the Hamiltonian of the problem once minimized over the admissible controls.
For later purposes we note that:
\begin{equation}
\label{fo:partial_F2}
\frac{\delta F^{(2)}}{\delta \mu}(\mu,\psi)(x,a)=f(x,\psi(x,a))
\qquad\text{and}\qquad 
\frac{\delta G^{(2)}}{\delta \mu}(\mu)(x,a)=g(x).
\end{equation}

\vskip 2pt\noindent
For $\mu\in\cM(\RR^d\times\RR_+)$, $\varphi\in\cC_c^\infty(\RR^d\times\RR_+)$ the space of infinitely differentiable functions with compact support in $\RR^d\times\RR$, and $\psi\in\cB(\RR^d\times\RR_+;A)$ the space of bounded measurable functions on $\RR^d\times\RR_+$, we define the Hamiltonian $\HH^{(2)}$ by:
\begin{equation}
\label{fo:H2}
\begin{split}
    \HH^{(2)}(\mu,\varphi,\psi)
    &=<\frac12\Delta_x \mu - \div_x(\psi \mu) - V  \partial_a\mu- (V-<\mu,V>) \mu,\;\varphi> +  F^{(2)}(\mu,\psi)\\
    &=<\mu,\frac12\Delta_x \varphi + \psi \cdot\nabla_x\varphi + V  \partial_a\varphi - (V-<\mu,V>)\varphi+  f(\cdot,\psi(\cdot,\cdot))>.
\end{split}
\end{equation}

\subsubsection{\textbf{The Adjoint PDE}}
\label{sub:adjoint1}

Let us assume that $\psi$ is an admissible feedback control function and let $\bmu=(\mu_t)_{0\le t\le T}$ satisfies the associated FPK equation \eqref{fo:kolmo2}. We define the corresponding adjoint variable as the function $u$ 
solving the PDE $\partial_t u=-\partial_\mu\HH^{(2)}(\mu,u,\psi)$ with terminal condition $u_T(x,a)=\partial_\mu G(x)= g(x)$. In the present situation, this adjoint equation reads:
\begin{equation}
\label{fo:adjoint2}
    \partial_t u = -  \frac12\Delta_x u - \psi_t \cdot\nabla_x u-V\partial_a u+(V-<\mu,V>)u -V<\mu,u>- f\bigl(\cdot,\psi_t(\cdot,\cdot)\bigr),
\end{equation}
which can be rewritten as
\begin{equation}
\label{fo:adjoint2'}
    0=\partial_t u + \frac12\Delta_x u+\psi_t\cdot\nabla_x u+V\partial_a u-(V-<\mu,V>)u+V<\mu,u>+\frac12|\psi_t|^2+\tilde f
\end{equation}
in the case of separable running cost function as defined in \eqref{fo:separable_running_cost}. 

\begin{lemma}
\label{le:existence_adjoint2}
For each bounded admissible feedback function $\psi\in\Psi$, if $\bmu=(\mu_t)_{0\le t\le T}$ is the associated flow of probability measures given as a solution of the FPK equation \eqref{fo:kolmo2}, the adjoint PDE \eqref{fo:adjoint2} admits a solution in the sense of viscosity.
\end{lemma}

\begin{proof}
The proof is exactly the same as the proof of Lemma~\ref{pr:existence_adjoint1}, so we refrain from giving it. 
\end{proof}

\subsubsection{\textbf{Minimization of the Hamiltonian.}}
\label{sub:minimization}
For the sake of convenience, the following assumption will make it easier to refer to the minimized Hamiltonian.
  
\begin{assumption}
We assume the existence of a function $\hat\psi: \cM(\RR^d \times \RR)\times\cC_c^\infty(\RR^d \times \RR)\ni(\mu,\varphi)\mapsto \hat\psi[\mu,\varphi]\in\cB(\RR^d \times \RR;A)$
such that:
$$
    \forall (\mu,\varphi)\in\cM(\RR^d \times \RR)\times\cC_c^\infty(\RR^d \times \RR),\quad \hat\psi[\mu,\varphi]\in\arginf_{\psi\in\cB(\RR^d \times \RR;A)}\HH^{(2)}(\mu,\varphi,\psi).
$$
\end{assumption}

In the important case of interest given by separable running cost functions of the form \eqref{fo:separable_running_cost}, we can compute explicitly such a function $\hat\psi$.  
Indeed, using \eqref{fo:H2} we find:
\begin{equation}
\begin{split}
    \HH^{(2)}(\mu,\varphi,\psi)&=<\mu,\frac12\Delta_x \varphi + \psi \cdot\nabla_x\varphi + V  \partial_a\varphi - (V-<\mu,V>)\varphi+  f(\cdot,\psi(\cdot,\cdot))>\\
    &=\int\hskip -6pt\int \mu(dx,da)\Bigl[\frac12\Delta_x \varphi(x,a) + V(x) \partial_a \varphi(x,a)  - (V(x)-<\mu,V>)\varphi(x,a)+\tilde f(x)\\
    &\hskip 35pt
+ \nabla_x\varphi(x,a) \cdot \psi(x,a) +\frac12 |S\psi(x,a)|^2 \Bigr]
\end{split}
\end{equation}
and we can minimize under the integral signs (i.e. for $x$ and $a$ fixed) leading to the minimizer (assuming for simplicity that $A = \RR^d$):
\begin{equation}
\label{fo:psi_opt}
    \psi[\mu,\varphi](x,a)=- \nabla_x\varphi(x,a)
\end{equation}
$\mu$ almost everywhere, and also on the support of the measure $\mu$ because of the smoothness of the test function $\varphi$. As in the case of Markovian feedback functions, we notice that the above argument only requires the differentiability of the test function in the $x$ variable.

\subsubsection{\textbf{The Forward-Backward PDE System}}
\label{sec:fbpde-open}
As in the case of Markov feedback control functions considered in Subsection \ref{sub:MP}, we can derive an infinite dimensional version of the necessary condition of the classical Pontryagin maximum principle similar to Proposition \ref{pr:MP}, and conclude that if $\hat\psi$ is an optimal control which happens to be bounded, if $\hat\mu$ is the corresponding controlled state and if $\hat u$ is a solution of the corresponding adjoint equation, then we can check as before that the partial derivatives of $\hat u$ with respect to $x$ in the sense of distributions are actually bounded measurable functions and that:
$$
\HH^{(2)}(\hat\mu_t,\hat u_t,\hat\psi_t)=\inf_\psi \HH^{(2)}(\hat\mu_t,\hat u_t,\psi),\qquad\qquad 0\le t\le T.
$$
So in the case of a separable running cost function, formula \eqref{fo:psi_opt} says that the optimal control $\hat\psi$  (recall that its existence is guaranteed by Lemma~\ref{le:existence_control2}) must satisfy $\hat\psi_t(x,a)=-\nabla_x\hat u_t(x,a)$, and injecting this formula in the adjoint equation and the FPK equation we see that the couple $(\hat\mu,\hat u)$ is a solution of the forward backward PDE system on $[0,T]\times\RR^d\times\RR_+$:
\begin{equation}
\label{fo:pde_system2}
\begin{cases}
    &\partial_t\mu=\frac12 \Delta_x\mu +\div_x(\nabla_xu\;\mu) - V\partial_a\mu -(V-<\mu,V>)\mu\\
    &0=\partial_t u + \frac12\Delta_x u- \frac12 |\nabla_x u|^2+V\partial_a u-(V-<\mu,V>)u +V<\mu,u>+\tilde f.
\end{cases}
\end{equation}
As in the previous section, we assume that one of the optimal feedback functions whose existence we know, is bounded, and we work with such an optimal feedback control function.
Following the same steps, we argue below that the corresponding solution of the adjoint equation has a gradient which can be identified with the negative of the optimal control thanks to the appropriate version of the Pontryagin maximum principle. Notice that while we do not claim uniqueness for the solutions of the PDE system \eqref{fo:pde_system2}, we know that if we have a bounded optimal feedback control, then existence of a solution is guaranteed.

\vskip 2pt
\begin{proposition}
\label{pr:boundedness_for_a}
Let $\hat\psi$ be an optimal feedback control function which is bounded, let $\hat\bmu$ be the corresponding solution of the FPK equation \eqref{fo:kolmo2}, and let $\hat u$ be the solution of the corresponding adjoint equation. Then,
\begin{itemize}
    \item[(i)]  the unique viscosity solution of the second equation of \eqref{fo:pde_system2} depends only upon the space variable $x$ and is independent of $a$;
    \item[(ii)] furthermore, if $\tilde f\in C^3_b(\RR^d)$ the second equation in \eqref{fo:pde_system2} has a unique classical solution and the negative of its gradient is an optimal feedback control.
\end{itemize}
\end{proposition}

\begin{proof}
Unsurprisingly, the proof follows the same steps as before.

\vskip 2pt\noindent
\emph{Step 1.} Let $\hat\psi$ be an optimal feedback control function whose the existence  is given by Lemma~\ref{le:existence_control2}, and let us assume that $\hat\psi$ is bounded. Let $\hat\bmu=(\hat\mu_t(dx,da))_{0\le t\le T}$ be the corresponding flow of probability measures satisfying the FPK equation \eqref{fo:kolmo2}, and let $(t,x,a)\mapsto \hat u_t(x,a)$ be the solution of the corresponding adjoint equation whose existence is given by Lemma~\ref{le:existence_adjoint2}. Like in Remark \ref{re:stronger_apriori_bound}, we claim that the existence proof of an optimum can be "massaged" to give an a-priori bound, say $K>0$, on $\|\hat u\|_\infty$. 

Notice that in the present section, the diffusion matrix of the second equation in \eqref{fo:pde_system2} is degenerate since there is no second order derivative in $a$. This means that we will not be able to use directly \cite[Theorem 3.1]{MaZhang} or its corollary, though we are still able to use the results of \cite{PardouxPeng} in the present setting.

\vskip 2pt\noindent
\emph{Step 2.} Pontryagin's maximum principle suggests that we must have for each $t\in[0,T]$, 
$$
\HH^{(2)}(\hat\mu_t,\hat u_t,\hat\psi_t)=\inf_{\psi\in\Psi}\HH^{(2)}(\hat\mu_t,\hat u_t,\psi).
$$
So, since we restrict ourselves to $A=\RR^d$ and a separable running cost, one can show that if $\hat u$ is differentiable, then we can use integration by parts in the definition of $\HH^{(2)}(\hat\mu_t,\hat u_t,\psi)$, and use the minimizer $\hat\psi[\hat\mu_t,\hat u_t]$ identified in \eqref{fo:psi_opt}. So for $dt\,\hat\bmu(dx,da)$ - almost every $(t,x,a)\in[0,T]\times\RR^d\times\RR_+$ one has:
$$
\hat\psi_t(x,a)=-\nabla_x \hat u_t(x,a).
$$

\vskip 2pt\noindent
\emph{Step 3.} In the spirit of the proof of Proposition \ref{pr:existence_pde1}, we introduce an infinitely differentiable function $\chi:\RR\mapsto\RR$ satisfying $\chi(y)=y$ if $e^{-K_u}\le y\le e^{K_u}$, $\chi(y)=0$ if $y<e^{-K_u}/2$ or $y>2e^{K_u}$, and $|\chi'(y)|\le 2$ for all $y$ and $\tilde V_1(t,x)=V(x)-<\hat\mu_t,V>$.  We also define the function $\hat F(t,x)$ by
$$
 \hat F(t,x)= \tilde f(x) + V(x)<\hat \mu_t,\hat u_t>,\qquad (t,x)\in[0,T]\times\RR^d.
$$
Next, we define the real valued function
$f$ on $[0,T]\times(\RR^d\times\RR_+)\times\RR\times\RR$ by:
\begin{equation}
    \label{fo:driver2}
f(t,(x,a),y,z)=-\tilde V^1(t,x)\chi(y)\log\chi(y)+\hat F(t,x)\chi(y).
\end{equation}
which we use as driver for our main BSDE. Accordingly, we consider the semilinear PDE
\begin{equation}
    \label{fo:pde_for_v3}
0=\partial_t v + \frac12\Delta_x v + V\partial_a v - \tilde V^1 \chi(v)\log \chi(v) - \hat F\; \chi(v) 
\end{equation}
with the terminal condition $v_T(x,a)=e^{-g(x)}$. Notice that, contrary to the proof of Proposition~\ref{pr:existence_pde1},  $\hat\mu_t$ and $\hat u_t$ entering the definitions of $\tilde V^1$ and $\hat F$, are now functions of both $x$ and $a$. However, despite this remark, the functions 
$\tilde V^1$ and $\hat F$ are only functions of $t$ and $x$. 

To obtain a solution for~\eqref{fo:pde_for_v3}, we apply~\cite[Theorem 4.3]{PardouxPeng} on a regularized version: Using $\sigma=[\sigma^{ij}]_{i,j=1,\cdots,d+1}$ with $\sigma^{ij}=0$ if $i\ne j$, $\sigma^{ii}=1$ if $1\le i\le d$ and $\sigma^{d+1,d+1}=2\epsilon$, and $b^i=0$ if $1\le i\le d$ and $b^{d+1}=V$, we can still use \cite[Theorem 4.3]{PardouxPeng} and deduce that the solution of the FBSDE with driver $f$ given in \eqref{fo:driver2} provides us with the unique viscosity solution $v$ of \eqref{fo:pde_for_v3}, and that the same function $v$ is the unique classical solution of \eqref{fo:pde_for_v3} when $\tilde f$ is $C^3_b(\RR^d)$.

\vskip 2pt
At this stage, we emphasize that this unique viscosity solution is also an entropy solution of \eqref{fo:pde_for_v3}.
Indeed, if for each $\epsilon>0$ we consider the PDE
\begin{equation}
    \label{fo:pde_for_v2_epsilon}
0=\partial_t v^\epsilon + \frac12\Delta_x v^\epsilon + \epsilon\partial^2_{aa}v^\epsilon + V\partial_a v^\epsilon - \tilde V^1\, \chi(v^\epsilon)\log \chi(v^\epsilon) - \hat F\, \chi(v^\epsilon) 
\end{equation}
with the same terminal condition $v^\epsilon_T(x,a)=e^{-g(x)}$, \cite[Assumption (A1) p.1394]{MaZhang} is now satisfied, and since \cite[Assumption (A2) p.1394]{MaZhang} is also satisfied, we can now use \cite[Theorem 3.1 p.1397]{MaZhang} and \cite[Corollary 3.2 p.1403]{MaZhang} to conclude that the unique viscosity solution $v^\epsilon$ of \eqref{fo:pde_for_v2_epsilon} is continuously differentiable in the space variable, and that $\nabla_x v^\epsilon_t(x)$ is uniformly bounded. At this stage, we can already notice that since neither the terminal condition nor the coefficients $\tilde V^1$ and $\hat F$ depend upon the variable $a$, it follows that for each $\epsilon>0$, $v^\epsilon_t(x,a)$ is in fact independent of the variable $a$.
The bound on $\|\nabla_x v^\epsilon\|_\infty$ provided by \cite[Corollary 3.2]{MaZhang} does not depend upon $\epsilon$, which implies that the family $(v^\epsilon)_{\epsilon>0}$ is equi-continuous, and if we denote by $v$ the limit of any convergent sub-sequence when $\epsilon\searrow 0$, $v$ is a vanishing viscosity solution of the PDE \eqref{fo:pde_for_v3}, and it is in fact the unique entropy solution of \eqref{fo:pde_for_v3}.

\vskip 2pt\noindent
\emph{Step 4.} Using again a form of the Feynman-Kac formula as in the proof of Proposition~\ref{pr:existence_pde1}, we can show that $v=e^{-\hat u}$ and that $u$ is the unique viscosity solution of the second equation in \eqref{fo:pde_system2}, and the unique entropy solution of this equation as well. As before, these viscosity solutions are in fact classical solutions when $\tilde f\in C^3_b(\RR^d)$.
Since for each $\epsilon>0$, $v^\epsilon_t(x,a)$ is in fact independent of the variable $a$. This implies that the functions $v$ and $u$ do not depend upon the variable $a$ either.
\end{proof}

\vskip 4pt
Next, we state and prove the result which motivated the analysis of the paper. It is a plain consequence of what was proven above and in the previous section.

\begin{theorem}
\label{th:equality}
If we assume that one of the optimal feedback control functions is bounded, then the closed loop and open loop optimization problems with soft conditioning given by a bounded potential $V$ share a common optimal feedback control function, and consequently, their values are equal.   
\end{theorem}

\begin{proof}
If $(\hat\mu_t,\hat u_t)_{0\le t\le T}$ is the couple analyzed in Proposition~\ref{pr:boundedness_for_a} above from a bounded feedback function $\hat \psi$, we know that $\hat\psi$ does not depend upon the variable $a$ because $\hat u$ does not. So $\hat\psi$ is in fact a Markovian feedback control function, identifying the infima over closed loop (Markovian) and open loop controls.
\end{proof}

\begin{remark}
Also note that if for each $t\in[0,T]$, we denote by $\hat\mu^1_t$ the first marginal of $\hat\mu_t$, then the couple $(\hat\mu^1_t,\hat u_t)_{0\le t\le T}$
solves the forward-backward PDE system \eqref{fo:pde_system1}.
Indeed, since $V$ is independent of $a$, $<V,\hat\mu_t> = \int \int V(x) \hat \mu_t(dx,da)  = \int V(x) \hat \mu^1_t(dx)  = <V,\hat\mu^1_t>$. The same applies to $\hat F$ once we know that $\hat u_t(x,a)$ is independent of $a$. 
\end{remark}

\begin{remark}
\label{re:Daudin}
A direct argument (provided to us by Samuel Daudin) can be used if we are merely interested in the equality of the infima over the classes of Markov and extended feedback control functions. Indeed, using the notations of the above proof, if $\mu_t(dx,da)$ is the solution of the FPK equation for some extended feedback control $\psi$, and if we denote by $\mu_t(dx,da)=\mu_t(x,da)\mu^1(dx)$ its desintegration against its first marginal $\mu^1(dx)$, then $\phi_t(x)=\int_{[0,\infty)}\psi_t(x,a)\mu_t(x,da)$ is an admissible Markovian feedback control, $\bmu^1=(\mu^1_t)_{0\le t\le T}$ solves the corresponding FPK equation, and Schwarz inequality implies $J^V(\phi)\le J^V(\psi)$ which in turn, implies equality of the infima. 
\end{remark}

\begin{remark}
When the convex set $A$ of control values is not necessarily equal to the whole space $\RR^d$, the right hand side of \eqref{fo:psi_opt} needs to be replaced by its projection on the convex set $A$, namely the operator $\Pi_A$ defined by \eqref{fo:Pi_A}.
With the use of the operator $\Pi_A$, the form of the forward-backward PDE system is more involved, so we refrained from having to rely on the projection $\Pi_A$.
\end{remark}

\section{Numerical Experiments}
\label{sec:numerics}

For the purpose of numerical illustration we consider the following setting: $f(x,\alpha) = \frac{1}{2}|\alpha|^2$, $g(x) = |x-x^\star|$ for some fixed $x^\star$ which is interpreted as a target position, and $D$ is a ball centered at $0 \in \RR^d$ and of radius $R$. We focus on the one-dimensional and two-dimensional cases (i.e., $d=1,2$). We compare the optimal values obtained with two classes of controls: Markovian ƒeedback controls, which are functions of $(t,X_t)$, as well as functions of $(t,X_t,A_t)$. Based on our analysis (see Section~\ref{sub:reformulation}), optimizing over the latter class is equivalent to optimizing over the class of open-loop controls. For this class of controls, the state space is $\RR^{d+1}$, which is three-dimensional when $d=2$. For this reason, we rely on neural network-based methods to learn the optimal control. In the spirit of the algorithm analyzed in~\cite{carmona2022convergence} for McKean-Vlasov control, we approximate the control by a neural network $\phi_\theta$ with parameters $\theta$ which takes as inputs $(t,x) \in \RR^{d+1}$ (resp. $(t,x,a) \in \RR^{d+2}$) for the first (resp. second) class of controls. The optimal control problem becomes an optimization problem, which consists in minimization over $\theta$ the loss function defined as the total expected cost when using control $\phi_\theta$. We approximate the expectation in the denominator and the numerator using Monte Carlo samples. Furthermore, we discretize time using an Euler-Maruyama scheme over a uniform grid in time $t_n = n \Delta t$, with $\Delta t = T/N_T$ for some positive integer $N_T$. This leads to the following problem, which is an approximation of~\eqref{fo:J_V_of_alpha}: 
\begin{equation}
\label{fo:J_V_of_alpha_numerics}
    \tilde{J}^V(\theta)
    = \EE\left[\frac{1}{N} \sum_{i=1}^N \left(\sum_{n=0}^{N_T-1}
    \frac{f(X^i_{t_n}, \phi_{\theta}(t_n, X^i_{t_n})) e^{-A^i_{t_n}}}{\frac{1}{N} \sum_{j=1}^N e^{-A^j_{t_n}}} \Delta t
    + \frac{g(X^i_T)e^{-A^i_{T}}}{\frac{1}{N} \sum_{j=1}^N e^{-A^j_{T}}} \right)\right],
\end{equation}
subject to the dynamics:
\begin{equation}
\label{fo:traj_numerics}
    \begin{cases}
    X^i_0 = x_0,  &X^i_{t_{n+1}} = X^i_{t_{n}} + \phi_{\theta}(t_n, X^i_{t_n}) \Delta t + \sigma (W^i_{t_{n+1}} - W^i_{t_{n}}), 
    \\
    A^i_0 = 0, &A^i_{t_{n+1}} = A^i_{t_{n}} + V(X^i_{t_{n}}) \Delta t, \quad n = 0,\dots,N_T-1,
    \end{cases}
\end{equation}
where the $W^i$ are independent $d$-dimensional Brownian motions. The average over $j$ in the denominators are used to approximate expectations. The average over $i$ inside the expectation is superfluous since $X^i$ are i.i.d. However, we write the cost in this way since it is closer to the numerical implementation. Indeed, to optimize over $\theta$ we use stochastic gradient descent (SGD) (or rather one of its variants) and at each iteration we simulate $N$ trajectories $(X^i,A^i)_{i=1,\dots,N}$ following~\eqref{fo:traj_numerics}, compute the expression inside the expectation in~\eqref{fo:J_V_of_alpha_numerics}, and use its gradient with respect to $\theta$ to do one gradient descent step. 

\vskip 6pt
\paragraph{\bf Equality of the value functions. } 
In the one dimensional case $d=1$, we take $x^\star = 0.0$ and we consider $x_0 \in \{-1.25$ $-1.0,$ $-0.75,$ $-0.5,$ $-0.25,$ $0.0,$ $0.25,$ $0.5,$ $0.75,$ $1.0,$ $1.25\}$. For each value of $x_0$, we (approximately) compute the optimal value functions for both types of controls using the deep learning method described above. We use $N=100$ during training, and then we use $N=1000$ for testing, i.e., to compute the values reported below once the neural networks have been trained. We repeat the training and testing 5 times and report in the left pane of Figure~\ref{fig:values} the average value (solid and dashed lines). 
We note that, for a given value of $v$, the values for the two classes of controls match very well. Furthermore, when $V$ increases, the value decreases. This seems to be consistent with the fact that when $V$ tends to infinity, we expect to recover the original problem with stopping time (see Section~\ref{sec:conditional-exit-stopping}). In this latter case, at least for feedback controls, the value function satisfies an HJB equation inside the domain with Dirichlet boundary condition at the boundary (see~\cite{achdou2021optimal} for more details).  

\vskip 2pt
In the two dimensional case $d=2$, we take $x^\star = (0.0,0.0)$ and we consider $x_0 \in \{(-1.0,-1.0),$ $(-0.75,-0.75),$ $(-0.5,-0.5),$  $(-0.25,-0.25),$$(0.0,0.0),$ $(0.25,0.25),$ $(0.5,0.5),$ $(0.75,0.75),$ $(1.0,1.0)\}$. For each value of $x_0$, we (approximately) compute the optimal value functions for both types of controls using the deep learning method described above. We use $N=200$ during training, and then we use $N=1000$ for testing, i.e., to compute the values reported below once the neural networks have been trained. We repeat the training and testing 5 times and report in the right pane of Figure~\ref{fig:values} the average value (solid and dashed lines). 
We can observe the same phenomenon as in the 1D case, although the values are different due to the 2D structure. 

\vskip -4pt
\begin{figure}[h]
\centerline{
\includegraphics[width=8cm,height=4.5cm]{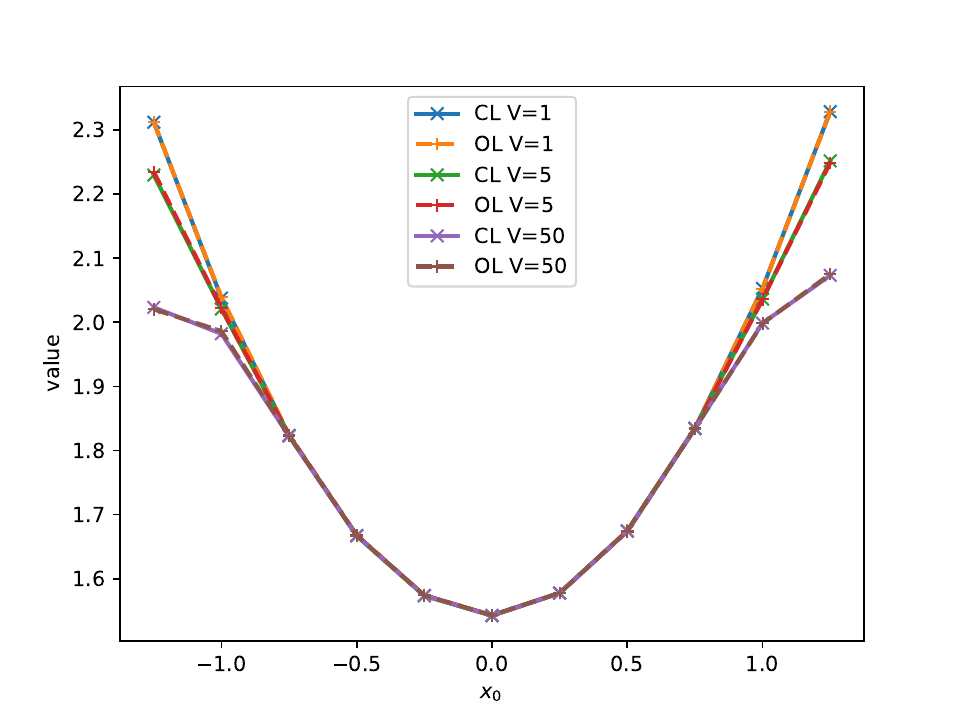}
\hskip 2pt
\includegraphics[width=7.5cm,height=4.5cm]{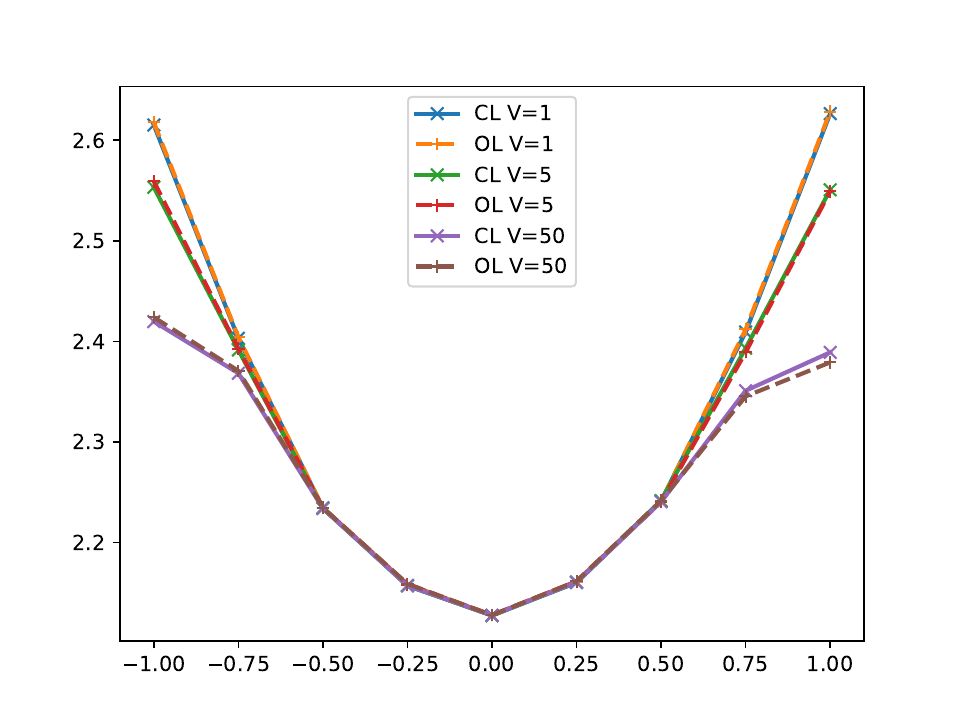}
}
\vskip -6pt
\caption{\label{fig:values}Value of the total cost for three values of $V$ outside the domain ($V=0$ inside the domain), $d=1$ (left pane), $d=2$ (right pane). The full lines correspond to the closed-loop Markovian case (label ``CL'') and the dashed lines correspond to the open-loop case (labeled ``OL'').}
\end{figure}

\vskip -6pt
\begin{figure}[h]
\centerline{
\includegraphics[width=6cm,height=4.5cm]{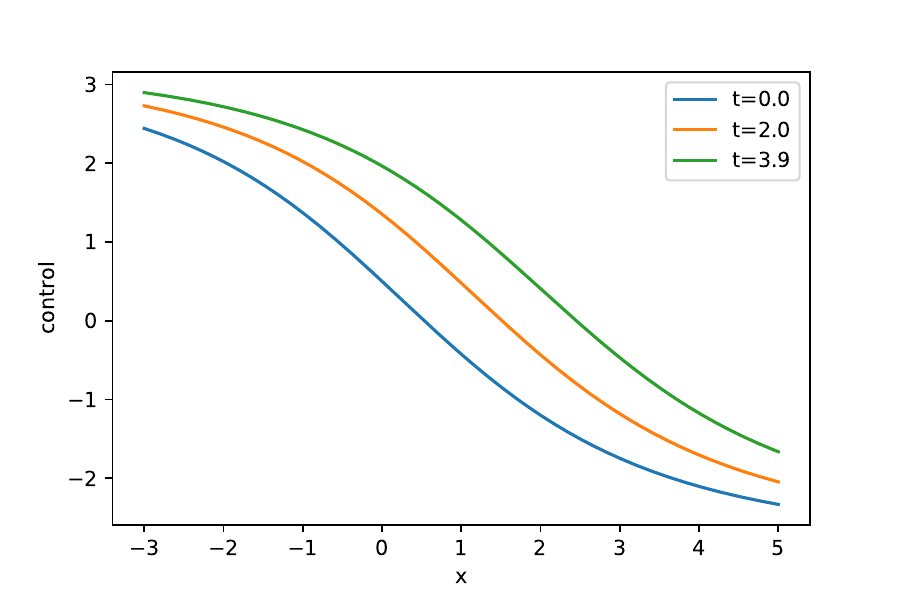}
\hskip 2pt
\includegraphics[width=6cm,height=4.5cm]{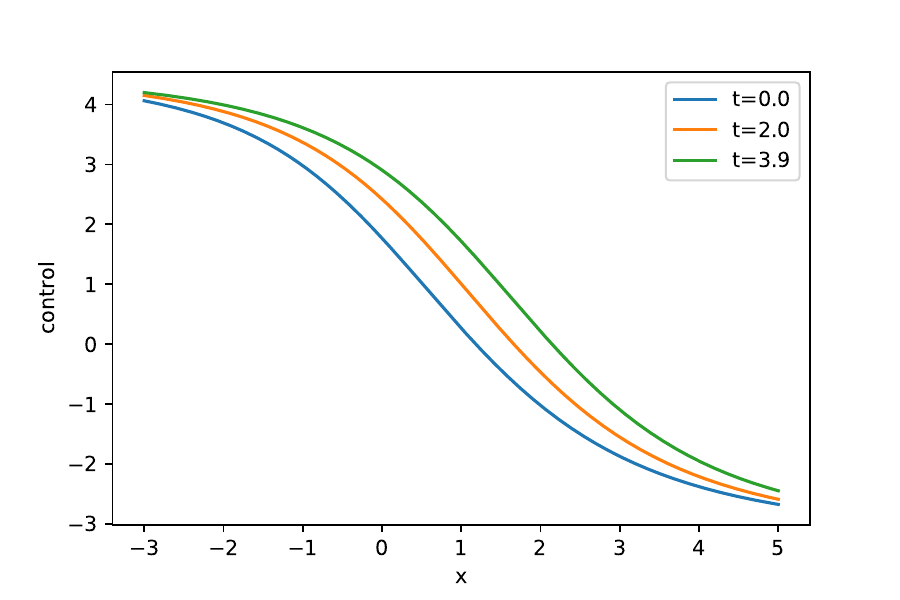}
}
\vskip -4pt
\centerline{
\includegraphics[width=6cm,height=4.5cm]{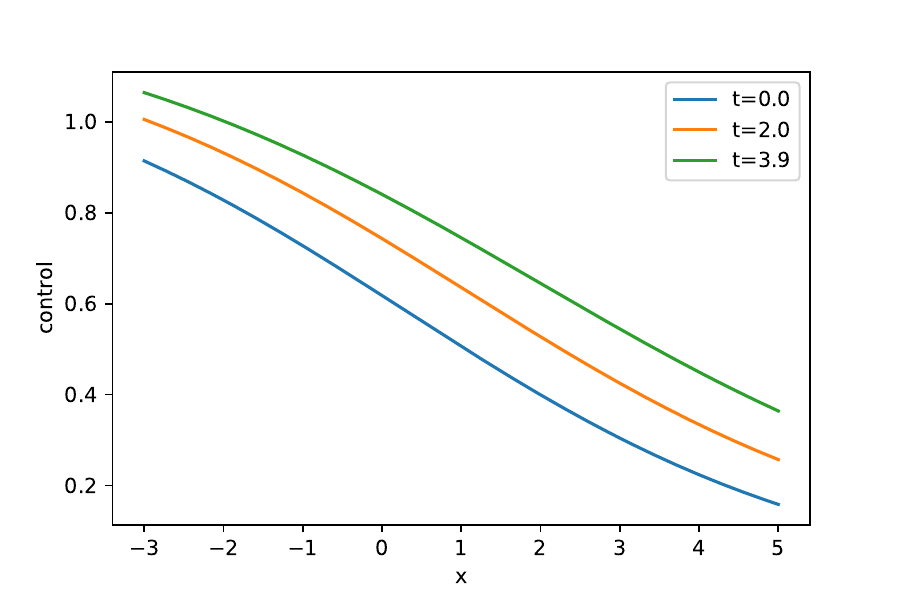}
\hskip -4pt
\includegraphics[width=6cm,height=4.5cm]{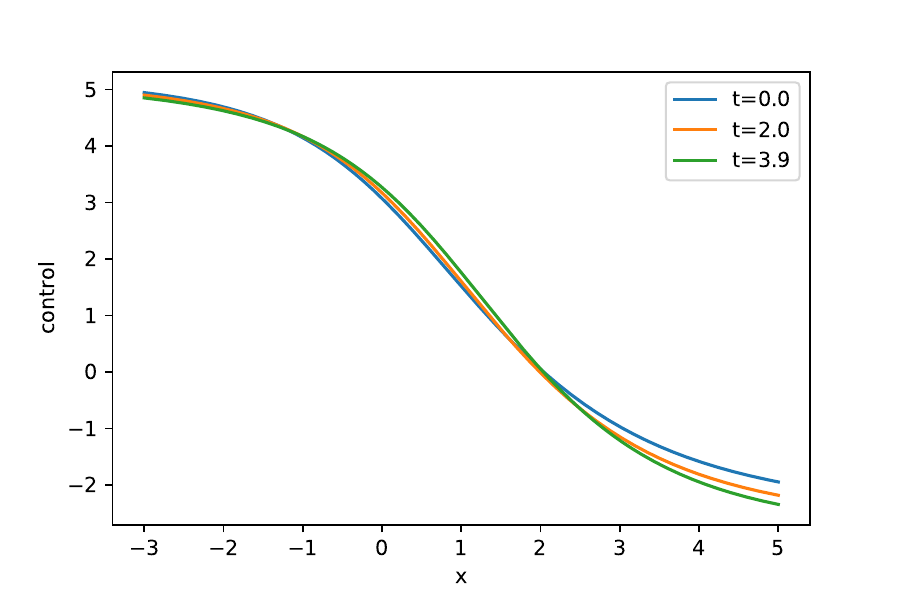}
\hskip -4pt
\includegraphics[width=6cm,height=4.5cm]{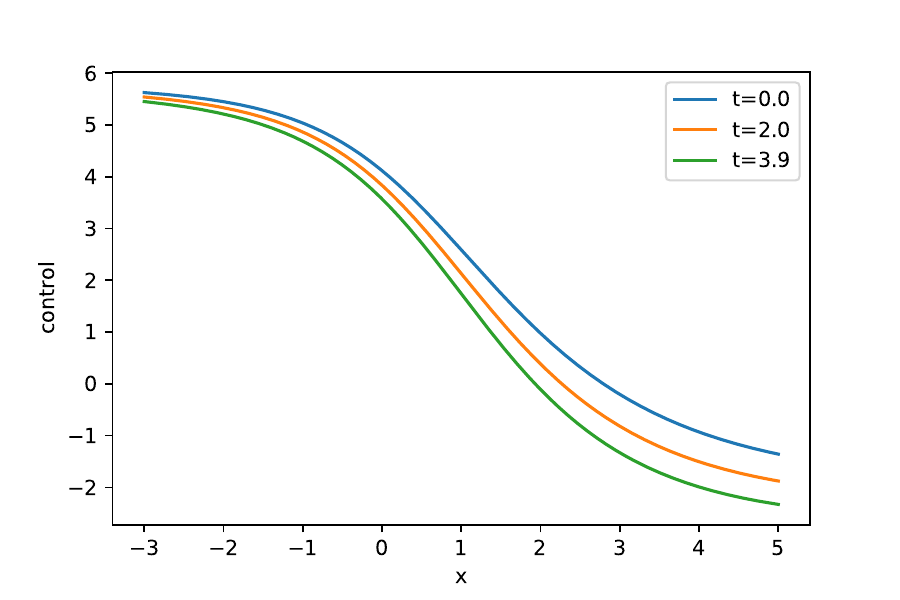}
}
\vskip -4pt
\caption{\label{fig:control} Plots of the optimal feedback function $\hat\phi_t(x)$ at different times $t$ for $d=1$, for initial starting points $x_0=-1$ (top row, left pane), $x_0=0$ (top row, right pane), and  $x_0=1$, $x_0=1$ and $x_0=2$ (bottom row, from left to right).
}
\end{figure}

\vskip 6pt
\paragraph{\bf Dependence upon the initial condition. } 
Figure \ref{fig:control} illustrate another peculiarity of the conditional control problem. In classical control problems, the optimal feedback control function is typically independent of the initial condition or the initial distribution. It is known that this property does not hold any longer for mean field control problems. This figure shows that it still does not hold in the present situation even though our model of conditional control cannot be reduced to a mean field control. Still the strong dependence upon the past forces the optimal control to remember where the state process started from!

\bibliographystyle{plain}

\appendix 

\section{Existence and uniqueness for the PDE system in short time}

We show existence and uniqueness of a classical solution for the  following PDE system, which corresponds to the case of closed-loop controls, functions of $(t,x)$:
\begin{equation}
\label{eq:pde_system_closedloop-mu-u}
\begin{cases}
    &\partial_t \mu=\frac12 \Delta_x \mu +\div_x(\nabla_xu\;\mu)  -(V-<\mu,V>)\mu\\
    &0=\partial_t u + \frac12\Delta_x u- \frac12 |\nabla_x u|^2-(V-<\mu,V>)u +V<\mu,u>+\tilde f
    \\
    &\mu_0 = m_0, \qquad u_T = g.
\end{cases}
\end{equation}

{\bf Assumptions. } In this part, we will assume that:
\begin{itemize}
    \item $m_0$ has compact support 
    \item $\tilde{f}$ satisfies: $\tilde{f}$ is twice differentiable and $\tfrac{1}{2}$-H\"older continuous; $\tilde{f}$, $\nabla \tilde{f}$ and $D^2 \tilde{f}$ are bounded; we consider constants $C_{\tilde{f}}, C_{\nabla \tilde{f}}, C_{D^2 \tilde{f}}$ such that $\|\tilde{f}\|_\infty \le C_{\tilde{f}}$, $\|\nabla \tilde{f}\|_\infty \le C_{\nabla\tilde{f}}$, $\|D^2 \tilde{f}\|_\infty \le C_{D^2 \tilde{f}}$.
    \item $g$ satisfies: $g$ is twice differentiable and $\tfrac{1}{2}$-H\"older continuous; $g$, $\nabla g$ and $D^2 g$ are bounded; let $C_g>0$ such that $\|g\|_\infty\le C_g$, $\|\nabla g\|_\infty\le C_g$ and $\|D^2 g\|_\infty\le C_g$. 
    \item $V$ satisfies: $V$ is twice differentiable and $\tfrac{1}{2}$-H\"older continuous; $\|\nabla V\|_\infty \le C_{\nabla V}$, $\|D^2 V\|_\infty \le C_{D^2 V}$; for the existence and uniqueness theorem provided below, we will allow $\|V\|_\infty$ to be small enough (in particular smaller than $1$, in contrast with the rest of the paper where we assumed $\|V\|_\infty=1$ for simplicity)
    \item $\tilde{u}$ satisfies: $\tilde{u}$ is differentiable in time and twice differentiable in space and it is $\frac{1}{2}$-H\"older continuous; as a function of $(t,x)$.
\end{itemize}
Additionally, to prove the existence and uniqueness result below, we will assume that $T$, $C_g$ and $\|V\|_\infty$ are small enough.

We will use the following classical notations:
\begin{itemize}
    \item $\cC^{\alpha}([0,T] \times \RR^d)$ and $\cC^{\alpha}(\RR^d)$ denote the sets of $\alpha-$H\"older functions on $[0,T] \times \RR^d$ and $\RR^d$ respectively.
    \item $\cC^{2+\alpha}$ (see Cardaliaguet's notes, top of page 37) is the set of functions $u: [0,T] \times \RR^d \to \RR$ such that the derivatives $\partial_t^k D_x^\ell u$ exist for any pair $(k,\ell)$ such that $2k+\ell \le 2$ and such that these derivatives are bounded, $\alpha$-H\"older continuous in space and $\frac{\alpha}{2}$-H\"older continuous in time.
    \item $\mathcal{C}^{1,2}$ denotes the space of functions that are one-time differentiable w.r.t. $t$ and two times differentiable w.r.t. $x$ and such that these derivatives are continuous. It is a subset of $\mathcal{C}^{2+\frac{1}{2}}$.
\end{itemize}

We introduce the following notations: 
\begin{align*}
    \cK &= \bigcup_{C_1 \in \RR_+} \cK_{C_1}, \qquad \\
    &\qquad \cK_{C_1} = \left\{\mu \in \cC^0([0,T], \cP_1(\RR^d)) \,\Big|\, \sup_{s \neq t} \frac{W_1(\mu_s, \mu_t)}{|s-t|^{1/2}} \le C_1, \quad \sup_{t \in [0,T]} \int |x|^2 \mu(t,dx) \le C_1 \right\},
    \\
    \cU &= \bigcup_{K_2\in \RR_+, C_2\in \RR_+, C_3  \in \RR_+} \cU_{K_2, C_2, C_3}, \qquad
    \\
    &\qquad \cU_{K_2, C_2, C_3} = \left\{ u \in \cC^{1,2} \,\Big|\, \| u\|_\infty \le K_2, \| \nabla_x u\|_\infty \le K_2, \|D^2_x u\|_\infty \le C_2, \|D^2_x u\|_\infty \le C_2, \|\partial_t u\|_\infty \le C_3 \right\}.
\end{align*}
Notice that these spaces depend implicitly on the time horizon $T$ and on $C_g$ but to alleviate the notations, we do not write this dependence explicitly. 

We endow 
\begin{itemize}
    \item $\cK$ (and its subsets) with the distance: $d_{\cK}(\mu,\mu') = \sup_{t \in [0,T]} W_1(\mu_t,\mu_t')$,
    \item $\cU$ (and its subsets) with the following norm:
$$
    \|u\|_\mathcal{U} = \sum_{k \in \NN, \ell \in \NN \,: \, 2k+\ell \le 2} \|\partial_t^k D^\ell u\|_\infty = \| u\|_\infty + \|\partial_t u\|_\infty + \| \nabla_x u\|_\infty + \|D^2_x u\|_\infty.
$$
\end{itemize}

For every $C_1>0$, $K_2>0$, $C_2>0$, $C_3>0$, $\cK_{C_1}$ and $\cU_{K_2,C_2, C_3}$ are Banach spaces. Furthermore, for every $C_1$, the set $\cK_{C_1}$ is a convex and closed subset of $\cC^0([0,T], \cP_1(\RR^d))$ which is compact. 

We further introduce the notations:
\begin{equation}
        \label{eq:def-Gamma-2}
        \Gamma_2(C_g) = 3[C_g + 1]^2 e^{6 C_g},
\end{equation}
\[
    C_3(C_2) = \max\{ 
        2C_2 + C_{\tilde{f}},
        2C_{\nabla V} C_2 + C_2 + C_{\nabla \tilde{f}},
        2C_{D^2 V} C_2 + 2C_{\nabla V} C_2 + C_2 + C_{D^2 \tilde{f}} 
        \},
\]
and
\begin{equation}
    \label{eq:def-Gamma-3}
    \Gamma_3(C_2) = C_2 + C_3(C_2).
\end{equation}
Note that $C_2 \le \Gamma_3(C_2)$.

We will also use the notation:
\begin{equation}
    \label{eq:def-KTCg}
    K(T,C_g) = e^{3(TC_3(\Gamma_2(C_g))+C_g)}(TC_3(\Gamma_2(C_g))+C_g).
\end{equation}
As $T \to 0$, $K(T,C_g) \to e^{3C_g}C_g \le \Gamma_2(C_g)$. 
As a matter of fact, we will take $T$ small enough such that:
\[
    K(T,C_g) \le \Gamma_2(C_g). 
\]
Note that, as $T \to 0$ and $C_g \to 0$, $K(T,C_g) \to 0.$ 

The main result of this section is the following existence and uniqueness result.
\begin{theorem}
\label{thm:main-existence-uniqueness-shortT}
Let $\epsilon>0$ and let $C_1>\int |x|^2 \mu_0(dx)$.
    There exists $T_1$, $C_{g,1}$, and $C_{V,1}$ depending only on the model parameters (except $g$ and $V$) and on $\epsilon$ such that, if $T<T_1$, $\|g\|_\infty < C_{g,1}$ and $\|V\|_\infty<C_{V,1}$, then there exists a unique solution $(u,\mu)$ to system~\eqref{eq:pde_system_closedloop-mu-u} in $\cK_{C_1} \times \cU_{K(T,C_g), \Gamma_2(C_g)+\epsilon, \Gamma_3(\Gamma_2(C_g))+\epsilon}$.
\end{theorem}

We will use the following proof strategy. We denote by $\Phi: \tilde{u} \mapsto u$ the function which maps a function $\tilde{u} \in \cU$ to the function $u$ of the solution $(\mu,u)$ to the following system, which is a modification of~\eqref{eq:pde_system_closedloop-mu-u}:
\begin{equation}
\label{fo:pde_system-auxiliary-proof}
\begin{cases}
    &\partial_t \mu=\frac12 \Delta_x \mu +\div_x(\nabla_xu\;\mu)  -(V-<\mu,V>)\mu
    \\
    &0=\partial_t u + \frac12\Delta_x u- \frac12 |\nabla_x u|^2-(V-<\mu,V>)\tilde{u} +V<\mu,\tilde{u}>+\tilde f
    \\
    &\mu(0) = m_0, \qquad u(T) = g.
\end{cases}
\end{equation}
The proof will consist of two steps which, at a high level, can be described as follows. In the {\bf first step}, we will prove that for $T$ small enough, $\Phi$ is well defined: we will show that, for every $\tilde{u} \in \cU_{K_2,C_2,C_3}$, the above system has a unique solution $(\mu,u)$ in a set of the form $\cK_{C_1} \times \cU_{K_2,C_2,C_3}$ where $C_1, K_2, C_2$ and $C_3$ are to be determined. This step will itself rely on a fixed point argument for the forward-backward system~\eqref{fo:pde_system-auxiliary-proof}. The fact that $\tilde{u}$ is fixed implies that the equation has no zero-order term and hence the Hopf-Cole transform can readily be applied. In the {\bf second step}, we will prove that $\Phi$ is a strict contraction on $\cU_{K_2,C_2,C_3}$, which will establish existence and uniqueness of the solution to~\eqref{eq:pde_system_closedloop-mu-u}.

To be specific, we now provide the main statements for each of the two steps. 
For the {\bf first step}, while keeping $\tilde{u}$ fixed, we will consider $\Psi := \Psi_2 \circ \Psi^{\tilde{u}}_1$, where $\Psi^{\tilde{u}}_1: \mu \mapsto u$ maps $\mu \in \cK$ to the solution $u$ of the backward PDE in~\eqref{fo:pde_system-auxiliary-proof}, and $\Psi_2: u \mapsto \mu'$ maps $u \in \cU$ to the solution $\mu'$ of the forward PDE in~\eqref{fo:pde_system-auxiliary-proof}. We will prove the following:

\begin{proposition}
\label{prop:existence-shortT-Psi-contraction}

Let $\epsilon>0$ and let $C_1>\int |x|^2 \mu_0(dx)$. 
 There exists $T_0>0$ and $C_{g,0}$ depending only on the model's parameters except $C_g$, on $\epsilon$ and on $C_1$ such that: if $T<T_0$ and $C_g < C_{g,0}$, and if $\tilde{u} \in \cU_{K(T,C_g), \Gamma_2(C_g)+\epsilon, \Gamma_3(\Gamma_2(C_g)) + \epsilon}$,
 then the function $\Psi^{\tilde{u}} = \Psi_2 \circ \Psi^{\tilde{u}}_1$ is well defined on $\cK_{C_1}$, $\Psi^{\tilde{u}}(\cK_{C_1}) \subseteq \cK_{C_1}$, $\Psi^{\tilde{u}}$ is a contraction on $\cK_{C_1}$, and furthermore $\Psi^{\tilde{u}}_1(\cK_{C_1}) \subseteq \cU_{K(T,C_g), \Gamma_2(C_g)+\epsilon, \Gamma_3(\Gamma_2(C_g)) + \epsilon}$.
\end{proposition}

For the {\bf second step}, we will prove the following result.
\begin{proposition}
\label{prop:existence-shortT-Phi-contraction}

Let $\epsilon>0$. 
There exists $T_1$, $C_{g,1}$, and $C_{V,1}$ depending only on the model parameters (except $g$ and $V$) and on $\epsilon$ such that, if $T<T_1$, $\|g\|_\infty < C_{g,1}$ and $\|V\|_\infty<C_{V,1}$, 
then the function $\Phi$ is well defined on $\cU_{K(T,C_g),\Gamma_2(C_g)+\epsilon, \Gamma_3(\Gamma_2(C_g)) + \epsilon}$, $\Phi(\cU_{K(T,C_g),\Gamma_2(C_g)+\epsilon, \Gamma_3(\Gamma_2(C_g)) + \epsilon}) \subseteq \cU_{K(T,C_g),\Gamma_2(C_g)+\epsilon, \Gamma_3(\Gamma_2(C_g)) + \epsilon}$ and, moreover, $\Phi$ is a strict contraction on the set $\cU_{K(T,C_g),\Gamma_2(C_g)+\epsilon, \Gamma_3(\Gamma_2(C_g)) + \epsilon}$. 
\end{proposition}

From these two propositions, the proof of Theorem~\ref{thm:main-existence-uniqueness-shortT} is concluded by applying Banach fixed point theorem to the contraction $\Phi$ on the complete metric 
space $\cU_{K(T,C_g), \Gamma_2(C_g)+\epsilon, \Gamma_3(\Gamma_2(C_g)) + \epsilon}$. 
In the next two subsections, we prove Proposition~\ref{prop:existence-shortT-Psi-contraction} and Proposition~\ref{prop:existence-shortT-Phi-contraction}. 
The proof of each proposition is itself split into lemmas. 

\subsection{Proof of Proposition~\ref{prop:existence-shortT-Psi-contraction} }

We use the notation and assumptions in the statement of Proposition~\ref{prop:existence-shortT-Psi-contraction}. 
We split the proof into two lemmas.

As mentioned above, we define $\Psi^{\tilde{u}} = \Psi_2 \circ \Psi^{\tilde{u}}_1$, where $\Psi^{\tilde{u}}_1: \mu \mapsto u$ maps $\mu \in \cK$ to the solution $u$ of the backward PDE:
\begin{equation}
    \label{eq:u-auxiliary-proof-form}
    -\partial_t u(t,x) = \frac12\Delta_x u(t,x)- \frac12 |\nabla_x u(t,x)|^2 + F(t, x, \mu_t), \qquad u(T,x) = g(x),
\end{equation}
where $F$ implicitly depends on $\tilde{u}$ and is defined as:
\begin{equation}
    \label{eq:proof-Phi1-def-F}
    F(t,x,\mu) = -(V(x)-<\mu(t),V>)\tilde{u}(t,x) +V(x)<\mu(t),\tilde{u}(t)>+\tilde f(x), 
\end{equation}
and $\Psi_2: u \mapsto \mu$ maps $u \in \cU$ to the solution $\mu$ of the forward PDE:
\begin{equation}
    \label{eq:mu-auxiliary-proof-form}
    \partial_t \mu(t,x) = \frac12 \Delta_x \mu(t,x) +\div_x(\nabla_xu(t,x)\;\mu(t,x)) -(V(x)-<\mu_t,V>) \mu(t,x), \qquad \mu(0,x) = m_0(x).
\end{equation}

We start with the following result about $\Psi_1^{\tilde{u}}$. 
\begin{lemma}
\label{lem:existence-shortT-Psi1}
    Assume the conditions in Proposition~\ref{prop:existence-shortT-Psi-contraction} hold. 
    There exists $T_0>0$ and $C_{g,0}$ depending only on the model's parameters except $C_g$, on $\epsilon$ and on $C_1$ such that: if $T<T_0$ and $C_g < C_{g,0}$, and if $\tilde{u} \in \cU_{K(T,C_g), \Gamma_2(C_g)+\epsilon, \Gamma_3(\Gamma_2(C_g)) + \epsilon}$, 
    then the function $\Psi^{\tilde{u}}_1$ is well defined on $\cK_{C_1}$, $\Psi^{\tilde{u}}_1(\cK_{C_1}) \subseteq \cU_{K(T,C_g), \Gamma_2(C_g)+\epsilon, \Gamma_3(\Gamma_2(C_g)) + \epsilon}$ and $\Psi^{\tilde{u}}_1$ is a contraction on $\cK_{C_1}$.
\end{lemma}

\begin{proof}

\textbf{Part (i). $\Psi^{\tilde{u}}_1$ is well defined over $\cK_{C_1}$.}

Consider $\mu \in \cK_{C_1}$. Let $u = \Psi^{\tilde{u}}_1(\mu)$. Let $\mu \in \cK$. We use the Hopf-Cole transform as follows. Let $w = e^{-u}$. We have: $\partial_t w = - w \partial_t u$, $\nabla w = - w \nabla u$, and $\Delta w = w |\nabla u|^2 - w \Delta u$. So, using the equation satisfied by $u$, we deduce that $w$ solves the backward PDE:
$$
    \partial_t w(t,x) = -\frac{1}{2} \Delta w(t,x) + w(t,x) F(t,x,\mu(t)), \qquad w(T,x) = e^{-g(x)}.
$$
Letting $v(t,x) = w(T-t,x)$, we obtain that $v$ solves the forward PDE:
\begin{equation}
\label{eq:existence-shortT-aux-v}
    \partial_t v(t,x) - \frac{1}{2} \Delta v(t,x) + \tilde{F}(t,x) v(t,x) = 0, \qquad v(0,x) = v_0(x),
\end{equation}
where $\tilde{F}(t,x) = -F(T-t,x,\mu(T-t))$ and $v_0(x) = e^{-g(x)}$. This is a heat equation.
We have that $\tilde{F}:[0,T] \times \RR^d \to \RR$ and $v_0: \RR^d \to \RR$ are in $\cC^{1/2}([0,T] \times \RR^d)$ and $\cC^{1/2}(\RR^d)$ respectively. Indeed, recall that we assumed: 
$\tilde{f}$ is $\frac{1}{2}$-H\"older continuous, 
$g$ is bounded and $1/2$-H\"older, 
$V$ is $\frac{1}{2}$-H\"older continuous, 
and $(t,x) \mapsto \tilde{u}(t,x)$ is bounded and $\frac{1}{2}$-H\"older continuous.

So:
\begin{itemize}
    \item For $\tilde{F}$:
    \begin{align*}
        &|\tilde{F}(s,x) - \tilde{F}(t,y)| 
        \\
        &= |-F(T-s,x,\mu(T-s)) + F(T-t,y,\mu(T-t))| 
        \\
        &= |(V(x)-<\mu_s,V>)\tilde{u}(s,x) - (V(y)-<\mu_t,V>)\tilde{u}(t,y)| 
        \\
        &\qquad + |V(x)<\mu_s,\tilde{u}(s)> - V(y)<\mu_t,\tilde{u}(t)>| + |\tilde f(x) - \tilde f(y)| 
        \\
        &\le C |V(x)-V(y)| + C|<\mu_s - \mu_t,V>| + C|\tilde{u}(s,x) - \tilde{u}(t,y)| 
        \\
        &\qquad + C |V(x) - V(y)| + C|<\mu_s-\mu_t,\tilde{u}(s)>| + C|<\mu_t,\tilde{u}(s)-\tilde{u}(t)>| + |\tilde f(x) - \tilde f(y)|
        \\
        &\le C (|x-y|^{1/2} + |s-t|^{1/2}),
    \end{align*}
    where we used the $\frac{1}{2}$-H\"older continuity of $x \mapsto V(x)$, $(t,x) \mapsto \tilde{u}(t,x)$, $t \mapsto <\mu_t,1>$, and $x \mapsto \tilde{f}(x)$, and the fact that $\tilde{u}$ and $V$ are bounded. We also used the fact that, since $\mu \in \cK_{C_1}$, $W_1(\mu_s,\mu_t) \le C_1 |s-t|^{1/2}$. 
    \item Using the fact that the exponential function is locally Lipschitz and $g$ is bounded and $\frac{1}{2}$-H\"older, $|v_0(x) - v_0(y)| \le C|g(x) - g(y)| \le C|x-y|^{1/2}$ where the value of $C$ changes from one inequality to the next one and it may depend on $\|\tilde{u}\|_\infty$, which is at most $K(T,C_g)$.
\end{itemize}

Then the above equation~\eqref{eq:existence-shortT-aux-v} for $v$ has a unique weak solution, and this solution is of class $\cC^{2+\frac{1}{2}}$.
This implies the unique solvability of~\eqref{eq:u-auxiliary-proof-form}, with a solution $u \in \cC^{2+\frac{1}{2}}$. Hence $\Psi^{\tilde{u}}_1$ is well defined on $\cK_{C_1}$.  

\vskip 12pt

\textbf{Part (ii). $\Psi^{\tilde{u}}_1(\cK_{C_1}) \subseteq \cU_{K(T,C_g), \Gamma_2(C_g)+\epsilon, \Gamma_3(\Gamma_2(C_g))+\epsilon}$. }

To alleviate the notation, in the sequel we let $C_3 = C_3(\Gamma_2(C_g))$ with $\Gamma_2(C_g)$ as in the statement. 
We need to bound the norm of $u$ and its derivatives involved in the definition of $\cU$.

\vskip 6pt

{\bf Bound on $\|u\|_\infty$. } 
By Feynman-Kac formula, we have that $w$ satisfies:
\[
    w(t,x) = \EE\left[  e^{-\int_t^T F(s, X_s, \mu(s)) ds} e^{-g(X_T)} \Big| X_t = x\right].
\]
We have that $x \mapsto g(x)$ is bounded by $C_g$ and $x \mapsto F(x,\mu(t))$ is bounded by a constant $C_3$ (recall that we assumed
$\tilde{f}$ is bounded by $C_{\tilde{f}}$, 
$g$ is bounded by $C_g$, 
$(t,x) \mapsto \tilde{u}(t,x)$ is bounded by $C_{\tilde{u}}$,
and $2C_{\tilde{u}} + C_{\tilde{f}} \le C_3$).
We deduce 
\[
    e^{-TC_3 -C_g}
    \le w(t,x)
    \le e^{TC_3 +C_g},
\]
and hence
\[
    u(t,x) = -\log(w(t,x)) 
    \in [- T C_3 - C_g, T C_3 + C_g].
\]
So 
\[
    \|u\|_\infty < T C_3 + C_g \le K(T,C_g).
\]

\vskip 6pt

{\bf Bound on $\|\nabla u\|_\infty$. } Similarly, notice that, for each $i=1,\dots,d$, $\partial_i w$ solves the PDE:
$$
    \partial_t \partial_i w(t,x) = -\frac{1}{2} \Delta \partial_i w(t,x) + \partial_i w(t,x) F(t,x,\mu(t)) + w(t,x) \partial_i F(t,x,\mu(t)), \qquad \partial_i w(T,x) = -\partial_i g(x) e^{-g(x)}.
$$
By Feynman-Kac formula, we have that $\partial_i w$ satisfies:
\[
    \partial_i w(t,x) 
    = \EE\left[ 
        -\int_t^T e^{-\int_t^r F(s, X_s, \mu(s)) ds} w(r, X_r) \partial_i F(r, X_r, \mu(r)) dr
        + e^{-\int_t^T F(s, X_s, \mu(s)) ds} (-\partial_i g(X_T)e^{-g(X_T)}) \Big| X_t = x
    \right]. 
\]
Moreover,
\[
    \partial_i F 
    = -\partial_i V \tilde{u}  -(V-<\mu,V>)\partial_i \tilde{u} 
    + \partial_i V <\mu,\tilde{u}> 
    + \partial_i \tilde f.
\]
So $\|\partial_i F\|_\infty \le C_3$ from our assumptions. 

We deduce
\[
    |\partial_i w(t,x)|
    \le  T e^{T C_3} e^{TC_3+C_g} C_3
        + e^{TC_3}C_ge^{C_g} 
        =  T e^{2T C_3+C_g} C_3
        + e^{TC_3+C_g}.
\]
Since $\partial_i w = -\partial_i u e^{-u} = -\partial_i u w$, we deduce:
\begin{align*}
    |\partial_i u|
    &\le |\partial_i w| |\tfrac{1}{w}|
    \\
    &\le (T C_3 e^{2T C_3+C_g}
        + e^{TC_3+C_g}) e^{TC_3+C_g}
    \\
    &= (T C_3 e^{T C_3}
        + C_g) e^{2(TC_3+C_g)} 
    \\
    &\le e^{3(TC_3+C_g)} (TC_3 + C_g)
    \\
    &= K(T,C_g). 
\end{align*}
 
So:
\[
    \|\nabla u\|_\infty \le K(T,C_g).
\]

\vskip 6pt
{\bf Bound on $\|D^2 u\|_\infty$. } We proceed similarly for the second order derivative. Let $i, j \in \{1,2,\dots,d\}$. We have: 
\begin{align*}
    \partial_t \partial_{ji} w(t,x) 
    &= -\frac{1}{2} \Delta \partial_{ji} w(t,x) + \partial_{ji} w(t,x) F(t,x,\mu(t))  + \partial_{i} w(t,x) \partial_{j} F(t,x,\mu(t)) + w(t,x) \partial_{ji} F(t,x,\mu(t)) 
    \\
    &=: -\frac{1}{2} \Delta \partial_{ji} w(t,x) + \partial_{ji} w(t,x) F(t,x,\mu(t)) + \check{F}_{j,i}(t,x)
    \\
    \partial_{ji} w(T,x) 
    &= -\partial_{ji} g(x) e^{-g(x)} + \partial_{i} g(x) \partial_{j} g(x) e^{-g(x)}
    \\
    &=:\check{G}_{ji}(T,x).
\end{align*}
By Feynman-Kac formula, we have that $\partial_{ji} w$ satisfies:
\[
    \partial_{ji} w(t,x) 
    = \EE\left[ 
        -\int_t^T e^{-\int_t^r F(s, X_s, \mu(s)) ds} \check{F}_{j,i}(r, X_r) dr
        + e^{-\int_t^T F(s, X_s, \mu(s)) ds} \check{G}_{ji}(T, X_T) \Big| X_t = x
    \right]. 
\]
We deduce
\begin{align*}
    |\partial_{ji} w(t,x)|
    &\le  T e^{T C_3} \|\check{F}_{j,i}\|_\infty
        + e^{TC_3}\|\check{G}_{i}\|_\infty
    \\
    &\le T e^{T C_3} (\|\partial_i w\|_\infty \|\partial_j F\|_\infty + \|w\|_\infty \|\partial_{ji} F\|_\infty)
        + e^{TC_3}(C_g e^{C_g} + C_g^2 e^{C_g} ).
\end{align*}
As before, $\|\partial_{i} F\|_\infty \le C_3$. 
Moreover,
\begin{align*}
    \partial_{ji} F 
    &= -\partial_{ji} V \tilde{u} - \partial_j V \partial_i \tilde{u}  -(V-<\mu,V>)\partial_{ji} \tilde{u} 
    \\
    &\qquad + \partial_{ji} V <\mu,\tilde{u}> 
    + \partial_{ji} \tilde f.
\end{align*}
So $\|\partial_{ji} F\|_\infty \le C_3$ from our assumptions. 

Hence:
\begin{align*}
    |\partial_{ji} w(t,x)|
    &\le T e^{T C_3} \left([T e^{2T C_3+C_g} C_3
        + e^{TC_3+C_g}] C_3 + [TC_3 + C_g] C_3\right)
        + e^{TC_3}(C_g e^{C_g} + C_g^2 e^{C_g} )
    \\
    &= T C_3 e^{T C_3} \left([T C_3 e^{T C_3} 
        + 1] e^{TC_3+C_g} + [TC_3 + C_g]\right)
        + e^{TC_3}(1 + C_g) C_g e^{C_g}
    \\
    &\le  \left([TC_3 + C_g + 1] e^{3(TC_3+C_g)} \right) (T C_3+C_g)
    \\
    &\le  2[TC_3 + C_g + 1] e^{3(TC_3+C_g)}.
\end{align*}

Since $\partial_{ji} w = -\partial_{ji} u e^{-u} + \partial_j u \partial_i u e^{-u} = (-\partial_{ji} u + \partial_j u \partial_i u)w$, we deduce:
\begin{align*}
    |\partial_{ji} u|
    &\le |\partial_{ji} w | |\tfrac{1}{w}| + |\partial_j u| |\partial_i u|
    \\
    &\le 2[TC_3 + C_g + 1] e^{3(TC_3+C_g)} e^{TC_3 + C_g} + |C'_{TC_3, C_g}|^2
    \\
    &\le 3[TC_3 + C_g + 1]^2 e^{6(TC_3+C_g)}
    \\
    &=: C''_{TC_3, C_g}.
\end{align*}
Note that, as $T \to 0$, 
\begin{align*}
    C''_{TC_3, C_g} 
    \to 
    C''_{0, C_g} 
    = 3[C_g + 1]^2 e^{6 C_g}
    \le \Gamma_2(C_g).
\end{align*}

\vskip 6pt

{\bf Bound on $\|\partial_t u\|_\infty$. } Last, we obtain a bound on $\|\partial_t u\|_\infty$. From the PDE~\eqref{eq:u-auxiliary-proof-form} satisfied by $u$, we have $\partial_t u(t,x) =  -\frac12\Delta_x u(t,x) + \frac12 |\nabla_x u(t,x)|^2 - F(t, x, \mu_t)$. Using the above bounds, we get that: 
\begin{align*}
    \|\partial_t u\|_\infty 
    &\le \frac{1}{2} \|D^2 u\|_\infty + \frac{1}{2} \|\nabla u\|_\infty^2 + \|F\|_\infty 
    \\
    &\le \frac{1}{2} C''_{TC_3, C_g} + \frac{1}{2}(C'_{TC_3, C_g})^2 + C_3
    \\
    &= \frac{1}{2} 3[TC_3 + C_g + 1]^2 e^{6(TC_3+C_g)} + \frac{1}{2}(e^{3(TC_3+C_g)} (TC_3 + C_g))^2 + C_3
    \\
    &\le 3[TC_3 + C_g + 1]^2 e^{6(TC_3+C_g)} + C_3
    \\
    &=: C'''_{TC_3, C_g}.
\end{align*}
Note that, as $T \to \infty$,
\begin{align*}
    C'''_{TC_3, C_g} 
    \to C'''_{0, C_g}
    = 3[C_g + 1]^2 e^{6 C_g} + C_3
    \le \Gamma_3(\Gamma_2(C_g)).
\end{align*}

\vskip 6pt

{\bf Conclusion of this part. } 
Overall, we obtain that, for any $\epsilon>0$ and any $C_g>0$, there exists $T_0>0$ depending only on the model's parameters and on $\epsilon$ such that: for all $T<T_0$,
\begin{align*}
    \Psi^{\tilde{u}}_1(\cK_{C_1}) 
    \subseteq 
    \, &\cU_{K(T,C_g), \Gamma_2(C_g)+\epsilon, \Gamma_3(\Gamma_2(C_g))+\epsilon}.
\end{align*}

\vskip 6pt

\textbf{Part (iii). $\Psi^{\tilde{u}}_1$ is a contraction on $\cK_{C_1}$. }

Consider $\mu,\mu' \in \cK_{C_1}$. Let $u = \Psi^{\tilde{u}}_1(\mu)$ and $u = \Psi^{\tilde{u}}_1(\mu')$. Based on the above analysis, we have $u, u' \in \cU_{K(T,C_g), \Gamma_2(C_g)+\epsilon, \Gamma_3(\Gamma_2(C_g))+\epsilon}$ for some $\epsilon>0$ depending on $T$ in a non-decreasing way. To alleviate the notations, we let $K_2 = K(T,C_g)$ and $C_2 = \Gamma_2(C_g)+\epsilon$. Note that $K_2$ decreases when $T$ and $C_g$ decrease and $C_2$ when $T$ decreases. Furthermore, we take $T,C_g$ small enough that $K_2 \le C_2$.

Let $w = e^{-u}$, $w' = e^{-u'}$, and $\delta w = w - w'$. 
Intuitively, we want to show that:
$$
    \|\delta w\|_{\cU} 
    \le C d_{\cK}(\mu,\mu'),
$$
where the constant $C$ is strictly smaller than $1$ at least when $T$ and $C_g$ are small enough.

We obtain such a bound in the following way. 

\vskip 6pt

{\bf Bound on $\|\delta w\|_\infty$. } 
First, we note that $\delta w$ satisfies: $\delta w_T = 0$ and:
\begin{equation}
\label{eq:proof-Psi1-deltaw}
    \partial_t \delta w = -\frac{1}{2} \Delta \delta w + \delta w F(t,x,\mu_t) 
    + w' \delta F,
\end{equation}
where $\delta F = F(t,x,\mu_t) - F(t,x,\mu_t')$. 

By Feynman-Kac formula, we have:
\begin{equation}
\label{eq:proof-Psi1-deltaw-FC}
    |\delta w(t,x)| 
    = 
    \EE \left[ \int_t^T e^{-\int_t^r F(s, X_s, \mu_s) ds} |w'(r,X_r) \delta F(r,X_r) dr| \right]
\end{equation}
Next, we bound $\|\delta F\|_\infty$. We have:
\begin{align}
    &|\delta F(t,x)|  = |F(t,x,\mu(t)) - F(t,x,\mu'(t))| 
    \notag\\
    &= |<\mu(t),V>\tilde{u}(t,x) +V(x)<\mu(t),\tilde{u}(t)>
    -<\mu'(t),V>\tilde{u}(t,x) -V(x)<\mu'(t),\tilde{u}(t)> |   
    \notag\\
    &\le |<\mu(t) - \mu'(t),V>||\tilde{u}(t,x)| +V(x)|<\mu(t) - \mu'(t),\tilde{u}(t)>|
    \notag\\
    &= \|\nabla V\|_\infty |<\mu(t) - \mu'(t), \frac{V}{\|\nabla V\|_\infty}>||\tilde{u}(t,x)| +V(x) \|\nabla\tilde{u}(t)\|_\infty|<\mu(t) - \mu'(t),\frac{\tilde{u}(t)}{\|\nabla\tilde{u}(t)\|_\infty}>|
    \notag\\
    &\le K(T,C_g) (\|\nabla V\|_\infty+1) W_1(\mu(t), \mu'(t))
    \\
    &\le C_{\nabla V, K_2} W_1(\mu(t), \mu'(t)),
    \label{eq:Psi1-bound-deltaF}
\end{align}
where $C_{\nabla V, K_2}$ is a constant that depends on $\|\nabla V\|_\infty$ and $K_2=K(T,C_g)$ but remains bounded when they remain bounded. We used the fact that $\frac{V}{\|\nabla V\|_\infty}$ and $\frac{\tilde{u}(t)}{\|\nabla\tilde{u}(t)\|_\infty}$ are $1$-Lipschitz. 

Going back to~\eqref{eq:proof-Psi1-deltaw-FC}, we have:
\begin{align}
    |\delta w(t,x)| 
    &= 
    \EE \left[ \int_t^T e^{-\int_t^r F(s, X_s, \mu_s) ds} |w'(r,X_r) \delta F(r,X_r) dr| \right]
    \notag 
    \\
    &\le C_{\nabla V, C_2} \|w'\|_\infty T \sup_{r \in [0,T]} W_1(\mu_r, \mu_r')
    \notag 
    \\
    &\le T C_{\nabla V, K_2} \sup_{r \in [0,T]} W_1(\mu_r, \mu_r'),
    \label{eq:Psi1-bound-deltaw}
\end{align}
where $C_{\nabla V, K_2}$ is a constant that depends on $\|\nabla V\|_\infty$ and $K_2=K(T,C_g)$ but remains bounded when they remain bounded.

\vskip 6pt

{\bf Bound on $\|\nabla \delta w\|_\infty$. } 
 Taking the gradient on both sides in~\eqref{eq:proof-Psi1-deltaw}, we deduce:
\begin{equation}
\label{eq:proof-Psi1-nabla-deltaw}
    \partial_t \nabla \delta w = -\frac{1}{2} \Delta \nabla w + \nabla \delta w F(t,x,\mu_t) + \underbrace{\delta w \nabla F(t,x,\mu_t) 
    + \nabla w' \delta F
    + w' \nabla \delta F}_{\tilde{F}}. 
\end{equation} 
This equation is to be understood coordinate by coordinate. 
By Feynman-Kac formula, we have:
\begin{equation}
\label{eq:proof-Psi1-nabla-deltaw-FC}
    |\nabla \delta w(t,x)| 
    = 
    \EE \left[ \int_t^T e^{-\int_t^r F(s, X_s, \mu_s) ds} |\tilde F(r,X_r) dr| \right].
\end{equation}
In order to bound $\tilde F$, we note that:
\begin{itemize}
    \item We have, recalling the definition~\eqref{eq:proof-Phi1-def-F} of $F$:
    \begin{align*}
        \nabla F(t,x,\mu) = -\nabla V(x)\tilde{u}(t,x) -(V(x)-<\mu(t),V>)\nabla \tilde{u}(t,x) + \nabla V(x)<\mu(t),\tilde{u}(t)> + \nabla \tilde f(x).
    \end{align*}
    Hence
    \begin{align*}
        \|\nabla F\|_\infty \le K_2(1+2C_{\|\nabla V\|_\infty}) + C_{\nabla \tilde f}.
    \end{align*}
    \item Similarly, for $\delta F$ as defined above,
    \begin{align*}
        &|\nabla \delta F(t, x)| 
        = |F(t,x,\mu(t)) - F(t,x,\mu'(t))| 
        \\
        &\le |<\mu(t),V>\nabla \tilde{u}(t,x) - <\mu'(t),V>\nabla \tilde{u}(t,x)| + |\nabla V(x)<\mu(t),\tilde{u}(t)> - \nabla V(x)<\mu'(t),\tilde{u}(t)>|
        \\
        &\le   \|\nabla \tilde{u}\|_\infty \|\nabla V\|_\infty |<\mu(t)-\mu'(t),\frac{V}{\|\nabla V\|_\infty}>| +  \|\nabla V\|_\infty \|\nabla \tilde{u}\|_\infty |<\mu(t)-\mu'(t),\frac{\tilde{u}(t)}{\|\nabla \tilde{u}\|_\infty}>|
        \\
        &\le K_2 C_{\|\nabla V\|_\infty} \sup_{t \in [0,T]} W_1(\mu(t), \mu'(t)),
    \end{align*}
    using the fact that $\frac{V}{\|\nabla V\|_\infty}$ and $\frac{\tilde{u}(t)}{\|\nabla \tilde{u}\|_\infty}$ are $1$-Lipschitz.
\end{itemize}
Hence,
\[
    \|\tilde{F}\|_\infty 
    \le 
    C_{\nabla V, K_2} \sup_{t \in [0,T]} W_1(\mu(t), \mu'(t)).
\]

Going back to~\eqref{eq:proof-Psi1-nabla-deltaw-FC}, we deduce that:
\begin{equation}
\label{eq:Psi1-bound-nabla-deltaw}
    \|\nabla \delta w\|_\infty \le TC_{\nabla V, K_2} \sup_{t \in [0,T]} W_1(\mu(t), \mu'(t)),
\end{equation}
where here again the constant $C_{\nabla V, K_2}$ depends on $\|\nabla V\|_\infty$ and $K_2$ but remains bounded when they remain bounded.

\vskip 6pt

{\bf Bound on $\|D^2 \delta w\|_\infty$. }

We take the partial derivative $\partial_i$ in the $j$-coordinate of ~\eqref{eq:proof-Psi1-nabla-deltaw}:
\begin{equation}
\label{eq:proof-Psi1-D2-deltaw}
    \begin{split}
    \partial_t \partial_{i,j} \delta w 
    &= -\frac{1}{2} \Delta \partial_{i,j} w + \partial_{i,j} \delta w F(t,x,\mu_t) + \partial_{j} \delta w \partial_{i} F(t,x,\mu_t) \\
    &\qquad + \partial_{i} \delta w \partial_{j} F(t,x,\mu_t) + \delta w \partial_{i,j} F(t,x,\mu_t) 
    + \partial_{i,j} w' \delta F + \partial_{j} w' \partial_{i} \delta F
    + \partial_{i} w' \partial_{j} \delta F + w' \partial_{i,j} \delta F
    \\
    &= -\frac{1}{2} \Delta \partial_{i,j} w + \partial_{i,j} \delta w F(t,x,\mu_t) + \tilde{\tilde{F}},
    \end{split} 
\end{equation} 
with 
$\tilde{\tilde{F}} = \partial_{j} \delta w \partial_{i} F(t,x,\mu_t) + \partial_{i} \delta w \partial_{j} F(t,x,\mu_t) + \delta w \partial_{i,j} F(t,x,\mu_t) 
    + \partial_{i,j} w' \delta F + \partial_{j} w' \partial_{i} \delta F
    + \partial_{i} w' \partial_{j} \delta F + w' \partial_{i,j} \delta F$. Using similar arguments as above, we find:
\[
    \|\tilde{\tilde{F}}\|_\infty
    \le C_{\nabla V, D^2 V, C_2} \sup_{t \in [0,T]} W_1(\mu(t), \mu'(t))
\]
where the constant $C_{\nabla V, D^2 V, C_2}$ depends on $\|\nabla V\|_\infty$, $\|D^2 V\|_\infty$ and $C_2$ (recall that $K_2 \le C_2$) but remains bounded when they remain bounded.

Using Feynman-Kac formula for  $\partial_{i,j} \delta w$ for each $(i,j)$, we obtain:
\begin{equation}
\label{eq:Psi1-bound-D2-deltaw}
    \|D^2 \delta w\|_\infty \le T C_{\nabla V, D^2 V, C_2} \sup_{t \in [0,T]} W_1(\mu(t), \mu'(t)).
\end{equation}

\vskip 6pt

{\bf Bound on $\|\partial_t \delta w\|_\infty$. } 

Using~\eqref{eq:proof-Psi1-deltaw} once again, we deduce
\begin{align}
    \|\partial_t \delta w\|_\infty 
    &\le \frac{1}{2} \|\Delta \delta w\|_\infty + \|\delta w\|_\infty \|F\|_\infty 
    + \|w'\|_\infty \|\delta F\|_\infty
    \notag
    \\
    &\le \left[ T C_{\nabla V, D^2 V, C_2, C_{\tilde f}}  
    +  \|w'\|_\infty K(T,C_g) (\|\nabla V\|_\infty + 1)  \right] \sup_{t \in [0,T]} W_1(\mu(t), \mu'(t))
    \notag
    \\
    &\le \left[ T C_{\nabla V, D^2 V, C_2, C_{\tilde f}}  
    + 2 e^{TC_3 +C_g} K(T,C_g) (\|\nabla V\|_\infty + 1) \right] \sup_{t \in [0,T]} W_1(\mu(t), \mu'(t)),
    \label{eq:Psi1-bound-partialt-deltaw}
\end{align}
where we used~\eqref{eq:Psi1-bound-D2-deltaw}, \eqref{eq:Psi1-bound-deltaw} and \eqref{eq:Psi1-bound-deltaF}. Recall that, for $T$ and $C_g$ small enough, the coefficient multiplying $\sup_{t \in [0,T]} W_1(\mu(t), \mu'(t))$ is strictly smaller than $1$.

\vskip 6pt

{\bf Conclusion. } Combining, \eqref{eq:Psi1-bound-deltaw}, \eqref{eq:Psi1-bound-nabla-deltaw}, \eqref{eq:Psi1-bound-D2-deltaw} and \eqref{eq:Psi1-bound-partialt-deltaw}, we obtain the following result: for any $\epsilon>0$, there exists $T_0>0$ and $C_{g,0}$ depending only on the model's parameters except $C_g$ and on $\epsilon$ such that for all $T<T_0$ and $C_g < C_{g,0}$,
\[
    \|\delta w\|_\cU \le C \sup_{t \in [0,T]} W_1(\mu(t), \mu'(t)),
\]
with $C<1$. 
Hence, under these assumptions, $\Psi^{\tilde{u}}_1$ is a strict contraction on $\cK_{C_1}$. This concludes the proof of Lemma~\ref{lem:existence-shortT-Psi1}. 
\end{proof}

Next, we turn our attention to $\Psi_2$ and prove the following. 
\begin{lemma}
\label{lem:existence-shortT-Psi2}

Let $C_1 > \int |x|^2 \mu_0(dx)$. Let $C_2>0$ and $C_3>0$. There exists $T_0>0$ depending only on the model's parameters, on $C_1$  and on $C_2$, such that if $T<T_0$, the function  $\Psi_2$ is well defined on $\cU_{C_2,C_2,C_3}$, $\Psi_2(\cU_{C_2,C_2,C_3}) \subseteq \cK_{C_1}$ and it is a contraction on $\cU_{C_2,C_2,C_3}$.
\end{lemma}

 Combining Lemmas~\ref{lem:existence-shortT-Psi1} and~\ref{lem:existence-shortT-Psi2} with $K_2 = K(T,C_g)$, $C_2 = \Gamma_2(C_g) + \epsilon$ and $C_3 = \Gamma_3(\Gamma_2(C_g))+\epsilon$ under the assumption that $C_1 > \int |x|^2 \mu_0(dx)$ yields Proposition~\ref{prop:existence-shortT-Psi-contraction} (notice that $\cU_{K_2,C_2,C_3} \subseteq \cU_{C_2,C_2,C_3}$). We now prove Lemma~\ref{lem:existence-shortT-Psi2}. 

\begin{proof}[Proof of Lemma~\ref{lem:existence-shortT-Psi2}]

\textbf{Part (i). $\Psi_2$ is well defined on $\cU_{C_2, C_2,C_3}$. }
This fact stems directly from the existence of a solution to the KFP equation.

\vskip 6pt
\textbf{Part (ii). $\Psi_2(\cU_{C_2, C_2,C_3}) \subseteq \cK_{C_1}$ for a suitable value of $C_1$. }

Let $u \in \cU_{C_2, C_2,C_3}$.  Let $\mu$ be the solution to the KFP driven by $\nabla u$, see the forward equation in~\eqref{fo:pde_system-auxiliary-proof}. We need to show the two inequalities in the definition of $\cK$.

We provide an upper bound on the second moment of $\mu_t$.
We have, with $\phi = - \nabla u$, 
\begin{equation*}
\begin{split}
    \partial_t <|x|^2,\mu_t>
    &=<\frac12 \Delta|x|^2 +\phi_t(x)\nabla|x|^2 -(V(x)-<V,\mu_t>)|x|^2,\mu_t>\\
    &\le d +2<x\cdot\phi_t(x),\mu_t>+<|x|^2,\mu_t>\\
    &\le d +4<|\phi_t(x)|^2,\mu_t>+5<|x|^2,\mu_t>\\
    &\le d +4C_2+5<|x|^2,\mu_t>
\end{split}
\end{equation*}
so:
$$
    <|x|^2,\mu_t> \,\le\, <|x|^2,\mu_0> +  t C_{d,C_2}+C\int_0^t<|x|^2,\mu_s> ds.
$$
Since $<|x|^2,\mu_0> + t C_{d,C_2}$ is non-decreasing in $t$, Gronwall's inequality gives
\begin{equation}
\label{eq:psi2-bound-part1}
    \int_{\RR^d}|x|^2\mu_t(dx) 
    \le 
    \left(\int_{\RR^d}|x|^2\mu_0(dx) + T C_{d,C_2}\right) e^{CT}.
\end{equation}

As for the first bound in the definition of $\cK$, we use the definition of $\mu$. Contrary to the usual MFG PDE system, here we cannot interpret $\mu$ as the law of an It\^o process. However, we can argue as follows using the characterization of the Wasserstein-1 distance using Lipschitz functions. Let us recall that, thanks to the Kantorovich-Rubinstein duality theorem, we have:
$$
    W_1(\mu_s, \mu_t)
    =
    \inf_{\varphi: 1-Lip} \int \varphi(x) d(\mu_s(x) - \mu_t(x)) 
    =
    \inf_{\varphi: 1-Lip, \varphi(0) = 0} \int \varphi(x) d(\mu_s(x) - \mu_t(x))
$$
where the second equality holds because $\mu_s$ and $\mu_t$ both have total mass equal to $1$ so we can replace $\varphi(\cdot)$ by $\varphi(\cdot) - \varphi(0)$ without changing the integral. Let $\varphi$ be a $1$-Lipschitz function such that $\varphi(0) = 0$. Since we assume that $m_0$ has compact support, $\EE[|\varphi(X_0)|^2]^{1/2}$ is bounded by a constant $C_{m_0}$ depending only on the support of $m_0$. For any $s \neq t$ in $[0,T]$, we have:
\begin{align*}
    <\mu_s - \mu_t, \varphi>
    &= \frac{\EE[\varphi(X_s) e^{-A_s}]}{\EE[e^{-A_s}]} - \frac{\EE[\varphi(X_t) e^{-A_t}]}{\EE[e^{-A_t}]}
    \\
    &= \frac{\EE[(\varphi(X_s) -\varphi(X_t)) e^{-A_s}]}{\EE[e^{-A_s}]} 
    + \frac{\EE[\varphi(X_t) (e^{-A_s} - e^{-A_t})]}{\EE[e^{-A_s}]}
    + \EE[\varphi(X_t) e^{-A_t}]\left(\frac{1}{\EE[e^{-A_s}]} - 
    \frac{1}{\EE[e^{-A_t}]}
    \right).
\end{align*}
We bound each of the three terms. Below, it is important that $C$ is a constant whose value is independent of $\varphi$. 
\begin{itemize}
    \item For the first term:
    \begin{align*}
        \left|\frac{\EE[(\varphi(X_s) -\varphi(X_t)) e^{-A_s}]}{\EE[e^{-A_s}]}\right| 
        &\le e^{T} \EE[|\varphi(X_s) -\varphi(X_t)|] 
        \\
        &\le e^{T} \EE[|X_s -X_t|] 
        \\
        &\le e^{T} \EE[|\int_t^s -\nabla u(r, X_r) dr|] 
        \\
        &\le e^{T} \|\nabla u\|_\infty |s-t|
        \\
        &\le C_{T,C_2} |s-t|.
    \end{align*} 
    \item For the second term:
    \begin{align*}
        \left|\frac{\EE[\varphi(X_t) (e^{-A_s} - e^{-A_t})]}{\EE[e^{-A_s}]}\right| 
        &\le e^{T} \EE[|\varphi(X_t)| |e^{-A_s} - e^{-A_t}|] 
        \\
        &\le e^{T} \EE[|\varphi(X_t)|^2]^{1/2} \EE[|e^{-A_s} - e^{-A_t}|^2]^{1/2}.
    \end{align*}
    We first note that $A_s,A_t \in [0,T]$ and the exponential function is locally Lipschitz so it is Lipschitz on $[0,T]$. So $|e^{-A_s} - e^{-A_t}| \le C|A_s - A_t| \le C|s - t|$. Moreover, $\EE[|\varphi(X_t)|^2]^{1/2} \le \EE[|\varphi(X_0)|^2]^{1/2} + \EE[|\varphi(X_t) - \varphi(X_0)|^2]^{1/2}$, and $\EE[|\varphi(X_0)|^2]^{1/2}<C_{m_0}$, while 
    \begin{align*}
        \EE[|\varphi(X_t) - \varphi(X_0)|^2]
        &\le \EE[|X_t - X_0|^2] 
        \\
        &= \EE[|\int_0^t -\nabla u(r, X_r) dr|^2] 
        \\
        &\le (T \|\nabla u\|_\infty)^2 
        \\
        &\le T^2 C_2^2.
    \end{align*} 
 
    \item To bound $\EE[\varphi(X_t) e^{-A_t}]\left(\frac{1}{\EE[e^{-A_s}]} - 
    \frac{1}{\EE[e^{-A_t}]}\right)$, we first note that 
    \begin{align*}
        \EE[|\varphi(X_t)| e^{-A_t}] 
        &\le \EE[|\varphi(X_t) - \varphi(X_0)|] + \EE[|\varphi(X_0)|]
        \\
        &\le T C_2 + C_{m_0}
        \\
        &\le C_{T, C_2, m_0}    
    \end{align*}
    with similar arguments as above, and moreover $\left|\frac{1}{\EE[e^{-A_s}]} - 
    \frac{1}{\EE[e^{-A_t}]} \right| = \left|\frac{1}{\EE[e^{-A_s}]\EE[e^{-A_t}]}(\EE[e^{-A_t} - e^{-A_s}])\right| \le C_T |t-s|$ with similar arguments as above.
\end{itemize}
So we have shown that: there exists a constant $C_{T,C_2,m_0}$ such that for every 1-Lipschitz function $\varphi$ and every $s, t \in [0,T]$, $<\mu_s - \mu_t, \varphi> \le C_{T,C_2,m_0} |t-s|$. Taking the supremum over 1-Lipschitz functions $\varphi$ yields:
$\sup_{s \neq t} \frac{W_1(\mu_s, \mu_t)}{|s-t|^{1/2}} \le C_1 = \sqrt{T} C_{T,C_2,m_0}$, where $C_{T,C_2,m_0}$ is bounded for bounded $T$. 

Combining this result with~\eqref{eq:psi2-bound-part1}, we obtain that 
$ \mu \in \cK_{C_1'(T)}$ with
$$
    C_1'(T) = \max\left\{\left(\int_{\RR^d}|x|^2\mu_0(dx) + T C_{d,C_2}\right) e^{CT}, \sqrt{T} C_{T,C_2,m_0}\right\}.
$$
Hence, for this $C_1'(T)$, 
$$
    \Psi_2(\cU_{C_2,C_3}) \subseteq \cK_{C_1'(T)}.
$$
Remember that we assumed $C_1 > \int_{\RR^d}|x|^2\mu_0(dx)$. So for $T$ small enough depending only on the model's parameters, on $C_1$ and on $C_2$, we have: $C_1'(T) \le C_1$. In that case,
$$
    \Psi_2(\cU_{C_2,C_2,C_3}) \subseteq \cK_{C_1}.
$$

\vskip 12pt

\textbf{Part (iii). $\Psi_2$ is a strict contraction on $\cU_{C_2,C_2,C_3}$. }

Let $u,u' \in \cU_{C_2,C_2,C_3}$. Let $\mu = \Psi_2(u),\mu' = \Psi_2(u)$. 

From Lemma~\ref{lem:W1-mu-delta-nabla-u} below, we get:
\begin{align*}
    W_1(\mu_t, \mu_t')
    & \le T C_{T,C_2} \| \nabla \delta u \|_\infty.
\end{align*}
where $C_{T,C_2}$ is bounded for bounded $T$. 
So for $T < 1/C_{T,C_2}$, $\Psi_2$ is a strict contraction. 
\end{proof}

\subsection{Step 2. Proof of Proposition~\ref{prop:existence-shortT-Phi-contraction}}

We split the proof into two lemmas. 

\begin{lemma}
\label{lem:proof-existence-deltau-Nu}
Let $\epsilon>0$. 
Let $T_0>0$ and $C_{g,0}$ be as in Proposition~\ref{prop:existence-shortT-Psi-contraction} and consider any $T<T_0$ and $C_g < C_{g,0}$. Consider $\tilde u^1,\tilde u^2 \in \cU_{K(T,C_g), \Gamma_2(C_g)+\epsilon, \Gamma_3(\Gamma_2(C_g))+\epsilon}$. Denote by $(u^1,\mu^1), (u^2,\mu^2)$ the solutions to system~\eqref{fo:pde_system-auxiliary-proof} corresponding respectively to $\tilde u = \tilde u^1$ and $\tilde u = \tilde u^2$. Let us denote $\delta u = u^1-u^2$, $\delta \tilde{u} = \tilde{u}^1-\tilde{u}^2$. 
Then there exists a constant $C_{\epsilon, T_0, C_{g,0}}$ depending only on $\epsilon$, on $T_0$, on $C_{g,0}$, and on the model's parameters such that: 
\begin{equation}
        \|\delta u\|_\cU 
    \le \max\{T, \|V\|_\infty, K(T,C_g)\} C_{\epsilon, T_0, C_{g,0}} (1+e^{T L_{\nabla u^2}})\|\delta\tilde{u}\|_\cU.
    \label{eq:proof-existence-step2-abc}
\end{equation}
\end{lemma}

\begin{proof}
Note that we know, from Proposition~\ref{prop:existence-shortT-Psi-contraction}, that $u^1$ and $u^2$ are in $\cU_{K(T,C_g), \Gamma_2(C_g)+\epsilon, \Gamma_3(\Gamma_2(C_g))+\epsilon}$.  To see this, we first find the fixed points $\mu^1, \mu^2 \in \cK_{C_1}$ (for some $C_1$ as in the statement, which depends on $T_0$ and $\Gamma_2(C_g)+\epsilon$) of $\Psi^{\tilde{u}}$ and then we apply $\Psi_2$ to obtain $u^1$ and $u^2$. We denote $\delta u = u^1-u^2$, $\delta \tilde{u} = \tilde{u}^1-\tilde{u}^2$ and $\delta \mu = \mu^1-\mu^2$.

The goal is to bound each term appearing in the definition of the norm $\|\delta u\|_\cU$ by the norm $\|\delta \tilde{u}\|_\cU$. We split the proof into several steps. To alleviate the notation, we denote $K_2 = K(T,C_g)$, $C_2 = \Gamma_2(C_g)+\epsilon$ and $C_3 = \Gamma_3(\Gamma_2(C_g))+\epsilon$. Recall that $T$ and $C_g$ are small enough that $K_2 \le C_2$. 

\vskip 6pt
{\bf Part (ii). Bound on $\|u\|_\infty$. } 

We have:
\begin{align*}
    0&=
    \partial_t (u^1-u^2) + \frac12\Delta_x (u^1-u^2) - \frac12 |\nabla_x u^1|^2 + \frac12 |\nabla_x u^2|^2 
    \\ 
    &\qquad -(V-<\mu^1,V>)\tilde{u}^1 +(V-<\mu^2,V>)\tilde{u}^2  
    \\
    &\qquad +V<\mu^1,\tilde{u}^1> -V<\mu^2,\tilde{u}^2>
    \\
    &=
    \partial_t \delta u + \frac12\Delta_x \delta u - \frac12 \nabla_x \delta u (\nabla_x u^1+\nabla_x u^2)
    \\ 
    &\qquad -V \delta \tilde{u} +<\mu^1,V>\delta\tilde{u} +<\delta\mu,V>\tilde{u}^2  
    \\
    &\qquad +V<\mu^1,\delta\tilde{u}> + V<\delta\mu,\tilde{u}^2>,
\end{align*}
with terminal condition $\delta u(T,x) = 0$ for all $x$. 
Rearranging the terms, we obtain:
\begin{align*}
    0
    &=\partial_t \delta u + \frac12\Delta_x \delta u - \frac12 \tilde{h} \nabla_x \delta u 
    \\ 
    &\qquad - (V - <\mu^1,V>)\delta\tilde{u} +V<\mu^1,\delta\tilde{u}> + \tilde{q} 
\end{align*}
with $\tilde{h} = (\nabla_x u^1+\nabla_x u^2)$, and $\tilde{q} = <\delta\mu,V>\tilde{u}^2   + V<\delta\mu,\tilde{u}^2>$.

Note that:
\begin{equation*}
    |- (V - <\mu^1,V>)\delta\tilde{u} +V<\mu^1,\delta\tilde{u}>| \le 2 \|\delta\tilde{u}\|_\infty,  \qquad 
    |\tilde{q}| \le 2 C_2 \sup_{t\in[0,T]}W_1(\mu^1_t,\mu^2_t).
\end{equation*}

We can check that $\underline{u}(t,x) = - (T-t)2(\|\delta\tilde{u}\|_\infty + C_2 \sup_{t\in[0,T]}W_1(\mu^1_t,\mu^2_t))$ and $\overline{u}(t,x) = (T-t)2(\|\delta\tilde{u}\|_\infty + C_2 \sup_{t\in[0,T]}W_1(\mu^1_t,\mu^2_t))$ are respectively a sub-solution and a super-solution. By comparison principle, we get: 
\begin{equation}
\label{eq:deltau-infty-deltatilreu-infty}
    \|\delta u\|_\infty \le T2 \max(C_2, 1)(\|\delta\tilde{u}\|_\infty + \sup_{t\in[0,T]}W_1(\mu^1_t,\mu^2_t)).
\end{equation}

\vskip 6pt
{\bf Part (ii). Hopf-Cole and Feynman-Kac formulas. } 

We introduce $w^i = e^{-u^i}, i=1,2$. They solve: 
$$
    \partial_t w^i(t,x) = -\frac{1}{2} \Delta w^i(t,x) + w^i(t,x) F^i(t,x,\mu^i(t)), \qquad w^i(T,x) = e^{-g(x)},
$$
where $F^i(t, x, \mu) = -(V(x)-<\mu,V>)\tilde{u}^i(t,x) +V(x)<\mu,\tilde{u}^i(t)>+\tilde f(x)$. 

Taking the gradient on both sides, we obtain (this PDE is understood coordinate by coordinate):
\begin{equation}
\label{eq:proof-existence-step2-nabla-w}
    \partial_t \nabla w^i(t,x) = -\frac{1}{2}  \Delta \nabla w^i(t,x) + \nabla w^i(t,x) F^i(t,x,\mu^i(t)) + w^i(t,x) \nabla F^i(t,x,\mu^i(t)), \qquad \nabla w^i(T,x) = -\nabla g(x) e^{-g(x)}.
\end{equation}
Now, taking the difference and denoting $\delta w = w^1 - w^2$ and $\delta F(t,x) = F^1(t,x,\mu^1(t)) - F^2(t,x,\mu^2(t))$,
\begin{equation}
\label{eq:proof-existence-step2-nabla-delta-w}
    \partial_t \nabla \delta w(t,x) = -\frac{1}{2}  \Delta \nabla \delta w(t,x) + \nabla \delta w(t,x) F^1(t,x,\mu^1(t))  + \check{F}(t,x), 
    \qquad \nabla \delta w(T,x) = 0,
\end{equation}
where 
\begin{equation}
    \label{eq:proof-existence-def-checkF}
    \check{F}(t,x) = - \nabla w^1(t,x) \delta F(t,x) + w^1(t,x) \nabla F^1(t,x,\mu^1(t)) - w^2(t,x) \nabla F^2(t,x,\mu^2(t)).
\end{equation}
We deduce by Feynman-Kac formula that $\nabla \delta w$ satisfies:
\begin{align}
\label{eq:proof-existence-step2-gradientw}
    \nabla \delta w(t,x) = \EE\left[\int_t^T e^{-\int_t^r F^1(s,W_s,\mu^1(s)) ds} \check{F}(r,W_r)dr \,\Big|\, W_t = x \right],
\end{align}
where $W$ is a standard Brownian motion. We want to bound the right-hand side by a quantity that is proportional to $T$ and that depends on $\delta\tilde u$ and $\delta \mu$.

\vskip 6pt
{\bf Part (iii). Bound on $\check{F}$. } 
We focus on $\check{F}$ defined in~\eqref{eq:proof-existence-def-checkF} and show that $\|\check{F}\|_\infty
    \le C(\|\delta \tilde u\|_\infty + \|\delta \nabla \tilde u\|_\infty + \sup_{t\in[0,T]}W_1(\mu^1_t,\mu^2_t) )$.

Recall that $\tilde u^i \in \mathcal{E}$. Moreover, 
\begin{align*}
    \delta F(t,x) 
    &= -(V(x)-<\mu^1(t),V>)\delta\tilde{u}(t,x) + <\delta\mu(t),V>)u^2(t,x)  
    \\
    &\qquad +V(x)<\mu^1(t),\delta\tilde{u}(t)> + V(x)<\delta\mu(t),\tilde{u}^2(t)>.
\end{align*}
So, using the fact that $V$ and $u^i,i=1,2,$ are Lipschitz, we deduce that 
\begin{align}
    \|\delta F \|_\infty
    \le C(\|\delta \tilde u\|_\infty + \sup_{t\in[0,T]}W_1(\mu^1_t,\mu^2_t)),
    \label{eq:proof-existence-phi-deltaF-tildeu}
\end{align}
where the constant $C$ may depend on the Lipschitz constant of $V$ and on $C_2$. 

So we can bound the first term in $\check{F}$: $\|\nabla w^1 \delta F\|_\infty \le C(\|\delta\tilde{u}\|_\infty + \sup_{t\in[0,T]}W_1(\mu^1_t,\mu^2_t))$, for some constant $C$ depending on $C_2$.  

For the second and third terms, we have:
$$
    w^1(t,x) \nabla F^1(t,x,\mu^1(t)) - w^2(t,x) \nabla F^2(t,x,\mu^2(t))
    = \delta w(t,x) \nabla F^1(t,x,\mu^1(t)) + w^2(t,x) \nabla \delta F(t,x), 
$$
which can be bounded by $C(\|\delta w\|_\infty + \|\delta\tilde{u}\|_\infty + \sup_{t\in[0,T]}W_1(\mu^1_t,\mu^2_t))$, for some constant $C$ depending on $C_2$. We analyze each term as follows: 
\begin{itemize}
    \item Since $w^i = e^{-u^i}$ and $\|u^i\|_\infty \le C_2$ and since the exponential function is locally Lipschitz, $\|\delta w\|_\infty \le C \|\delta u\|_\infty$. Remember that equation~\eqref{eq:deltau-infty-deltatilreu-infty} in \textbf{Part (i)} gave $\|\delta u\|_\infty \le 2T\max(C_2, 1)(\|\delta\tilde{u}\|_\infty + \sup_{t\in[0,T]}W_1(\mu^1_t,\mu^2_t))$. So we obtain: 
    \begin{equation}\label{eq:proof-existence-phi-deltaw-deltatildeu}
        \|\delta w\|_\infty \le 2 C T(\|\delta\tilde{u}\|_\infty + \sup_{t\in[0,T]}W_1(\mu^1_t,\mu^2_t)),    
    \end{equation} 
    where $C$ may depend on $C_2$. 
    \item We have $\nabla \delta F(t,x) 
    = - \nabla V(x) \delta\tilde{u}(t,x)
    -(V(x)-<\mu^1(t),V>)\delta \nabla \tilde{u}(t,x) + <\delta\mu(t),V>) \nabla u^2(t,x)  
    +V(x)<\mu^1(t),\delta\tilde{u}(t)> + \nabla V(x)<\delta\mu(t),\tilde{u}^2(t)>$. So $\|\nabla \delta F \|_\infty \le C(\| \delta \tilde u \|_\infty + \| \delta \nabla \tilde u \|_\infty + \sup_{t\in[0,T]}W_1(\mu^1_t,\mu^2_t) )$. 
\end{itemize} 
Hence,
$$
    \|\check{F}\|_\infty
    \le C(T+1)(\|\delta \tilde u\|_\infty + \|\delta \nabla \tilde u\|_\infty + \sup_{t\in[0,T]}W_1(\mu^1_t,\mu^2_t) ),
$$
where the constant $C$ depends on $C_2$ and $\|V \|_\infty$.

\vskip 6pt
{\bf Part (iv). Bound on $\|\nabla \delta w\|_\infty$.  } 
Going back to \eqref{eq:proof-existence-step2-gradientw}, we deduce:
\begin{align*}
    \|\nabla \delta w\|_\infty 
    &\le \EE\left[\int_0^T e^{-\int_0^r F^1(s,W_s,\mu^1(s)) ds} \|\check{F} \|_\infty dr \right]
    \\
    &\le T  C(T+1)(\|\delta \tilde u\|_\infty + \|\delta \nabla \tilde u\|_\infty + \sup_{t\in[0,T]}W_1(\mu^1_t,\mu^2_t) ),
\end{align*}
where the constant $C$ depends on $C_2$, $V$ and $\|V \|_\infty$.

Using again the local Lipschitz property of the exponential function, we deduce that a similar bound holds for  $\|\nabla \delta u\|_\infty $ except that the constant $C$ depends on $C_2$, $\|V \|_\infty$ and  $\|\nabla V \|_\infty$. 

\vskip 6pt
{\bf Part (v). Bound on $\|D^2 \delta w\|_\infty$.  }

We can take one more derivative in~\eqref{eq:proof-existence-step2-nabla-delta-w} and proceed as we did for $\nabla w$ by using Feynman-Kac formula. Using the fact that $u^i \in \mathcal{U}$, we deduce that 
\begin{equation}\label{eq:proof-existence-phi-D2deltaw}
    \|D^2 \delta w\|_\infty \le T  C(T+1)(\|\delta \tilde u\|_\infty + \|\delta \nabla \tilde u\|_\infty + \|\delta D^2\tilde u\|_\infty + \sup_{t\in[0,T]}W_1(\mu^1_t,\mu^2_t) )
\end{equation}
where the constant $C$ depends on $C_2$, $V$, $\|\nabla V \|_\infty$ and  $\|D^2 V \|_\infty$.  Using again the local Lipschitz property of the exponential function, we deduce that a similar bound holds for  $\|D^2 \delta u\|_\infty $ except that the constant $C$ depends on $C_2$, $\|V \|_\infty$, $\|\nabla V \|_\infty$ and  $\|D^2 V \|_\infty$. 

\vskip 6pt
{\bf Part (vi). Bound on $\|\partial_t \delta w\|_\infty$. }  

Note that $\delta w$ solves:
\begin{align*}
    \partial_t \delta w(t,x) 
    & = -\frac{1}{2} \Delta \delta w(t,x) + \delta w(t,x) F^1(t,x,\mu^1(t)) + w^1(t,x) \delta F(t,x),
\end{align*}
where we recall that: 
\[
    F^1(t, x, \mu) = -(V(x)-<\mu,V>)\tilde{u}^1(t,x) +V(x)<\mu,\tilde{u}^1(t)>+\tilde f(x)
\]
and
\begin{align*}
    \delta F(t,x) 
    &= -(V(x)-<\mu^1(t),V>)\delta\tilde{u}(t,x) + <\delta\mu(t),V>)u^2(t,x)  
    \\
    &\qquad +V(x)<\mu^1(t),\delta\tilde{u}(t)> + V(x)<\delta\mu(t),\tilde{u}^2(t)>.
\end{align*}

To bound $\|\partial_t \delta w\|_\infty$, we will combine the following inequalities:
\begin{itemize}
    \item for $D^2 \delta w$: \eqref{eq:proof-existence-phi-D2deltaw}
    \item for $\delta w$: \eqref{eq:proof-existence-phi-deltaw-deltatildeu}
    \item for $F^1$: it is bounded (by a constant depending only on the model's parameters and $C_2,C_3$) since we assumed $\tilde{f}$ is bounded by $C_{\tilde{f}}$,  $g$ is bounded by $C_g$,  $(t,x) \mapsto \tilde{u}^1(t,x)$ is bounded by $C_{\tilde{u}^1} \le C_2$, and $2C_{\tilde{u}^1} + C_{\tilde{f}} \le C_3$
    \item for $w^1$: its norm is bounded by a constant depending only on $C_2$
    \item for $\delta F$:
    \begin{align*}
        \|\delta F(t,x)\|
        &\le 2\|V\|_\infty \|\delta\tilde{u}\|_\infty + C(\|u^2\|_\infty + \|\tilde{u}^2\|_\infty) \sup_{t\in[0,T]}W_1(\mu^1_t,\mu^2_t),
    \end{align*}
    where $C$ depends on the Lipschitz constant of $V$ and on $C_2$.  
\end{itemize}

Combining the above bounds, we obtain:
$$
    \|\partial_t \delta w\|_\infty  \le \max\{T, \|V\|_\infty, K(T,C_g)\}  C(T+1)(\|\delta \tilde u\|_\infty + \|\delta \nabla \tilde u\|_\infty + \|\delta D^2\tilde u\|_\infty + \sup_{t\in[0,T]}W_1(\mu^1_t,\mu^2_t) )
$$
where the constant $C$ depends (in a non-increasing way) on $C_2$, $C_3$, $V$, $\|\nabla V \|_\infty$ and  $\|D^2 V \|_\infty$. Using again the local Lipschitz property of the exponential function, we deduce that a similar bound holds for  $\|\partial_t \delta u\|_\infty $.

\vskip 6pt
{\bf Part (vii). Conclusion. } 

Combining the above results, we have: 
\begin{equation}
    \label{eq:proof-existence-step2-ab}
    \|\delta u\|_\cU \le \max\{T, \|V\|_\infty, K(T,C_g)\}(T+1)C(\|\delta\tilde{u}\|_\cU  + \sup_{t\in[0,T]}W_1(\mu^1_t,\mu^2_t) ).
\end{equation}

From Lemma~\ref{lem:W1-mu-delta-nabla-u}, we have
$
    W_1(\mu^1_t,\mu^2_t) \le TC e^{C L_{\nabla u^2}}\|\delta\nabla u\|_\infty.
$
We then deduce:
\begin{align*}
    \|\delta u\|_\cU 
    &\le \max\{T, \|V\|_\infty, K(T,C_g)\} C_{\epsilon, T_0, C_{g,0}} (1+e^{T L_{\nabla u^2}})\|\delta\tilde{u}\|_\cU
\end{align*}
where the constant $C_{\epsilon, T_0, C_{g,0}}$ depends only on $\epsilon$, on $T_0$, on $C_{g,0}$, and on the model's parameters. This yields~\eqref{eq:proof-existence-step2-abc}. 
\end{proof}

\begin{lemma}
Let $\epsilon>0$. 
Let $T_0>0$ and $C_{g,0}$ be as in Proposition~\ref{prop:existence-shortT-Psi-contraction}. 
There exists $T_1<T_0$, $C_{g,1}<C_{g,0}$, and $C_{V,1}$ depending only on the model parameters (except $g$ and $V$) and on $\epsilon$ such that, if $T<T_1$, $C_g < C_{g,1}$ and $\|V\|_\infty<C_{V,1}$, then $\Phi$ is a strict contraction on  $\cU_{K(T,C_g), \Gamma_2(C_g)+\epsilon, \Gamma_3(\Gamma_2(C_g))+\epsilon}$.
\end{lemma}
\begin{proof}

We prove that there exist $T_1<T_0$, $C_{g,1}<C_{g,0}$, and $C_{V,1}$ depending only on the model parameters (except $g$ and $V$) and on $\epsilon$ such that, if $T<T_1$, $\|g\|_\infty < C_{g,1}$ and $\|V\|_\infty<C_{V,1}$, then: there exists $\Gamma<1$ such that for all $\tilde u^1,\tilde u^2 \in \cU_{\Gamma_2(C_g)+\epsilon, \Gamma_3(\Gamma_2(C_g))+\epsilon}$, 
    \[
        \|\Phi(\tilde{u}^1) - \Phi(\tilde{u}^2)\|_\cU \le \Gamma \|\tilde{u}^1 - \tilde{u}^2\|_\cU.
    \]
    Take $\tilde u^1,\tilde u^2 \in \cU_{K(T,C_g), \Gamma_2(C_g)+\epsilon, \Gamma_3(\Gamma_2(C_g))+\epsilon}$. From Proposition~\ref{prop:existence-shortT-Psi-contraction}, we know that $u^1:= \Phi(\tilde{u}^1)$ and $u^2:=\Phi(\tilde{u}^2)$ are in $\cU_{K(T,C_g), \Gamma_2(C_g)+\epsilon, \Gamma_3(\Gamma_2(C_g))+\epsilon}$.
    
    By Lemma~\ref{lem:proof-existence-deltau-Nu},  
    \begin{align*}
        \|\delta u\|_\cU 
        &\le \max\{T, \|V\|_\infty, K(T,C_g)\} C_{\epsilon,T_0}(1+e^{T L_{\nabla u^2}})\|\delta\tilde{u}\|_\cU
        \\
        &= \Gamma \|\delta\tilde{u}\|_\cU
    \end{align*}
    with $\Gamma := \max\{T, \|V\|_\infty, K(T,C_g)\} C_{\epsilon,T_0} (1+e^{T (\Gamma_2(C_g)+\epsilon)})$,  
    where we used the fact that $L_{\nabla u^2} \le \|u^2\|_\cU \le \Gamma_2(C_g)+\epsilon$. To have $\Gamma<1$, it is sufficient to choose $T_1, C_{g,1}, C_{V,1}$ such that:
    \[
        \max\{T_1, C_{V,1}, K(T_1,C_{g,1})\} < \min\left(\tfrac{1}{\Gamma_2(C_g)+\epsilon}, \tfrac{1}{C_{\epsilon,T_0}(1+e)}\right).
    \] 

\end{proof}

\subsection{Auxiliary lemma to bound $W_1$}
\begin{lemma}
\label{lem:W1-mu-delta-nabla-u}

    Let $u_1, u_2 \in \cU$. Let $\mu_i = \Psi_2(u_i), i=1,2$. Then
    $$
        W_1(\mu^1_t,\mu^2_t) \le T C_{T,L_{\nabla u^2}}\|\delta\nabla u\|_\infty,
    $$
    where $C$ is bounded when $T,L_{\nabla u^2}$ are bounded.
\end{lemma}
\begin{proof}[Proof of Lemma~\ref{lem:W1-mu-delta-nabla-u}]
{\bf First step: Bounding $W_1(\mu^1_t,\mu^2_t)$ by $\sup_{t \in [0,T]} \EE \left[ |X^1_t - X^2_t|^2 \right]$. }

Let us recall that, thanks to the Kantorovich-Rubinstein duality theorem, we have:
$$
    W_1(\mu_s, \mu_t)
    =
    \inf_{\varphi: 1-Lip} \int \varphi(x) d(\mu_s(x) - \mu_t(x)) 
    =
    \inf_{\varphi: 1-Lip, \varphi(0) = 0} \int \varphi(x) d(\mu_s(x) - \mu_t(x))
$$
where the second equality holds because $\mu_s$ and $\mu_t$ both have total mass equal to $1$ so we can shift the function by a constant without changing the integral. Let $\varphi$ be a $1$-Lipschitz function such that $\varphi(0) = 0$. Since we assume that $m_0$ has compact support, $\EE[|\varphi(X_0)|^2]^{1/2}$ is bounded by a constant $C_{m_0}$ depending only on the support of $m_0$.

Let $\varphi$ be a 1-Lipschitz function satisfying $\varphi(0)=0$. We have:
\begin{align*}
    <\mu^1_t - \mu^2_t, \varphi>
    &= \frac{\EE[\varphi(X^1_t) e^{-A^1_t}]}{\EE[e^{-A^1_t}]} - \frac{\EE[\varphi(X^2_t) e^{-A^2_t}]}{\EE[e^{-A^2_t}]}
    \\
    &= \frac{\EE[(\varphi(A^1_t) -\varphi(X^2_t)) e^{-A^1_t}]}{\EE[e^{-A^1_t}]} 
    + \frac{\EE[\varphi(X^2_t) (e^{-A^1_t} - e^{-A^2_t})]}{\EE[e^{-A^1_t}]}
    + \EE[\varphi(X^2_t) e^{-A^2_t}]\left(\frac{1}{\EE[e^{-A^1_t}]} - 
    \frac{1}{\EE[e^{-A^2_t}]}
    \right).
\end{align*}
We bound each of the three terms. Below, it is important that $C$ is a constant whose value is independent of $\varphi$ (except for the fact that $\varphi$ has Lipschitz constant $1$). Its value may depend on the Lipschitz constant of $V$. 
\begin{itemize}
    \item $\left|\frac{\EE[(\varphi(X^1_t) -\varphi(X^2_t)) e^{-A^1_t}]}{\EE[e^{-A^1_t}]}\right| \le \EE[|\varphi(X^1_t) -\varphi(X^2_t)|^2]^{1/2} \EE[e^{-2A^1_t}]^{1/2}  \frac{1}{\EE[e^{-A^1_t}]}\le C \EE[|X^1_t -X^2_t|^2]^{1/2}$ using H\"older inequality and the fact that $A^1_t \in [0,T]$. 
    \item $\left|\frac{\EE[\varphi(X^2_t) (e^{-A^1_t} - e^{-A^2_t})]}{\EE[e^{-A^1_t}]}\right| \le C \EE[|\varphi(X^2_t)| |e^{-A^1_t} - e^{-A^2_t}|] \le C\EE[|\varphi(X^2_t)|^2]^{1/2} \EE[|e^{-A^1_t} - e^{-A^2_t}|^2]^{1/2}$. We first note that $A^1_t,A^2_t \in [0,T]$ and the exponential function is locally Lipschitz so it is Lipschitz on $[0,T]$. So $|e^{-A^1_t} - e^{-A^2_t}| \le C|A^1_t - A^2_t| \le C\int_0^t |V(X^1_s) - V(X^2_s)|ds \le C \int_0^t |X^1_s - X^2_s|ds$. Moreover, $\EE[|\varphi(X^2_t)|^2]^{1/2} \le \EE[|\varphi(X_0)|^2]^{1/2} + \EE[|\varphi(X^2_t) - \varphi(X_0)|^2]^{1/2}$, and $\EE[|\varphi(X_0)|^2]^{1/2}<+\infty$ since we assume $m_0$ has compact support, while $\EE[|\varphi(X^2_t) - \varphi(X_0)|^2] \le \EE[|X^2_t - X_0|^2] \le C$ where the last constant $C$ depends on $\|u\|_\cU$. As a consequence: 
    $\left|\frac{\EE[\varphi(X^2_t) (e^{-A^1_t} - e^{-A^2_t})]}{\EE[e^{-A^1_t}]}\right| \le C \int_0^t \EE[|X^1_s - X^2_s|^2]^{1/2} ds$
    \item To bound $\EE[\varphi(X^2_t) e^{-A^2_t}]\left(\frac{1}{\EE[e^{-A^1_t}]} - 
    \frac{1}{\EE[e^{-A^2_t}]}\right)$, we first note that $\EE[|\varphi(X^2_t)| e^{-A^2_t}] \le C$ (roughly same argument as above), and moreover $\left|\frac{1}{\EE[e^{-A^1_t}]} - 
    \frac{1}{\EE[e^{-A^2_t}]} \right| = \left|\frac{1}{\EE[e^{-A^1_t}]\EE[e^{-A^2_t}]}(\EE[e^{-A^2_t} - e^{-A^1_t}])\right| \le C |\EE[e^{-A^2_t} - e^{-A^1_t}]| \le C \int_0^t \EE[|X^1_s - X^2_s|^2]^{1/2} ds$ (same argument as above).
\end{itemize}
So overall, collecting terms:
\begin{align*}
    <\mu^1_t - \mu^2_t, \varphi>
    &\le C \sup_{t \in [0,T]} \EE[|X^1_t - X^2_t|^2]^{1/2},
\end{align*}
where the constant $C$ is independent of $\varphi$ but may depend on $T$ (while remaining bounded for bounded $T$) and on the Lipschitz constant of $V$. As consequence,
\begin{align*}
    \sup_{t \in [0,T]} W_1(\mu^1_t, \mu^2_t)
    &\le C \sup_{t \in [0,T]} \EE[|X^1_t - X^2_t|^2]^{1/2}.
\end{align*}

{\bf Second step: Bounding $\sup_{t \in [0,T]} \EE \left[ |X^1_t - X^2_t|^2 \right]$ by $\|\nabla \delta u\|_\infty$. }

We have
\begin{align*}
    \EE \left[ |X^1_t - X^2_t|^2 \right]
    &\le C \int_0^t \EE \left[ |\nabla u^1_s(X^1_s) - \nabla u^2_s(X^2_s)|^2 \right]ds
    \\
    &\le C \left(\int_0^t \EE \left[ |\nabla u^1_s(X^1_s) - \nabla u^2_s(X^1_s)|^2 \right] ds + \int_0^t \EE \left[ |\nabla u^2_s(X^1_s) - \nabla u^2_s(X^2_s)|^2 \right] ds\right)
    \\
    &\le C\left( t \|\nabla \delta u\|_\infty + L_{\nabla u^2}\int_0^t \EE \left[ |X^1_s - X^2_s|^2 ds \right] \right)
\end{align*}
where $L_{\nabla u^2}$ is the supremum over $t \in [0,T]$ of the Lipschitz constant of $\nabla u^2_t(\cdot)$.

By Gr\"onwall's lemma, 
\begin{align*}
    \sup_{t \in [0,T]} \EE \left[ |X^1_t - X^2_t|^2 \right]
    \le T C e^{T L_{\nabla u^2}} \|\nabla \delta u\|_\infty,
\end{align*}
where $C$ is independent of $u^i$.

\vskip 6pt
Combining the the results of the two steps above, we obtain:
\begin{align*}
    \sup_{t \in [0,T]} W_1(\mu^1_t, \mu^2_t)
    \le T C e^{T L_{\nabla u^2}} \|\nabla \delta u\|_\infty,
\end{align*}
where $C$ is independent of $u^i, i=1,2$ and is bounded for bounded $T$.
\end{proof}

\end{document}